\documentclass[11pt]{amsart}
\usepackage{amsmath}
\usepackage{mathtools}
\usepackage{esint}
\usepackage{tikz}
\usepackage{comment}

\theoremstyle{plain}
\newtheorem{theorem}{Theorem}[section]
\newtheorem{lemma}[theorem]{Lemma}
\newtheorem{prop}[theorem]{Proposition}
\newtheorem{corollary}[theorem]{Corollary}

\theoremstyle{definition}
\newtheorem{remark}[theorem]{Remark}
\newtheorem{definition}[theorem]{Definition}
\newtheorem{assumption}[theorem]{Assumption}

\numberwithin{equation}{section}

\def\be{\begin{equation}}
\def\ee{\end{equation}}

\newcommand{\g}[0]{\textmd{g}}

\newcommand{\dif}{\mathrm{d}}

\newcommand{\s}{\mathcal S}

\begin{document}

\title[Asymptotic Expansions of Solutions near Isolated Singular Points]
{Asymptotic Expansions of Solutions\\ of the Yamabe Equation and the $\sigma_k$-Yamabe Equation\\ 
near Isolated Singular Points}
\author[Qing Han]{Qing Han}
\address{Department of Mathematics\\
University of Notre Dame\\
Notre Dame, IN 46556} \email{qhan@nd.edu}
\author[Xiaoxiao Li]{Xiaoxiao Li}
\address{Department of Mathematics\\
University of Notre Dame\\
Notre Dame, IN 46556}
\email{Xiaoxiao.Li.244@nd.edu}
\author[Yichao Li]{Yichao Li}
\address{Department of Mathematics\\
University of Notre Dame\\
Notre Dame, IN 46556} \email{yli20@nd.edu}

\begin{abstract}
We study asymptotic behaviors of positive solutions to the Yamabe equation and 
the $\sigma_k$-Yamabe equation near isolated singular points 
and establish expansions up to arbitrary orders. 
Such results generalize an earlier pioneering work by 
Caffarelli,  Gidas,  and Spruck, and a work by Korevaar, Mazzeo, Pacard, and Schoen, 
on the Yamabe equation, and a work by Han,  Li, and Teixeira on the $\sigma_k$-Yamabe equation. 
The study is based on  a combination of  classification of global singular solutions 
and an analysis of linearized operators at these global singular solutions. Such linearized equations 
are uniformly elliptic near singular points for $1\le k\le n/2$ and become degenerate 
for $n/2<k\le n$. In a significant portion of the paper, we establish a degree 1 expansion for the $\sigma_k$-Yamabe equation
for $n/2<k<n$, 
generalizing a similar result for $k=1$ by Korevaar, Mazzeo, Pacard, and Schoen and for $2\le k\le n/2$ by Han,  Li, and Teixeira. 
\end{abstract}

\maketitle

\section{Introduction}\label{sec-intro}


In a pioneering paper \cite{CaffarelliGS1989}, Caffarelli,  Gidas,  and Spruck 
studied the Yamabe equation of the form 
\begin{equation}\label{eq-Yamabe} 
-\Delta u=\frac14n(n-2)u^{\frac{n+2}{n-2}}\quad\text{in }B_1\setminus\{0\},\end{equation}
and proved that positive singular solutions of \eqref{eq-Yamabe} in $B_1\setminus\{0\}$
are asymptotic to radial singular solutions of \eqref{eq-Yamabe} in $\mathbb R^n\setminus\{0\}$. 
In \cite{KorevaarMPS1999}, Korevaar, Mazzeo, Pacard, and Schoen 
studied refined asymptotics and expanded such solutions to the next order. 
Geometrically, for any positive solution $u$ of the equation \eqref{eq-Yamabe}, 
the corresponding conformal metric 
$$g=u^{\frac4{n-2}}|dx|^2$$
has a constant  scalar curvature $R_g=n(n-1)$.

To state these results for the Yamabe equation in consistence with similar results for the $\sigma_k$-Yamabe equation, 
we 
express the equation \eqref{eq-Yamabe}
on the cylinder $\mathbb R_+\times\mathbb S^{n-1}$. 
Introduce the cylindrical coordinates $(t,\theta)\in \mathbb R_+\times \mathbb S^{n-1}$ by
\begin{equation}\label{eq-cylinder} t=-\ln|x|,\quad \theta=\frac{x}{|x|}.\end{equation}
Set
\begin{equation}\label{eq-def-U}v(t, \theta)=|x|^{\frac{n-2}2}u(x).\end{equation}
A straightforward calculation  transforms the equation 
\eqref{eq-Yamabe} to 
\begin{equation}\label{eq-U-intro}
v_{tt}+\Delta_{\theta}v-\frac14(n-2)^2v+\frac14n(n-2)v^{\frac{n+2}{n-2}}=0
\quad\text{in }\mathbb R_+\times \mathbb S^{n-1}.\end{equation}

In the following, we always consider positive solutions $u$ of \eqref{eq-Yamabe} in $B_1\setminus \{0\}$, 
with a nonremovable singularity at the origin. With $v$ given by \eqref{eq-def-U}, studying 
behaviors of $u$ as $x\to 0$ is equivalent to studying $v$ as $t\to\infty$. 
For convenience, we shall say that $v$ has a nonremovable singularity at infinity. 
A radial solution of \eqref{eq-Yamabe}, a solution in terms of $|x|$, induces a solution of \eqref{eq-U-intro} 
in terms of $t$ only. For convenience again, we refer to this solution as a radial solution of \eqref{eq-U-intro}. 

In terms of $v$, Caffarelli,  Gidas,  and Spruck \cite{CaffarelliGS1989} proved the following result
by a  ``measure theoretic" version of the moving plane technique. 
\smallskip 

\noindent
{\bf Theorem A} (\cite{CaffarelliGS1989}). 
{\it For $n\ge 3$, let $v$ be a positive solution of \eqref{eq-U-intro} in $\mathbb R_+\times\mathbb S^{n-1}$, 
with a nonremovable singularity at infinity. Then, 
there exists a
radial solution $\xi(t)$ of \eqref{eq-U-intro} such that 
\begin{equation}\label{eq-estimate-u-0}
v(t,\theta)-\xi(t)\to 0\quad\text{as }t\to\infty.\end{equation} 
Moreover, $\xi$ is a positive periodic function on $\mathbb R$. } 

\smallskip

Subsequent to \cite{CaffarelliGS1989}, there have
been many results related to the theme of Theorem A. 
In one direction, the estimate \eqref{eq-estimate-u-0} was refined to expansions of higher orders
for the Yamabe equation. 
In another direction, the estimate \eqref{eq-estimate-u-0} was established for other types of the 
equations.

Korevaar, Mazzeo, Pacard, and Schoen \cite{KorevaarMPS1999} extended the expansion \eqref{eq-estimate-u-0}
after the order $\xi$ by a combination of rescaling
analysis, classification of global singular solutions as in \cite{CaffarelliGS1989},
and analysis of linearized operators at these global singular solutions.


\smallskip 

\noindent
{\bf Theorem B} (\cite{KorevaarMPS1999}).
{\it For $n\ge 3$, let $v$ be a positive solution of \eqref{eq-U-intro} 
in $\mathbb R_+\times\mathbb S^{n-1}$, 
with a nonremovable singularity at infinity, 
and let $\xi(t)$ be the 
radial solution of \eqref{eq-U-intro} 
satisfying \eqref{eq-estimate-u-0}. Then, there exists a
spherical harmonic $Y$ of degree 1 such that,  for any $(t, \theta)\in \mathbb R_+\times\mathbb S^{n-1}$, 
\begin{equation}\label{eq-estimate-u-1}
\Big|v(t,\theta)-\xi(t)-e^{-t}\big[-\xi'(t)+\frac12(n-2)\xi(t)\big]Y(\theta)\Big|\le Ce^{-\beta t},\end{equation} 
where 
$\beta\in (1,2]$ and $C$ are positive constants. 
}

\smallskip


The study of singular solutions of the Yamabe equation is related
to the characterization of the size of the limit set of the image domain in
$\mathbb S^n$ of the developing map of a locally conformally flat $n$-manifold. 
Schoen and Yau \cite{SchoenYau1988} proved that, for a complete conformal
metric $g$ in a domain 
$\Omega\subset\mathbb S^n$ with the scalar curvature having a positive lower
bound, the Hausdorff dimension of $\partial\Omega$ has to be $\le (n - 2)/2$. 
Schoen \cite{Schoen1988} constructed complete conformal metrics on $\mathbb S^n\setminus\Lambda$
when $\Lambda$ is either a finite discrete set on $\mathbb S^n$ containing at least two points or a set arising as the
limit set of a Kleinian group action. 
Mazzeo and Pacard \cite{MazzeoP1999} constructed complete conformal metrics 
on $\mathbb S^n\setminus\Lambda$ with scalar curvature 1
if $\Lambda$ consists of a finite number of disjoint smooth closed submanifolds of
dimension $\le (n - 2)/2$. 


Theorem B can be viewed as the expansion up to order 1 for positive solutions of \eqref{eq-U-intro}, 
as $t\to\infty$. 

The first objective of this paper is to study expansions of these solutions
up to  arbitrary orders. To this end, we prove the following result. 

\begin{theorem}\label{thrm-Asymptotic-uk}
Let $v$ be a positive solution of \eqref{eq-U-intro} in $\mathbb R_+\times\mathbb S^{n-1}$,
with a nonremovable singularity at infinity,
and $\xi$ be a positive periodic solution of \eqref{eq-U-intro} satisfying \eqref{eq-estimate-u-0}.
Then, there exists a 
positive sequence $\{\mu_i\}_{i\ge 1}$, strictly increasing and divergent to $\infty$,  such that, 
for any positive integer $m$ and any $(t, \theta)\in \mathbb R_+\times\mathbb S^{n-1}$, 
\begin{equation}\label{eq-estimate-u-k}
\Big|v(t,\theta)-\xi(t)-\sum_{i=1}^m\sum_{j=0}^{i-1}c_{ij}(t, \theta)t^je^{-\mu_it}\Big|
\le Ct^me^{-\mu_{m+1} t},\end{equation}
where $C$ is a positive constant depending on $\xi$ and $m$, 
and  
$c_{ij}$ is a bounded smooth function on $\mathbb R_+\times\mathbb S^{n-1}$, for each $i, j$ in the  summation. 
Moreover, $\mu_1=1$. 
\end{theorem}

In the following, we will refer to $\xi$ and $\{\mu_i\}$ as the {\it leading term} and the {\it index set} 
in the expansion of $v$, respectively. 
As Theorem \ref{thrm-Asymptotic-uk} demonstrates, 
the index set determines the decay rate of $v(t,\theta)-\xi(t)$ in the following pattern: 
\begin{align*}
e^{-\mu_1 t},\, te^{-\mu_2 t},\, e^{-\mu_2t},\,\cdots,\, e^{-\mu_{m-1}t},\,
t^{m-1}e^{-\mu_mt},\,\cdots,\, e^{-\mu_mt},\, \cdots.\end{align*}
We will define the index set in Section \ref{sec-Yamabe}. 

We need to emphasize that the index set $\{\mu_i\}$ is determined by the leading term $\xi$. 
Solutions of  \eqref{eq-U-intro} with different leading terms have different index sets. 
This is sharply different from many other similar types of estimates, where the index sets 
are determined by underlying equations, independent of specific solutions.  
The coefficients $c_{ij}$ are determined by the leading term $\xi$, up to the kernels of some 
linear equations also determined by $\xi$. Each $c_{ij}$ is a finite sum of ``separable forms" in the following sense. 
Let $\{\lambda_i\}$ be the sequence of eigenvalues of $-\Delta_{\theta}$ on $\mathbb S^{n-1}$, 
arranged in an increasing order
with $\lambda_i\to\infty$ as $i\to\infty$, and let  $\{X_i\}$ be 
a sequence of the corresponding normalized eigenfunctions of 
$-\Delta_\theta$ on $L^2(\mathbb S^{n-1})$. 
Then,  
\begin{equation}\label{eq-coefficients} 
c_{ij}(t, \theta)=\sum_{l=0}^{M_{ij}}a_{ijl}(t)X_{l}(\theta),\end{equation}
where $M_{ij}$ is a nonnegative integer depending only on $\xi$, $n$, $i$, and $j$, 
and $a_{ijl}$ is a smooth periodic function. The period of $a_{ijl}$ is the same  as that of $\xi$ if $\xi$ is a nonconstant 
periodic solution, and is $2\pi/\sqrt{n-2}$ if $\xi$ is a constant solution. 
In the proof of Theorem \ref{thrm-Asymptotic-uk}, we will 
construct the summation part in \eqref{eq-estimate-u-k}, or $c_{ij}$ in \eqref{eq-coefficients} specifically, 
in a rather mechanical way. 
It has two sources, the kernel of the linearized equation and the nonlinearity. 

In some special cases, powers  of $t$ are absent in the summation, and 
\eqref{eq-estimate-u-k} has the form 
\begin{equation}\label{eq-estimate-u-k-special}\Big|
v(t,\theta)-\xi(t)-\sum_{i=1}^mc_{i}(t, \theta)e^{-\mu_it}\Big|
\le Ce^{-\mu_{m+1} t},\end{equation}
where 
$c_{i}$ has the form as in \eqref{eq-coefficients}.

\smallskip

There have been many results related to the theme of estimates 
\eqref{eq-estimate-u-0} and \eqref{eq-estimate-u-1} for other equations. 
Han, Li, and Teixeira \cite{HanLi2010} studied 
the $\sigma_k$-Yamabe equation near isolated singularities and derived similar estimates for 
its solutions. 
Caffarelli, Jin, Sire, and Xiong \cite{CaffarelliJSX2014} studied 
fractional semi-linear elliptic
equations with isolated singularities. We now turn our attention to the  $\sigma_k$-Yamabe equation, 
a family of conformally invariant equations which
include \eqref{eq-Yamabe}. 

Let $g$ be a metric in the punctured ball $B_1\backslash \{0\}\subset\mathbb{R}^n$. 
The Weyl-Schouten tensor $A_{g}$ of $g$ is given by
\begin{equation*}
A_{g}=\frac{1}{n-2}\Big\{Ric_{g}-\frac{{R}_{g}}{2(n-1)}g\Big\}, 
\end{equation*}
where $Ric_g$ and $R_g$ denote the Ricci and scalar curvature of $g$, respectively. 
Denote by $\sigma_k(\g^{-1}\circ A_{g})$ the $k$-th elementary symmetric function 
 of the eigenvalues of $A_g$ with respect to $g$. 
We consider the equation 
\begin{equation}\label{eq-k-Yamabe-g}
\sigma_k(g^{-1}\circ A_g)=c_k\quad\text{in }B_1\setminus\{0\},
\end{equation}
for some positive constant $c_k$. 
In the following, we always choose 
\begin{equation}\label{eq-choice-c}c_k=2^{-k}\Big(\,\begin{matrix}n\\k\end{matrix}\,\Big).\end{equation}
For $k=1$,  \eqref{eq-k-Yamabe-g} reduces to 
$R_g=n(n-1).$

We assume that 
$g$ is conformal to the Euclidean metric; namely, for some positive smooth function $u$ in $B_1\backslash \{0\}$, 
\begin{equation*}
g=u^{\frac{4}{n-2}}(x)|dx|^2.
\end{equation*}
In terms of $u$, the equation \eqref{eq-k-Yamabe-g} has the form
\begin{equation}\label{eq-k-Yamabe-u}
\sigma_k\left(-(n-2)u\nabla^2u+n\nabla u\otimes\nabla u-|\nabla u|^2\mathrm{Id}\right)
={2^{-k}}{(n-2)^{2k}c_k}u^{\frac{2kn}{n-2}}.
\end{equation}
We always assume that $u$ has a nonremovable singularity at $x=0$. 

In the cylindrical coordinates $(t, \theta)$ introduced in \eqref{eq-cylinder}, we write 
\begin{equation*}g=e^{-2w(t,\theta)}(dt^2+d\theta^2).\end{equation*} 
Then, 
\begin{equation}\label{eq-change-u-w}|x|^{\frac{n-2}2}u(x)=e^{-\frac{n-2}{2}w(t,\theta)}.\end{equation}
For convenience, we write $g_0=dt^2+d\theta^2$. Then,  the equation \eqref{eq-k-Yamabe-u} reduces to 
\begin{equation}\label{eq-k-Yamabe-w}
\sigma_k\Big(g_0^{-1}\circ\Big\{ A_{g_0}+\nabla^2w+\nabla w\otimes \nabla w
-\frac12|\nabla w|^2g_0\Big\}\Big)=c_ke^{-2kw}.
\end{equation}
We point out that $w$ in \eqref{eq-change-u-w} is different from $v$ in \eqref{eq-def-U}.


We say a solution to \eqref{eq-k-Yamabe-u} or \eqref{eq-k-Yamabe-w} is in the 
$\Gamma_k^+$ class if its associated Weyl-Schouten tensor is in $\Gamma_k^+$. 
For a positive solution $u$ of $\Gamma_k^+$, this means that the matrix 
$\big(-(n-2)u\nabla^2u+n\nabla u\otimes\nabla u-|\nabla u|^2\mathrm{Id}\big)$
belongs to $\Gamma_k^+$. 

Based on earlier results by Li \cite{Li2006}, 
Han,  Li, and Teixeira \cite{HanLi2010} proved the following results for 
the $\sigma_k$-Yamabe equation \eqref{eq-k-Yamabe-w}, 
similar to Theorem A and Theorem B for the Yamabe equation \eqref{eq-U-intro}. 


\smallskip 

\noindent
{\bf Theorem C} (\cite{HanLi2010}). 
{\it For $n\ge 3$ and $2\le k\le n$, 
let $w(t,\theta)$ be a smooth solution of \eqref{eq-k-Yamabe-w} on $\mathbb R_+\times \mathbb S^{n-1}$ 
in the $\Gamma_k^+$ class, with a nonremovable singularity at infinity. Then, there exists a
radial solution $\xi(t)$ of \eqref{eq-k-Yamabe-w} on $\mathbb R\times \mathbb S^{n-1}$ 
in the $\Gamma_k^+$ class such that, for any $t>1$, 
\begin{equation}\label{eq-approximation-0}
|w(t,\theta)-\xi(t)|\le Ce^{-\alpha t},\end{equation} 
where $\alpha$ and $C$ are  positive constants. 
}

\smallskip 


If $k\le n/2$, the constant $\alpha$ in \eqref{eq-approximation-0} can be chosen to be 1, and
the estimate \eqref{eq-approximation-0} can be improved. 

\smallskip 

\noindent
{\bf Theorem D} (\cite{HanLi2010}).
{\it For $n\ge 3$ and $2\le k\le n/2$, 
let $w(t,\theta)$ be a smooth solution of \eqref{eq-k-Yamabe-w} on $\mathbb R_+\times \mathbb S^{n-1}$ 
in the $\Gamma_k^+$ class, with a nonremovable singularity at infinity, and let $\xi(t)$ be the 
radial solution of \eqref{eq-k-Yamabe-w} on $\mathbb R\times \mathbb S^{n-1}$ 
in the $\Gamma_k^+$ class for which \eqref{eq-approximation-0} holds. Then, there exists a
spherical harmonic $Y$ of degree 1 such that,  for any $t>1$, 
\begin{equation}\label{eq-approximation-1-small-k}
|w(t,\theta)-\xi(t)-e^{-t}\big(1+\xi'(t)\big)Y(\theta)|\le Ce^{-\beta t},\end{equation} 
where $\beta\in (1,2]$ and $C$ are positive constants. 
}

\smallskip 



Refer to \cite{HanLi2010} and references there for more information concerning the $\sigma_k$-Yamabe equation, 
and in particular to \cite{Chang-Hang-Yang2004}, \cite{Gonzalez2005}, 
and \cite{Guan-Lin-Wang2005} for the size of the singular sets, and to \cite{Mazzieri-Ndiaye}  and \cite{Mazzieri-Segatti2012}
for the existence of solutions with isolated singularity.  

In the second part of this paper, we study further asymptotic expansions of solutions $w$ of \eqref{eq-k-Yamabe-w}. 
Our first task is to investigate whether Theorem D holds for $n/2<k\le n$. 

For the case $k>n/2$,  
Gursky and Viacolvsky \cite{Gursky-Viacolvsky2006} and Li \cite{Li2006} proved that 
$\alpha$ in \eqref{eq-approximation-0} can be chosen as $\alpha=2-n/k$. This is the starting point of our study. 
We will prove that \eqref{eq-approximation-1-small-k}
indeed holds for solutions of \eqref{eq-k-Yamabe-w} for the case $n/2<k<n$.

\begin{theorem}\label{thrm-approximation-order-1}
For $n\ge 3$ and $n/2<k< n$, 
let $w(t,\theta)$ be a smooth solution of \eqref{eq-k-Yamabe-w} on $\mathbb R_+\times \mathbb S^{n-1}$ 
in the $\Gamma_k^+$ class, with a nonremovable singularity at infinity,  and let $\xi(t)$ be the 
radial solution of \eqref{eq-k-Yamabe-w} on $\mathbb R\times \mathbb S^{n-1}$ 
in the $\Gamma_k^+$ class for which \eqref{eq-approximation-0} holds. Then, there exists a
spherical harmonic $Y$ of degree 1 such that,  for any $t>1$, 
\begin{equation}\label{eq-approximation-1-large-k}
|w(t,\theta)-\xi(t)-e^{-t}\big(1+\xi'(t)\big)Y(\theta)|\le Ce^{-\beta t},\end{equation} 
where $\beta\in (1,2)$ and $C>0$ are positive constants. 
\end{theorem}

A significant portion of the paper will be devoted to the proof of Theorem \ref{thrm-approximation-order-1}. 
We now compare \eqref{eq-approximation-1-small-k} for $2\le k\le n/2$ and \eqref{eq-approximation-1-large-k} for 
$n/2<k<n$. Although \eqref{eq-approximation-1-small-k} and \eqref{eq-approximation-1-large-k} 
have the same form, there are significant differences  caused by different behaviors of radial solutions $\xi$. 
Refer to \cite{HanLi2010} for  properties of radial solutions.
For $2\le k< n/2$, radial solutions $\xi$ with a nonremovable singularity at infinity are always bounded.  However, 
for $n/2<k<n$, radial solutions $\xi$ are unbounded for $t>0$ and grow at the rate of $t$ as $t\to \infty$. 
In fact, $\xi$ has an asymptotic expansion in the following form 
\begin{equation}\label{eq-expansion-xi}
\xi(t)=t+a_0+a_1e^{-(2-\frac{n}{k})t}+\cdots.\end{equation}
See Lemma \ref{lemma-expansion-xi-t} for a complete description of the expansion for radial solutions. 
By substituting \eqref{eq-expansion-xi} in \eqref{eq-approximation-1-large-k}, we note 
that in the expansion of $w$ there are finitely many terms between  
$t$ and $e^{-t}$, decaying 
exponentially at orders given by integer multiples of $2-{n}/{k}$, up to 1. There is a similar pattern for $k=n/2$. 

Theorem D and Theorem \ref{thrm-approximation-order-1} are viewed as the expansion up to order 1 
for the $\sigma_k$-Yamabe equation. 
Based on these results, we can establish the expansion up to arbitrary order. 

\begin{theorem}\label{thrm-approximation-order-arbitrary}
For $n\ge 3$ and $2\le k< n$, 
let $w(t,\theta)$ be a smooth solution of \eqref{eq-k-Yamabe-w} on $\mathbb R_+\times \mathbb S^{n-1}$ 
in the $\Gamma_k^+$ class,  with a nonremovable singularity at infinity, and let $\xi(t)$ be the 
radial solution of \eqref{eq-k-Yamabe-w} on $\mathbb R\times \mathbb S^{n-1}$ 
in the $\Gamma_k^+$ class for which \eqref{eq-approximation-0} holds. 
Then, there exists a 
positive sequence $\{\mu_i\}_{i\ge 1}$, strictly increasing and divergent to $\infty$,  such that, 
for any positive integer $m$ and any $(t, \theta)\in \mathbb R_+\times\mathbb S^{n-1}$, 
\begin{equation}\label{eq-expansion-sigma-k-arbitrary}
\Big|w(t,\theta)-\xi(t)-\sum_{i=1}^m\sum_{j=0}^{i-1}c_{ij}(t, \theta)t^je^{-\mu_it}\Big|
\le Ct^me^{-\mu_{m+1} t},\end{equation}
where $C$ is a positive constant depending on $\xi$, $n$, $k$, and $m$, 
and 
$c_{ij}$ is a bounded smooth function in $\mathbb R_+\times\mathbb S^{n-1}$, 
for each $i, j$ in the  summation. 
Moreover, $\mu_1=1$. 
\end{theorem}

We will define the index set $\{\mu_i\}_{i\ge 1}$ in Section \ref{sec-proof-main}, 
and demonstrate that it is determined by the radial solution $\xi$ for 
$2\le k\le n/2$ and is a fixed sequence for $n/2< k<n$. In fact, we can express the index set 
explicitly in terms of only $n$ and $k$ in the latter case. 
Moreover, each $c_{ij}(t, \theta)$  has the form as in \eqref{eq-coefficients}, 
where $a_{ijl}$ is a smooth periodic function, with the same period as $\xi$, 
for $2\le k<n/2$, and is constant for $n/2\le k<n$. 

We now briefly discuss the arrangement of this paper. 
In Section \ref{sec-Yamabe}, we prove Theorem \ref{thrm-Asymptotic-uk} for the Yamabe equation. 
In Section \ref{sec-identities}, 
we rewrite the $\sigma_k$-Yamabe equation \eqref{eq-k-Yamabe-w} 
as a linear form equal to a nonlinear form. 
The precise expression of the nonlinear form is needed later on. 
In Section \ref{sec-Radial-Solutions}, 
we discuss radial solutions and establish their asymptotic behaviors near infinity. 
In Section \ref{sec-Linearized-Equations}, 
we study the kernels of the linearized equations. 
In Section \ref{sec-proof-main}, we study asymptotic behaviors of solutions of the 
$\sigma_k$-Yamabe equation and prove Theorem \ref{thrm-approximation-order-1}
and Theorem \ref{thrm-approximation-order-arbitrary}. 
There are two appendices. 
In Appendix \ref{sec-Linear-Equations}, 
we discuss asymptotic expansions for solutions of linear equations, both ODEs and PDEs. 
In Appendix \ref{sec-appendix1}, we establish several technical identities concerning 
spherical harmonics. 

\smallskip 
{\bf Acknowledgement:} The first author acknowledges the support by the NSF
grant DMS-1404596, 
and the second and third authors acknowledge the support by the NSF grant DMS-1569162.
The authors would like to thank Zhengchao Han and Yanyan Li for helpful discussions.

\section{Asymptotic Behaviors for the Yamabe Equation}\label{sec-Yamabe}

In this section, we study solutions of the Yamabe equation and prove Theorem \ref{thrm-Asymptotic-uk}. 
We choose to present and discuss the main expansion result for the Yamabe equation firstly and separately, mainly 
due to the simple structure of the Yamaba equation, although it is a special case of the more general 
$\sigma_k$-Yamabe equation.

Following \cite{KorevaarMPS1999}, we study the Yamabe equation  in cylindrical coordinates as in \eqref{eq-U-intro}, 
which we record here as follows: 
\begin{equation}\label{eq-U}
v_{tt}+\Delta_{\theta}v-\frac14(n-2)^2v+\frac14n(n-2)v^{\frac{n+2}{n-2}}=0.\end{equation}
For radial solutions $\xi(t)$, \eqref{eq-U} reduces to 
\begin{equation}\label{eq-psi}
\xi''-\frac14(n-2)^2\xi+\frac14n(n-2)\xi^{\frac{n+2}{n-2}}=0\quad\text{on }\mathbb R.\end{equation}

In the following, we fix a positive solution $v$ of \eqref{eq-U}  on $\mathbb R_+\times \mathbb S^{n-1}$, 
with nonremovable singularity at infinity, 
and fix a positive periodic solution 
$\xi$ of \eqref{eq-psi}. Set 
\begin{equation}\label{eq-Asymptotic-Uk-1}\varphi(t,\theta)=v(t,\theta)-\xi(t).\end{equation}
A straightforward calculation, with the help of \eqref{eq-U} and \eqref{eq-psi},  yields 
\begin{equation}\label{eq-Asymptotic-Uk-2}\mathcal L\varphi=\mathcal F(\varphi),\end{equation}
where 
\begin{equation}\label{eq-U2-01}
\mathcal L \varphi=
\varphi_{tt}+\Delta_{\theta}\varphi-\frac14(n-2)^2\varphi+\frac14n(n+2)\xi^{\frac{4}{n-2}}\varphi,\end{equation}
and 
\begin{equation}\label{eq-Asymptotic-Uk-3}\mathcal F(\varphi)=-\frac14n(n-2)
\Big[\big(\xi+\varphi\big)^{\frac{n+2}{n-2}}-\xi^{\frac{n+2}{n-2}}
-\frac{n+2}{n-2}\xi^{\frac{4}{n-2}}\varphi\Big].\end{equation}
The linear operator $\mathcal L$ in \eqref{eq-U2-01} is simply the linearized operator of \eqref{eq-U} at $\xi$.  
This operator has periodic coefficients and  hence may be studied by
classical Floquet theoretic methods. Our main interest is its kernel. 

We now project the operator $\mathcal L$ to spherical harmonics. 
Let $\{\lambda_i\}$ be the sequence of eigenvalues of $-\Delta_{\theta}$ on $\mathbb S^{n-1}$, 
arranged in an increasing order
with $\lambda_i\to\infty$ as $i\to\infty$, and let  $\{X_i\}$ be 
a sequence of the corresponding normalized eigenfunctions of 
$-\Delta_\theta$ on $L^2(\mathbb S^{n-1})$; namely, $-\Delta_{\theta}X_i=\lambda_iX_i$ for each $i\ge 0$. 
Note that 
$$\lambda_0=0,\quad\lambda_1=\cdots=\lambda_n=1,\quad\lambda_{n+1}=2n,\quad\cdots,$$ 
and that each $X_i$ is a spherical harmonic of certain degree. In the following, 
we fix such a sequence $\{X_i\}$, which forms an orthonormal basis in $L^2(\mathbb S^{n-1})$.  
Refer to Appendix \ref{sec-Linear-Equations} for details.

For a fixed $i\ge 0$ and any $\psi=\psi(t)\in C^2(\mathbb R_+)$, we write 
\begin{equation}\label{eq-U2-01a}\mathcal L(\psi X_i)=(L_i\psi)X_i.\end{equation} 
By $-\Delta_\theta X_i=\lambda_iX_i$, we have 
\begin{equation}\label{eq-U2-01b}
L_i\psi=\psi_{tt}+\Big(\frac14n(n+2)\xi^{\frac{4}{n-2}}-\frac14(n-2)^2-\lambda_i\Big)\psi.\end{equation}
We now analyze the kernel of $L_i$ for each $i=0, 1, \cdots$. 

We first recall the classification of solutions of \eqref{eq-psi}.  
Let $\xi$ be a positive solution of \eqref{eq-psi}, with nonremovable singularity at infinity.  
Then, $\xi$ is either a positive constant or a nonconstant periodic smooth function. 
The constant solution can be found easily from \eqref{eq-psi} and is unique.  
(Refer to \cite{CaffarelliGS1989} for details.)

\begin{lemma}\label{lemma-Asymptotics-U1a}
Let $\xi$ be the positive constant solution of \eqref{eq-psi}. 

$\mathrm{(i)}$  For $i=0$, $\mathrm{Ker}(L_0)$ has a basis $\cos(\sqrt{n-2}t)$ and $\sin(\sqrt{n-2}t)$. 

$\mathrm{(ii)}$ There exists an increasing sequence of positive constants $\{\rho_i\}_{i\ge 1}$, 
divergent to $\infty$, such that  
for any $i\ge 1$, 
$\mathrm{Ker}(L_i)$ has a basis  
$e^{-\rho_i t}$ and $e^{\rho_i t}$. Moreover, $\rho_1=\cdots=\rho_n=1$. 
\end{lemma}

\begin{lemma}\label{lemma-Asymptotics-U2a}
Let $\xi$ be a positive nonconstant periodic solution of \eqref{eq-psi}. 

$\mathrm{(i)}$  For $i=0$, $\mathrm{Ker}(L_0)$ has a basis  
$p_0^+$ and $atp_0^++p_0^-$, for some smooth periodic functions $p_0^+$ and $p_0^-$ on $\mathbb R$,  
and some constant $a$. 

$\mathrm{(ii)}$ There exists an increasing sequence of positive constants $\{\rho_i\}_{i\ge 1}$, 
divergent to $\infty$, such that  
for any $i\ge 1$, 
$\mathrm{Ker}(L_i)$ has a basis 
$e^{-\rho_i t}p_i^+$ and $e^{\rho_i t}p_i^-$, for some smooth periodic functions 
$p_i^+$ and $p_i^-$  on $\mathbb R$. Moreover, $\rho_1=\cdots=\rho_n=1$. 

In addition, all periodic functions in $\mathrm{(i)}$ and $\mathrm{(ii)}$ have the same period as $\xi$.
\end{lemma}

Refer to \cite{KorevaarMPS1999}, \cite{MazzeoP1999},  and \cite{MazzeoDU1996} for details, or 
to \cite{HanLi2010} for a more general setting. 

We now make an important remark concerning the sequence $\{L_i\}$.

\begin{remark}\label{remark-application-linear-L-Yamabe} 
Let $\{L_i\}_{i\ge 0}$ be given by \eqref{eq-U2-01b}. 
By Lemma \ref{lemma-Asymptotics-U1a} and Lemma \ref{lemma-Asymptotics-U2a}, the sequence $\{L_i\}_{i\ge 0}$
satisfies Assumption \ref{assum-kernel-linear-equation} in Appendix \ref{sec-Linear-Equations}. As a consequence, 
Lemma \ref{lemma-nonhomogeneous-linearized-eq} is applicable to the operator $\mathcal L$
given by \eqref{eq-U2-01}.
\end{remark} 

To proceed, we describe our strategy of proving Theorem \ref{thrm-Asymptotic-uk}. 
Let $v$ be a positive solution of \eqref{eq-U}  on $\mathbb R_+\times \mathbb S^{n-1}$, 
with nonremovable singularity at infinity, 
and $\xi$ be a positive periodic solution of \eqref{eq-psi}. The function $\varphi$ in 
\eqref{eq-Asymptotic-Uk-1} satisfies \eqref{eq-Asymptotic-Uk-2}, i.e., $\mathcal L\varphi=\mathcal R(\varphi)$. 
According to Remark \ref{remark-application-linear-L-Yamabe}, we can apply Lemma 
\ref{lemma-nonhomogeneous-linearized-eq} to the linear operator $\mathcal L$. 
Since $\mathcal R(\varphi)$ is nonlinear in $\varphi$, we will apply Lemma 
\ref{lemma-nonhomogeneous-linearized-eq} successively. In each step, we aim to 
get a decay estimate of $\mathcal R(\varphi)$, with a decay rate better than that of $\varphi$. 
Then, we can subtract expressions with the lower decay rates generated by Lemma 
\ref{lemma-nonhomogeneous-linearized-eq} to improve the decay rate of $\varphi$.
To carry out this process, we need to make two preparations. 

As the first preparation, we introduce the index set. 
Let $\xi$ be a positive periodic solution of \eqref{eq-psi} 
and $\{\rho_i\}_{i\ge 1}$ be the sequence of positive constants as in 
Lemma \ref{lemma-Asymptotics-U1a} and Lemma \ref{lemma-Asymptotics-U2a}. 
We note that $\{\rho_i\}_{i\ge 1}$ is increasing 
and divergent to infinity, with $\rho_1=\cdots=\rho_n=1$.  
We denote by $\mathbb Z_+$ the collection of nonnegative integers.

Define the {\it index set} $\mathcal I$ by  
\begin{align}\label{eq-def-index}
\mathcal I=\Big\{\sum_{i\ge 1} m_i\rho_i;\, m_i\in \mathbb Z_+\text{ with finitely many }m_i>0\Big\}.
\end{align}
In other words, $\mathcal I$ is the collection of linear combinations of finitely many $\rho_1, \rho_2, \cdots$ 
with positive integer coefficients. 
It is possible that some $\rho_i$ can be written 
as a linear combination of some of $\rho_1, \cdots, \rho_{i-1}$ with positive integer coefficients, 
whose sum is at least two. 

We now explain the construction of the index set $\mathcal I$ 
by examining the equation $\mathcal L\varphi=\mathcal R(\varphi)$ in 
\eqref{eq-Asymptotic-Uk-2}. 
First for each $i\ge 1$, the exponential decay solution 
of $L_i\psi=0$ on $\mathbb R_+$ contributes 
a decay order $\rho_i$ for $\varphi$. Next, the nonlinear expression $\mathcal R(\varphi)$ 
of $\varphi$ adds 
linear combinations of finitely many  of $\{\rho_i\}_{i\ge 1}$ with positive integer coefficients 
to the collection of decay orders. 
This is the index set $\mathcal I$ defined in \eqref{eq-def-index}. 

For the second preparation, we note that the nonlinear term $\mathcal R(\varphi)$ will produce 
products of spherical harmonics. 
We  need the following result. 

\begin{lemma}\label{lemma-SphericalHarmonics} Let $Y_k$ and $Y_l$ be spherical harmonics of 
degree $k$ and $l$, respectively. Then, 
$$Y_kY_l=\sum_{i=0}^{k+l}Z_i,$$
where $Z_i$ is some spherical harmonic of degree $i$, for $i=0,1, \cdots, k+l$. 
\end{lemma}

\begin{proof} We adopt polar coordinates $(r, \theta)$ in $\mathbb R^n$. Then, $u_k(x)=r^kY_k(\theta)$ and 
$u_l(x)=r^lY_l(\theta)$ are homogeneous harmonic polynomials of degrees $k$ and $l$, respectively. Hence, 
$u_ku_l$ is a homogeneous polynomial of degree $k+l$. By a well-known decomposition result for 
homogeneous polynomials (\cite{Stein1971}), we have 
$$u_k(x)u_l(x)=v_{k+l}(x)+|x|^2v_{k+l-2}(x)+\cdots+|x|^{k+l-\tau}v_{\tau}(x),$$
where $\tau=1$ if $k+l$ is odd and $\tau=0$ if $k+l$ is even, and $v_i$ is a homogeneous harmonic polynomial 
of degree $i$, for $i=k+l, k+l-2, \cdots, \tau$. We obtain the desired result by restricting  
the above identity to the unit sphere. 
\end{proof} 


\smallskip

We are ready to prove Theorem \ref{thrm-Asymptotic-uk}. 

\begin{proof}[Proof of Theorem \ref{thrm-Asymptotic-uk}] Throughout the proof, we adopt the following notation: 
$f=O(h)$ if $|f|\le Ch$, for some positive constant $C$. 
All estimates in the following hold for 
any $t>1$ and any $\theta\in \mathbb S^{n-1}$. 

Let $\{X_i\}$ be an orthonormal basis of $L^2(\mathbb S^{n-1})$, 
formed by eigenfunctions of $-\Delta_\theta$, and $\{\lambda_i\}$ be the sequence of corresponding 
eigenvalues, arranged in an increasing order. Then, each $X_i$ is a spherical harmonic, and 
deg$(X_i)\le $ deg$(X_j)$ for any $i\le j$. In particular, $X_0$ is a constant, and 
$X_1, \cdots, X_n$ are spherical harmonics of degree 1.

Let 
$\mathcal L$ be the linearization at $\xi$ given by \eqref{eq-U2-01}, and $L_i$ be the projection of 
$\mathcal L$ given by \eqref{eq-U2-01b}. In the following, we assume $\xi$ is a nonconstant 
periodic solution. The proof below can be modified easily for the case that $\xi$ is a constant. 

According to 
Lemma \ref{lemma-Asymptotics-U2a}, 
there is an exponentially decaying solution in $\mathrm{Ker}(L_i)$ for $i\ge 1$, i.e., 
\begin{equation}\label{eq-decaying-solution-i}\psi_i^+(t)=e^{-\rho_i t}p_i^+(t).\end{equation} 
Let $\mathcal I$ be the index set defined in \eqref{eq-def-index}.


Set $\varphi$ as in \eqref{eq-Asymptotic-Uk-1}, i.e., 
\begin{equation*}
\varphi(t,\theta)=v(t,\theta)-\xi(t).\end{equation*}
By \eqref{eq-estimate-u-1}, 
we have 
\begin{equation}\label{eq-Asymptotic-Uk-0}
\varphi(t,\theta)=O(e^{-t}).\end{equation}
This is our starting estimate. Note that $\varphi$ satisfies \eqref{eq-Asymptotic-Uk-2}, 
with $\mathcal L$ and $\mathcal R(\varphi)$ given by 
\eqref{eq-U2-01} and \eqref{eq-Asymptotic-Uk-3}, respectively. 
In particular, we have  
\begin{equation}\label{eq-Asymptotic-Uk-03a}|\mathcal L\varphi|=|\mathcal F(\varphi)|\le C\varphi^2.\end{equation}
If $|\varphi|<\xi$, then 
$$\mathcal F(\varphi)=\sum_{i=2}^\infty c_i\varphi^i,$$
where $c_i$ is a smooth periodic function on $\mathbb R$, for each $i\ge 2$. We point out that we write 
the infinite sum just for convenience. We do not need the convergence of 
the infinite series and we always expand up to finite orders.

We now decompose the index set $\mathcal I$. Set 
$$\mathcal I_\rho=\{\rho_j:\, j\ge 1\},$$ 
and 
$$\mathcal I_{\widetilde \rho}=\Big\{\sum_{i=1}^kn_i\rho_i:\, n_i\in \mathbb Z_+, \sum_{i=1}^kn_i\ge 2\Big\}.$$
We assume $\mathcal I_{\widetilde \rho}$ is given by a strictly increasing sequence 
$\{\widetilde\rho_i\}_{i\ge 1}$, with 
$\widetilde \rho_1=2$.

We first consider the case that 
\begin{equation}\label{eq-special}\mathcal I_\rho\cap \mathcal I_{\widetilde \rho}=\emptyset.\end{equation}
In other words, no $\rho_i$ can be written 
as a linear combination of some of $\rho_1, \cdots, \rho_{i-1}$ with positive integer coefficients, except a single 
$\rho_{i'}$ which is equal to $\rho_i$. 
In this case, we arrange $\mathcal I$ as follows: 
\begin{equation}\label{eq-arrangement}
1=\rho_1\le\cdots\le\rho_{k_1}<\widetilde \rho_1<\cdots<\widetilde\rho_{l_1}<
\rho_{k_1+1}\le\cdots\le\rho_{k_2}<\widetilde \rho_{l_1+1}<\cdots.\end{equation}
For each $\widetilde \rho_i$, by the definition of $\mathcal I_{\widetilde \rho}$, 
we consider 
nonnegative integers $n_1, \cdots, n_{k_1}$ such that 
\begin{equation}\label{eq-requirement-m}
n_1+\cdots+n_{k_1}\ge 2, \quad n_1\rho_1+\cdots+n_{k_1}\rho_{k_1}= \widetilde \rho_i.\end{equation}
There are only finitely many collections of 
nonnegative integers $n_1$, $\cdots$, $n_{k_1}$ satisfying \eqref{eq-requirement-m}. 
Set 
\begin{align*}
\widetilde K_i&=\max\{n_1+2n_2+\cdots+k_1n_{k_1}:\\
&\qquad \qquad n_1, \cdots, n_{k_1}\text{ are nonnegative integers 
satisfying \eqref{eq-requirement-m}}\},
\end{align*}
and 
\begin{align}\label{eq-def-M}\widetilde M_i=\max\{m:\, \mathrm{deg}(X_m)\le \widetilde K_i\}.\end{align}

By \eqref{eq-Asymptotic-Uk-0}, we have 
\begin{equation}\label{eq-Asymptotic-Uk-02}\varphi=O(e^{-\rho_1 t}),\end{equation}
and then, by \eqref{eq-Asymptotic-Uk-03a},  
\begin{equation}\label{eq-Asymptotic-Uk-03}\mathcal L\varphi=O(e^{-2\rho_1t})
=O(e^{-\widetilde \rho_1 t}).\end{equation}
We divide the proof for the case \eqref{eq-special} into several steps.

{\it Step 1.} Note $\rho_{k_1}<\widetilde \rho_1=2\rho_1$. We claim that there exists an $\eta_1$ such that 
$$\varphi=\eta_1+O(e^{-\widetilde \rho_1 t}).$$
In fact, by \eqref{eq-Asymptotic-Uk-03}
and  Lemma \ref{lemma-nonhomogeneous-linearized-eq}(ii), we can take 
\begin{equation}\label{eq-Asymptotic-Uk-11}\eta_1(t,\theta)=\sum_{i=1}^{k_1}c_i(t)X_i(\theta)e^{-\rho_it},\end{equation}
where $c_i$ is a smooth periodic function. In the present case, the function $\psi_i^+$ 
in Lemma \ref{lemma-nonhomogeneous-linearized-eq} is given by  \eqref{eq-decaying-solution-i}. 
Set 
\begin{equation}\label{eq-Asymptotic-Uk-12}
\varphi_1=\varphi-\eta_1.\end{equation}
Then, $\mathcal L\eta_1=0$, $\mathcal L\varphi_1=\mathcal F(\varphi)$,  and 
\begin{equation}\label{eq-Asymptotic-Uk-13}\varphi_1= O(e^{-\widetilde \rho_1 t}).\end{equation}
Note that \eqref{eq-Asymptotic-Uk-13} improves \eqref{eq-Asymptotic-Uk-02}. 

{\it Step 2.} We claim there exists an $\widetilde \eta_1$ such that, with  
\begin{equation}\label{eq-Asymptotic-Uk-21}
\widetilde \varphi_1=\varphi_1-\widetilde \eta_1=\varphi-\eta_1-\widetilde \eta_1,\end{equation}
we have 
\begin{equation}\label{eq-Asymptotic-Uk-23}\mathcal L\widetilde \varphi_1= O(e^{-\widetilde \rho_{l_1+1}t}).\end{equation}
We will prove that $\widetilde \eta_1$ has the form 
\begin{equation}\label{eq-Asymptotic-Uk-22}
\widetilde\eta_{1}(t,\theta)=\sum_{i=1}^{l_1}
\Big\{\sum_{m=0}^{\widetilde M_i} 
c_{im}(t)X_{m}(\theta)\Big\}e^{-\widetilde \rho_it},
\end{equation}
where $\widetilde M_i$ is defined in \eqref{eq-def-M}, 
and $c_{im}$ is a smooth periodic function. 
Note that \eqref{eq-Asymptotic-Uk-23} improves \eqref{eq-Asymptotic-Uk-03}. 

To prove this, we take 
some function $\widetilde\eta_1$ 
to be determined, 
and then set $\widetilde \varphi_1$ by \eqref{eq-Asymptotic-Uk-21}. 
Then, 
\begin{equation}\label{eq-Asymptotic-Uk-21a}
\mathcal L\widetilde \varphi_1=\mathcal F(\varphi)-\mathcal L\widetilde\eta_1.\end{equation}
Note $3\rho_1\in \mathcal I_{\widetilde \rho}$. We discuss this step in several cases.

{\it Case 1. We assume $\rho_{k_1+1}<3\rho_1$.} Then, $\widetilde \rho_{l_1}<\rho_{k_1+1}<3\rho_1$
and $\widetilde \rho_{l_1+1}\le 3\rho_1$. We now analyze $\mathcal F(\varphi)$ in \eqref{eq-Asymptotic-Uk-21a}. 
Note 
$$\mathcal F(\varphi)=\mathcal F(\varphi_1+\eta_1)=\sum_{i=2}^\infty c_i(\varphi_1+\eta_1)^i.$$
It is worth mentioning again that we write the infinite sum just for convenience
and we always expand up to finite orders. 
For terms involving $\varphi_1$, we have, by \eqref{eq-Asymptotic-Uk-13}, 
$$\varphi_1^2\le Ce^{-4\rho_1t},\quad |\varphi_1\eta_1|\le Ce^{-3\rho_1t}.$$
Note that $\eta_1$ is given by \eqref{eq-Asymptotic-Uk-11}. 
We write 
$$\sum_{i=2}^\infty c_i\eta_1^i=\sum_{n_1+\cdots+n_{k_1}\ge 2}a_{n_1\cdots n_{k_1}}(t)
e^{-(n_1\rho_1+\cdots+n_{k_1}\rho_{k_1})t}X_1^{n_1}\cdots X_{k_1}^{n_{k_1}},$$
where $n_1, \cdots, n_{k_1}$ are nonnegative integers, and $a_{n_1\cdots n_{k_1}}$ 
is a smooth periodic function. By the definition of $\mathcal I_{\widetilde \rho}$, 
$n_1\rho_1+\cdots+n_{k_1}\rho_{k_1}$ is some 
$\widetilde \rho_i$. 
Hence, by Lemma \ref{lemma-SphericalHarmonics}, 
\begin{equation}\label{eq-property-I}
\sum_{i=2}^\infty c_i\eta_1^i=\sum_{i=1}^{\infty}
\Big\{\sum_{m=0}^{\widetilde M_i}
a_{im}(t)X_{m}(\theta)\Big\}e^{-\widetilde \rho_it},\end{equation}
where $a_{im}$ is a smooth periodic function. 
We now take the finite sum up to $l_1$ in the right-hand side and denote it by 
$I_1$, i.e., 
\begin{equation}\label{eq-definitioni-I}I_1=
\sum_{i=1}^{l_1}
\Big\{\sum_{m=0}^{\widetilde K_i}
a_{im}(t)X_{m}(\theta)\Big\}e^{-\widetilde \rho_it}.\end{equation}
Then, 
$$\mathcal F(\varphi)=I_1+O(e^{-\widetilde \rho_{l_1+1}t}),$$
and hence, by \eqref{eq-Asymptotic-Uk-21a}, 
$$\mathcal L\widetilde \varphi_1=\mathcal L\widetilde\eta_1-I_1+O(e^{-\widetilde \rho_{l_1+1}t}).$$
Consider $\widetilde\eta_1$ in the form 
$$\widetilde\eta_1(t,\theta)=\sum_{i=1}^{l_1}\sum_{m=0}^{\widetilde M_i}
\widetilde\eta_{im}(t)X_{m}(\theta).$$
To solve $\mathcal L\widetilde\eta_1=I_1$, we consider, for each $1\le i\le l_1$ and $0\le m\le \widetilde M_i$, 
\begin{equation}\label{eq-solving-linear1}L_m\widetilde\eta_{im}=a_{im}(t)e^{-\widetilde \rho_it}.\end{equation}
Since $\rho_m\neq\widetilde\rho_i$ for any $m$ and $i$, by Lemma \ref{lemma-decay-particular-solution-ODE}
and Remark \ref{remark-ODE-period}, we have a 
solution 
\begin{equation}\label{eq-solving-linear2}\widetilde\eta_{im}(t)=c_{im}(t)e^{-\widetilde \rho_it},
\end{equation} for some smooth periodic function 
$c_{im}$. In conclusion, we obtain a function $\widetilde \eta_1$ in the form \eqref{eq-Asymptotic-Uk-22}, 
and $\widetilde \varphi_1$ defined by \eqref{eq-Asymptotic-Uk-21} satisfies \eqref{eq-Asymptotic-Uk-23}. 
By \eqref{eq-Asymptotic-Uk-13} and \eqref{eq-Asymptotic-Uk-22}, we have 
\begin{equation}\label{eq-Asymptotic-Uk-23a}
\widetilde \varphi_1=O(e^{-\widetilde\rho_1t}).
\end{equation}

{\it Case 2:  We now assume $\rho_{k_1+1}>3\rho_1$.} Then, $\widetilde \rho_{l_1}\ge 3\rho_1$. 

Let $n_1$ be the largest integer such that $\widetilde \rho_{n_1}<3\rho_1$. Then, $\widetilde \rho_{n_1+1}=3\rho_1$. 
We can repeat the argument in Case 1 with $n_1$ replacing $l_1$. 
In defining $I_1$ in \eqref{eq-definitioni-I}, the summation is from $i=1$ to $n_1$. 
Similarly for $\widetilde \eta_1$ in \eqref{eq-Asymptotic-Uk-22}, we define
\begin{equation}\label{eq-Asymptotic-Uk-24}\widetilde\eta_{11}(t,\theta)=
\sum_{i=1}^{n_1}\Big\{\sum_{m=0}^{\widetilde M_i}
c_{im}(t)X_{m}(\theta)\Big\}e^{-\widetilde \rho_it},\end{equation}
for appropriate smooth periodic functions $c_{im}$, and then set
\begin{equation}\label{eq-Asymptotic-Uk-25}
\widetilde \varphi_{11}=\varphi_1-\widetilde\eta_{11}.
\end{equation}
A similar arguments yields
\begin{equation}\label{eq-Asymptotic-Uk-26}\mathcal L\widetilde \varphi_{11}
=O(e^{-\widetilde \rho_{n_1+1}t})= O(e^{-3\rho_1t}).\end{equation}
Moreover, by \eqref{eq-Asymptotic-Uk-13} and \eqref{eq-Asymptotic-Uk-24}, 
$$\widetilde \varphi_{11}= O(e^{-\widetilde \rho_1t})= O(e^{-2\rho_1t}).$$
We point out there is no $\rho_i$ between $\widetilde \rho_1$ and 
$\widetilde\rho_{n_1+1}$. 
Hence, by Lemma \ref{lemma-nonhomogeneous-linearized-eq}(ii), we have 
\begin{equation}\label{eq-Asymptotic-Uk-27}\widetilde \varphi_{11}=O(e^{-3\rho_1t}).\end{equation}
Note that \eqref{eq-Asymptotic-Uk-27} improves \eqref{eq-Asymptotic-Uk-23a} and hence \eqref{eq-Asymptotic-Uk-13}. 

Now, we are in a similar situation as at the beginning of Step 2, 
with $\widetilde\rho_{n_1+1}=3\rho_1$ replacing $\widetilde \rho_1=2\rho_1$. 
If $\rho_{k_1+1}<4\rho_1$, we proceed as in Case 1. If $\rho_{k_1+1}>4\rho_1$, we proceed as at the beginning 
of Case 2 by taking the largest integer $n_2$ such that $\widetilde \rho_{n_2}<4\rho_1$. 
After finitely many steps, we reach $\widetilde \rho_{l_1}$. 

In summary, we have $\widetilde\eta_1$ as in \eqref{eq-Asymptotic-Uk-22} and, 
by defining 
$\widetilde \varphi_1$ by \eqref{eq-Asymptotic-Uk-21}, 
we conclude \eqref{eq-Asymptotic-Uk-23}, 
as well as  \eqref{eq-Asymptotic-Uk-23a}. This finishes the discussion of Step 2. 

{\it Step 3.} Now we are in the same situation as in Step 1, 
with $\widetilde\rho_{l_1+1}$ replacing $\widetilde \rho_1$. We repeat the argument there with 
$k_1+1$, $k_2$ and $l_1+1$ replacing $1$, $k_1$ and $1$, respectively. 
Note $\rho_{k_2}<\widetilde \rho_{l_1+1}$. 
By \eqref{eq-Asymptotic-Uk-23} 
and Lemma \ref{lemma-nonhomogeneous-linearized-eq}(ii), we obtain 
$$\widetilde \varphi_1(t,\theta)=\sum_{i=k_1+1}^{k_2}c_i(t)X_i(\theta)e^{-\rho_i t}+O(e^{-\widetilde \rho_{l_1+1} t}),$$
where $c_i$ is a smooth periodic function, for $i=k_1+1, \cdots, k_2$.  
By \eqref{eq-Asymptotic-Uk-23a}, there is no need to adjust by terms involving $e^{-\rho_i t}$
corresponding to $i=1, \cdots, k_1$. 
Set 
\begin{equation}\label{eq-Asymptotic-Uk-31}\eta_2(t, \theta)=\sum_{i=k_1+1}^{k_2}c_i(t)X_i(\theta)e^{-\rho_i t},\end{equation}
and 
\begin{equation}\label{eq-Asymptotic-Uk-32}
\varphi_2=\widetilde \varphi_1-\eta_2.\end{equation}
Then, $\mathcal L\eta_2=0$, 
$\varphi_2=\varphi-\eta_1-\widetilde \eta_1-\eta_2,$  and 
\begin{equation}\label{eq-Asymptotic-Uk-33}\varphi_2=O(e^{-\widetilde \rho_{l_1+1} t}).\end{equation}

{\it Step 4.} The discussion is similar to Step 2. For some $\widetilde \eta_2$ to be determined, set 
\begin{equation}\label{eq-Asymptotic-Uk-41}
\widetilde \varphi_2=\varphi_2-\widetilde \eta_2.\end{equation}
Then, 
$$\mathcal L\widetilde \varphi_2=\mathcal F(\varphi)-\mathcal L\widetilde\eta_1-\mathcal L\widetilde\eta_2.$$
Note
$$\mathcal F(\varphi)=\mathcal F(\varphi_2+\eta_1+\widetilde\eta_1+\eta_2)
=\sum_{i=2}^\infty c_i(\varphi_2+\eta_1+\widetilde\eta_1+\eta_2)^i.$$
As in Step 2, we need to analyze 
$$\sum_{i=2}^\infty c_i(\eta_1+\widetilde\eta_1+\eta_2)^i.$$
In Step 2, by choosing $\widetilde\eta_1$ as in \eqref{eq-Asymptotic-Uk-22} appropriately, 
we use $\mathcal L\widetilde\eta_1$ to cancel the terms $e^{-\widetilde \rho_it}$ 
in $\mathcal F(\varphi)$, for $i=1, \cdots, l_1$. 
Proceeding similarly, we can find $\widetilde \eta_2$ in the form 
\begin{equation}\label{eq-Asymptotic-Uk-42}
\widetilde\eta_2(t,\theta)=\sum_{i=l_1+1}^{l_2}
\Big\{\sum_{m=0}^{\widetilde M_i} 
c_{im}(t)X_{m}(\theta)\Big\}e^{-\widetilde \rho_it}
\end{equation}
to cancel the terms $e^{-\widetilde \rho_it}$ in $\mathcal F(\varphi)$, for $i=l_1+1, \cdots, l_2$.
By defining 
$\widetilde \varphi_2$ by \eqref{eq-Asymptotic-Uk-41}, 
we conclude 
\begin{equation}\label{eq-Asymptotic-Uk-43}\mathcal L\widetilde \varphi_2=O(e^{-\widetilde \rho_{l_2+1}t}).\end{equation}

We can continue these steps indefinitely and hence finish the proof for the case \eqref{eq-special}. 

Next, we consider the general case; namely, some $\rho_i$ can be written 
as a linear combination of some of $\rho_1, \cdots, \rho_{i-1}$ with positive integer coefficients. 
We will modify discussion above to treat the general case. Whenever some $\rho_i$ coincides 
some $\widetilde \rho_{i'}$, an extra power of $t$ appears when solving $L_i\phi_i=a_i$,
according to Lemma \ref{lemma-decay-particular-solution-ODE},  and such 
a power of $t$ will generate more powers of $t$ upon iteration. 

For an illustration, we consider $\rho_{k_1}=\widetilde \rho_1$ instead of the strict inequality in 
\eqref{eq-arrangement}. This is the first time that some $\rho_i$ may coincide
some $\widetilde \rho_{i'}$.

We set $\varphi$ as in \eqref{eq-Asymptotic-Uk-1}. Then, we have \eqref{eq-Asymptotic-Uk-02}
and \eqref{eq-Asymptotic-Uk-03}, i.e., 
\begin{equation*}\label{eq-Asymptotic-Uk-02z}\varphi=O(e^{-\rho_1 t}),\end{equation*}
and 
\begin{equation*}\label{eq-Asymptotic-Uk-03z}\mathcal L\varphi=O(e^{-2\rho_1t})
=O(e^{-\widetilde \rho_1 t}).\end{equation*}
We proceed similarly as in Step 1. 
Take $k_*\in \{1, \cdots, k_1-1\}$ such that 
$$\rho_{k_*}<\rho_{k_*+1}=\cdots=\rho_{k_1}=\widetilde \rho_1=2\rho_1.$$ By 
Lemma \ref{lemma-nonhomogeneous-linearized-eq}(ii), we obtain 
$$\varphi(t, \theta)=\sum_{i=1}^{k_*}c_i(t)X_i(\theta)e^{-\rho_i t}+O(te^{-\widetilde \rho_1 t}),$$
where $c_i$ is a smooth periodic function, for $i=1, \cdots, k_*$. Instead of \eqref{eq-Asymptotic-Uk-11}, we define 
\begin{equation}\label{eq-Asymptotic-Uk-11z}\eta_1(t, \theta)=\sum_{i=1}^{k_*}c_i(t)X_i(\theta)e^{-\rho_i t},\end{equation}
and then define $\varphi_1$ as in 
\eqref{eq-Asymptotic-Uk-12}. Then, 
\begin{equation}\label{eq-Asymptotic-Uk-13z}\varphi_1=O(te^{-\widetilde \rho_1 t}).\end{equation}
Next, we proceed similarly as in Step 2. In the discussion of Case 1 in Step 2, we need to solve 
\eqref{eq-solving-linear1}, for each $1\le i\le l_1$ and $0\le m\le \widetilde{M}_i$. 
If $\rho_m\neq\widetilde \rho_i$, then $\widetilde\eta_{im}(t)$ is still given by 
\eqref{eq-solving-linear2}. If $\rho_m=\widetilde \rho_i$, then 
$\widetilde\eta_{im}$ has the form 
\begin{equation}\label{eq-solving-linear2z}
\widetilde\eta_{im}(t)=c_{i1m}(t)te^{-\widetilde \rho_it}+c_{i0m}(t)e^{-\widetilde \rho_it},
\end{equation}
where $c_{i1m}, c_{i0m}$ are smooth periodic functions. 
In particular, this is the case if $i=1$ and $m=k_1$. 
Then, by defining  $\widetilde\eta_1$ by \eqref{eq-Asymptotic-Uk-22}, with new  $\widetilde\eta_{im}(t)$ given by 
\eqref{eq-solving-linear2z}, 
and defining $\widetilde \varphi_1$ by \eqref{eq-Asymptotic-Uk-21}, we have \eqref{eq-Asymptotic-Uk-23}. 
We can modify the rest of the proof similarly. \end{proof} 

Denote the index set $\mathcal I$ by
a strictly increasing sequence $\{\mu_i\}_{i\ge 1}$ of positive constants. 
Here, we disregard the multiplicity. Obviously, $\mu_1=\rho_1=1$ and 
$\mu_2=\min\{2\rho_1, \rho_{n+1}\}$. 

According to the proof of Theorem \ref{thrm-Asymptotic-uk}, we can write the summation in 
\eqref{eq-estimate-u-k} in the following form: 
for any positive integer $m$ and any $(t, \theta)\in (1,\infty)\times\mathbb S^{n-1}$, 
\begin{equation}\label{eq-approximation1}\phi_m(t,\theta)
=\sum_{\rho_i\le \mu_m}c_i(t)X_i(\theta)e^{-\rho_it},\end{equation}
where $c_{i}$ is a smooth periodic function, and
\begin{equation}\label{eq-approximation2}
\widetilde\phi_m(t,\theta)=\sum_{\widetilde\rho_i\le \mu_m}\sum_{j=0}^{i-1}
\Big\{\sum_{l=0}^{\widetilde{M}_i}
c_{ijl}(t)X_{l}(\theta)\Big\}t^je^{-\widetilde\rho_it},\end{equation}
where 
$c_{ijl}$ is a smooth periodic function. 
We note that 
$\phi_m$
is a solution of $\mathcal L\phi_m=0$ and that $\widetilde\phi_m$ arises 
from the nonlinearity in the equation \eqref{eq-U}. 
In the special case $\mathcal I_\rho\cap\mathcal I_{\widetilde\rho}=\emptyset$, ${\widetilde\phi}_m$ has the form 
$$
\widetilde\phi_m(t,\theta)=\sum_{\widetilde\rho_i\le \mu_m}
\Big\{\sum_{l=0}^{\widetilde{M}_i}
c_{il}(t)X_{l}(\theta)\Big\}e^{-\widetilde\rho_it}.$$
If $\xi$ is a nonconstant 
periodic solution, then the period of $c_i$ in \eqref{eq-approximation1} and $c_{ijl}$ in \eqref{eq-approximation2} 
is the same  as that of $\xi$. If $\xi$ is a constant, 
then $c_i$ in \eqref{eq-approximation1} is constant and the period of $c_{ijl}$ in \eqref{eq-approximation2} 
is $2\pi/\sqrt{n-2}$.

\section{Some Identities}\label{sec-identities} 

Starting from this section, we study the $\sigma_k$-Yamabe equation. 
In this section, we establish some useful identities. 

Denote by $\theta=(\theta_2, \cdots, \theta_n)$ normal local coordinates on the sphere $\mathbb S^{n-1}$, 
and by $t$ the coordinate on $\mathbb R$. For convenience, we 
also write $t=\theta_1$. Let $w=w(t,\theta)$ be an at least $C^2$-function and consider the 
$n\times n$ matrix $\Lambda=\Lambda(w)$ given by 
\begin{align}\label{eq-matrix-Gamma-phi}
\begin{split}
\Lambda_{11}&={w}_{tt}-\frac12(1-w_t^2)-\frac{1}{2}|\nabla_{\theta}{w}|^2,\\
\Lambda_{ii}&={w}_{ii}+w_i^2+\frac12(1-w_t^2)-\frac{1}{2}|\nabla_{\theta}{w}|^2\quad\text{for } 2\le i\le n,\\
\Lambda_{ab}&={w}_{ab}+{w}_a{w}_b\quad\text{for } 1\le a\neq b\le n.
\end{split}
\end{align}

Consider two functions $\xi=\xi(t)$ and $\varphi=\varphi(t,\theta)$, at least $C^2$ in their arguments, and 
set
\begin{equation}\label{eq-sigma-k-w}w=\xi+\varphi.\end{equation}
We view $w$ as a perturbation of $\xi$ and now expand $\sigma_k(\Lambda(w)).$ 
In fact, we always have 
\begin{equation}\label{eq-sigma-expansion-012}
\sigma_{k}(\Lambda(\xi+\varphi))=\sigma_{k}(\Lambda(\xi))+L_\xi\varphi+R_\xi(\varphi),\end{equation}
where $L_\xi$ is the linearization of $\sigma_k(\Lambda(w))$ at $\xi$, and $R_\xi$ is the higher order term. 

For $\xi=\xi(t)$, $\Lambda(\xi)$ is a diagonal matrix. 
By \eqref{eq-matrix-Gamma-phi} (with $w$ replaced by $\xi$), we have 
$$\sigma_{k}(\Lambda(\xi))=2^{1-k}\Big(\,\begin{matrix} n\\ k\end{matrix}\,\Big) (1-\xi_t^2)^{k-1}
\Big(\frac{k}{n}\xi_{tt}+\frac{n-2k}{2n}(1-\xi_t^2)\Big).$$
Then, the equation \eqref{eq-k-Yamabe-w} for $\xi$ has the following form: 
\begin{equation}\label{eq-k-Yamabe-radial-equiv-z}
\xi_{tt}+\left(\frac{n}{2k}-1\right)(1-\xi_t^2)-\frac{n}{2k}(1-\xi_t^2)^{1-k}e^{-2k\xi}=0.
\end{equation}
Let $\xi$ be a  solution of \eqref{eq-k-Yamabe-radial-equiv-z} on $\mathbb R$ 
in the $\Gamma_k^+$ class.
According to \cite{ChangHanYang2005}, we have $|\xi_t|<1$ on $\mathbb R$ and, for some nonnegative constant $h$, 
\begin{equation}\label{eq-first-integral-z}
e^{(2k-n)\xi}(1-\xi_t^2)^k-e^{-n\xi}=h.
\end{equation}
In the following, we always assume $h>0$. 

The linearized operator $L_\xi$ was computed in \cite{HanLi2010}. 
We next provide a computation leading to 
expressions of both $L_\xi$ and $R_\xi$. 

\begin{lemma}\label{lemma-linearized-operator} 
Let $\xi$ be a solution of \eqref{eq-k-Yamabe-radial-equiv-z} on $\mathbb R$ 
in the $\Gamma_k^+$ class, satisfying \eqref{eq-first-integral-z}  for some $h>0$. Then, 
\begin{equation}\label{eq-linearization-sigma k}
    L_\xi\varphi=2^{1-k}\Big(\,\begin{matrix}n-1\\k-1\end{matrix}\,\Big)(1-\xi_t^2(t))^{k-1}
[\varphi_{tt}+a(t)\Delta_\theta\varphi+b(t)\varphi_t],
\end{equation}
and 
\begin{equation}\label{eq-quadratic-sigma k}
R_\xi(\varphi)=(1-\xi_t^2)^{k-1}Q_2(\varphi)
+\sum_{l=2}^{k}(1-\xi_t^2)^{k-l}P_{l}(\varphi),\end{equation}
where 
\begin{align}\label{eq-expressions-bc}\begin{split}
a(t)&=\frac{1}{n-1}\Big[\frac{n}{k}-1+\frac{n(k-1)}{k}\cdot\frac{e^{-n\xi}}{e^{-n\xi}+h}\Big],\\
b(t)&=\xi_t\Big[2-\frac{n}{k}-\frac{n(k-1)}{k}\cdot\frac{e^{-n\xi}}{e^{-n\xi}+h}\Big],
\end{split}\end{align}
and $Q_2(\varphi)$ is a quadratic polynomial of $\varphi_a$, 
and $P_l(\varphi)$ is a homogeneous polynomial of degree $l$ in terms of $\varphi_{ab}$, $\xi_t\varphi_{a}$, 
and $\varphi_a\varphi_b$, for $l=2, \cdots, k$. 
\end{lemma} 

\begin{proof} Throughout the proof, $Q_2(\varphi)$ and $P_l(\varphi)$ 
are always given as in the statement of the lemma. 
They can change from line to line. For simplicity, we write 
$Q_2$ and $P_l$ instead. 
By \eqref{eq-matrix-Gamma-phi}, the components of $\Lambda=\Lambda(\xi+\varphi)$ are given by  
\begin{align}\label{eq-expression-w-phi}
\begin{split}
\Lambda_{11}&=\xi_{tt}-\frac{1}{2}(1-\xi_t^2)+\varphi_{tt}
+\xi_t\varphi_t-\frac{1}{2}|\nabla_{\theta}\varphi|^2+\frac{1}{2}\varphi_t^2,\\
\Lambda_{1i}&=\varphi_{ti}+(\xi_t+\varphi_t)\varphi_i\quad\text{for } 2\le i\le n,\\
{\Lambda}_{ii}&=\frac12(1-\xi_t^2)+\varphi_{ii}-\xi_t\varphi_t+\varphi_i^2
-\frac{1}{2}|\nabla_{\theta}\varphi|^2-\frac{1}{2}\varphi_t^2\quad\text{for } 2\le i\le n,\\
{\Lambda}_{ij}&=\varphi_{ij}+\varphi_i\varphi_j\quad\text{for } 2\le i\neq j\le n.
\end{split}
\end{align}
Let $\bar\Lambda$ be the $(n-1)\times(n-1)$ matrix obtained by deleting the first row and the 
first column from the matrix $\Lambda$. 

Recall that  $\sigma_k(\Lambda)$ is the sum of $k\times k$ minors of $\Lambda$. 
Such minors can be arranged into two groups, depending on whether they include $\Lambda_{11}$, 
the $(1,1)$ component of $\Lambda$. For those not including $\Lambda_{11}$, the corresponding 
summation yields $\sigma_k(\bar\Lambda)$. 
We now consider  minors including $\Lambda_{11}$. For an illustration, we consider the 
$k\times k$ matrix consisting the first $k$ rows and the first $k$ columns of the matrix $\Lambda$. 
We can expand its determinant according to the first row and express it as the sum of $k$ terms. 
We keep the term containing $\Lambda_{11}$ and note that the rest 
terms can be written as 
$$\sum_{l=2}^{k}(1-\xi_t^2)^{k-l}P_{l}.$$
Therefore, 
\begin{equation}\label{eq-sigma-k-decomposition}
\sigma_k(\Lambda)=\Lambda_{11}\sigma_{k-1}(\bar\Lambda)
+\sigma_{k}(\bar\Lambda)+\sum_{l=2}^{k}(1-\xi_t^2)^{k-l}P_{l}.\end{equation}

By the definition of $\bar\Lambda$, it is clear that 
\begin{equation}\label{eq-sigma k of bar(Lambda)-phi}
\begin{split}
\sigma_{k-1}(\bar\Lambda)&=\Big(\begin{matrix}n-1\\k-1\end{matrix}\Big)(\frac{1}{2}(1-\xi_t^2))^{k-1}\\
        &\qquad+\Big(\begin{matrix}n-2\\k-2\end{matrix}\Big)(\frac{1}{2}(1-\xi_t^2))^{k-2}
        (\Delta_{\theta}\varphi-(n-1)\xi_t\varphi_t)\\
        &\qquad+(1-\xi_t^2)^{k-2}Q_2
        +\sum_{l=2}^{k-1}(1-\xi_t^2)^{k-1-l}P_{l}.
    \end{split}
\end{equation}
We point out that, for  terms involving $(1-\xi_t^2)^{k-2}$, 
we single out those whose coefficients are linear in derivatives of $\varphi$. 
A similar expression holds for $\sigma_{k}(\bar\Lambda)$. 
We now substitute \eqref{eq-sigma k of bar(Lambda)-phi} and the corresponding expression 
for $\sigma_{k}(\bar\Lambda)$, as well as the formula for $\Lambda_{11}$ given by \eqref{eq-expression-w-phi}, 
into \eqref{eq-sigma-k-decomposition}. Then, we obtain the desired result by a straightforward computation.
During the computation, we need to substitute $\xi_{tt}$ and $(1-\xi_t^2)^k$ 
by \eqref{eq-k-Yamabe-radial-equiv-z} and \eqref{eq-first-integral-z}, respectively.
\end{proof} 

Later on, we need the precise form of $R_\xi(\varphi)$ as in \eqref{eq-quadratic-sigma k}. 
A part of above computations will be needed in the proof of Lemma \ref{first integral for average}.

Let $\xi$ be a solution of \eqref{eq-k-Yamabe-radial-equiv-z}, 
satisfying \eqref{eq-first-integral-z}  for some $h>0$, 
and let $w$ be a solution of \eqref{eq-k-Yamabe-w}. Introduce $\varphi$ as in \eqref{eq-sigma-k-w}, i.e., $\varphi=w-\xi$. 
By Lemma \ref{lemma-linearized-operator}, we 
write the difference of \eqref{eq-k-Yamabe-w} and \eqref{eq-k-Yamabe-radial-equiv-z} as 
$$L_{\xi}\varphi+2^{1-k}k\left(\begin{matrix}n\\k\end{matrix}\right) e^{-2k\xi}\varphi=
2^{-k}\left(\begin{matrix}n\\k\end{matrix}\right) e^{-2k\xi}(e^{-2k\varphi}-1+2k\varphi)+R_\xi(\varphi).$$
By dividing $2^{1-k}\Big(\,\begin{matrix}n-1\\k-1\end{matrix}\,\Big)(1-\xi_t^2(t))^{k-1}$ and 
substituting $(1-\xi_t^2(t))^{k}$ by \eqref{eq-first-integral-z}, we can rewrite this identity 
as 
\begin{equation}\label{eq-nonlinear-linear-form}\mathcal L\varphi=\mathcal R(\varphi),\end{equation}
where 
\begin{equation}\label{eq-linear-operator-form}\mathcal L\varphi=\varphi_{tt}+a\Delta_\theta\varphi+b\varphi_t
+\frac{ne^{-n\xi}}{e^{-n\xi}+h}(1-\xi_t^2)\varphi,\end{equation}
and 
\begin{equation}\label{eq-nonlinear-operaor-form}\mathcal R(\varphi) 
= \frac{ne^{-n\xi}}{2k(e^{-n\xi}+h)}(1-\xi_t^2)(e^{-2k\varphi}-1+2k\varphi)
+Q_2(\varphi)+\sum_{l=2}^{k}(1-\xi_t^2)^{1-l}P_{l}(\varphi).
\end{equation}
Here, $a$, $b$, $Q_2$, and $P_l$ are as in Lemma \ref{lemma-linearized-operator}.

\section{Radial Solutions}\label{sec-Radial-Solutions}

In this section, we discuss some important properties of radial solutions of \eqref{eq-k-Yamabe-w}. 

The equation \eqref{eq-k-Yamabe-w} can be simplified significantly if $w$ does not depend on $\theta$, in which case
the corresponding $u$ is a radial solution of \eqref{eq-k-Yamabe-u}. Set $\xi(t)=w(t,\theta)$. Then, 
\begin{equation}\label{eq-k-Yamabe-radial-equiv}
\xi_{tt}+\left(\frac{n}{2k}-1\right)(1-\xi_t^2)-\frac{n}{2k}(1-\xi_t^2)^{1-k}e^{-2k\xi}=0.
\end{equation}
Let $\xi$ be a  solution of \eqref{eq-k-Yamabe-radial-equiv} on $\mathbb R$ 
in the $\Gamma_k^+$ class, with a nonremovable singularity at infinity.
According to \cite{ChangHanYang2005}, we have $|\xi_t|<1$ on $\mathbb R$ and, for some constant $h>0$, 
\begin{equation}\label{eq-first-integral}
e^{(2k-n)\xi}(1-\xi_t^2)^k-e^{-n\xi}=h.
\end{equation}
The identity \eqref{eq-first-integral} is referred to as the  first integral of $\xi$. 
We point out that \eqref{eq-k-Yamabe-radial-equiv} and \eqref{eq-first-integral} 
are simply \eqref{eq-k-Yamabe-radial-equiv-z} and \eqref{eq-first-integral-z}.


\begin{lemma}\label{lemma-xi-t-small-k}
For $n\ge 3$ and $2\le k<n/2$,  let $\xi$ be a
solution of \eqref{eq-k-Yamabe-radial-equiv} on $\mathbb R$ 
in the $\Gamma_k^+$ class, satisfying \eqref{eq-first-integral} for some $h>0$. 
Then, $\xi$ is periodic on $\mathbb R$. 
\end{lemma} 

Refer to \cite{ChangHanYang2005} for a proof. As for $k=1$, there is a constant solution given by 
$$\xi=-\frac{1}{2k}\log\big(1-\frac{2k}n\big).$$

\begin{lemma}\label{lemma-xi-t-large-k}
For $n\ge 3$ and $n/2\le k\le n$,  let $\xi$ be a
solution of \eqref{eq-k-Yamabe-radial-equiv} on $\mathbb R$ 
in the $\Gamma_k^+$ class, satisfying \eqref{eq-first-integral} for some $h>0$. 

$\mathrm{(i)}$ If $k=n/2$, then $h<1$ and there exists a constant $a_0$ such that, for $t\ge 0$, 
\begin{equation}\label{eq-estimate-xi-t-middle-k}
|\xi(t)-\sqrt{1-\sqrt[k]{h}}t-a_0|+|\xi_t(t)-\sqrt{1-\sqrt[k]{h}}|\leq Ce^{-n\sqrt{1-\sqrt[k]{h}}t},
\end{equation}
where $C$ is a positive constant. 

$\mathrm{(ii)}$ If $n/2<k\le n$, then there exists a constant $a_0$ such that, for $t\ge 0$, 
\begin{equation}\label{eq-estimate-xi-t-large-k}
|\xi(t)-t-a_0|+|\xi_t(t)-1|\leq Ce^{-(2-\frac nk)t},
\end{equation}
where $C$ is a positive constant. 
\end{lemma}

The estimates of $\xi$ were proved in \cite{ChangHanYang2005}, 
and then the estimates of $\xi_t$ follow easily from \eqref{eq-first-integral}. 

We note that radial solutions $\xi$ with a nonremovable singularity at infinity 
behave differently for $k< n/2$, $k=n/2$, and $k>n/2$. 
According to Lemma \ref{lemma-xi-t-small-k} and Lemma \ref{lemma-xi-t-large-k}, 
$\xi$ is bounded if $2\le k< n/2$, $\xi-\beta t$ is bounded for some constant $\beta\in (0,1)$ if 
$k=n/2$, and $\xi-t$ is bounded if $n/2<k\le n$. 

\smallskip 
Next, we expand $\xi$ up to arbitrary orders for  $n/2\le k\le n$. 
To this end, we introduce an index set. Denote by $\mathbb Z_+$ the collection of 
nonnegative integers. Set, for $k=n/2$,  
\begin{equation}\label{eq-def-index-0-k-middle}
\mathcal I_0=\big\{i\sqrt{1-\sqrt[k]{h}}:\, i\in \mathbb Z_+\big\},\end{equation} 
and, 
for $n/2<k\le n$,  
\begin{equation}\label{eq-def-index-0-k-large}
\mathcal I_0=\big\{i\Big(2-\frac nk\Big)+nj:\, i, j \in \mathbb Z_+\big\}.
\end{equation}
We denote  $\mathcal I_0$ by a strictly increasing sequence $\{\nu_i\}_{i=0}^\infty$. 

We now examine the case $n/2<k\le n$. 
Set \begin{equation}\label{eq-definition-rho}\rho_0=2-\frac nk.\end{equation}
Then, $0<\rho_0\le 1$. 
Let $l$ be the  integer such that $(l-1)\rho_0<n\le l\rho_0$. Then, $l\ge 3$, since $\rho_0\le 1$ and $n\ge 3$,
and $l\rho_0<n+\rho_0\le (l+1)\rho_0$. Hence, 
$$\nu_i=i\rho_0\quad\text{for }i=0, 1, \cdots, l.$$ 
If $n=l\rho_0$, then 
$$\nu_i=i\rho_0\quad\text{for any }i\ge 0.$$ 
If $n<l\rho_0$, then 
$$\nu_{l+1}=n+\rho_0,\quad \nu_{l+2}=(l+1)\rho_0, \quad \cdots.$$
We note that $n=l\rho_0$ for some $l$ if and only if $n/\rho_0$ is an integer. 
We now present several simple examples. 
For $k=n$, we have $\rho_0=1$ and hence $n/\rho_0$ is always an integer. For $k=n-1$, we have 
$\rho_0=(n-2)/(n-1)$ and hence $n/\rho_0=n(n-1)/(n-2)$, which is 6, 6, $20/3$, for $n=3, 4, 5$, respectively. 

\smallskip


We next prove an expansion of solutions of \eqref{eq-k-Yamabe-radial-equiv}. 
We will need Lemma \ref{lemma-decay-particular-solution-ODE} and 
Corollary \ref{cor-decay--sol-Ljw=f-ODE} in Appendix \ref{sec-Linear-Equations}.

\begin{lemma}\label{lemma-expansion-xi-t}
For $n\ge 3$ and $n/2\le k\le n$,  let $\xi$ be a
solution of \eqref{eq-k-Yamabe-radial-equiv} on $\mathbb R$ 
in the $\Gamma_k^+$ class, satisfying \eqref{eq-first-integral} for some $h>0$. Then, there 
exists a sequence of constants $\{a_i\}$ such that, for any $t\ge 1$ and any $m\ge 1$, 

$\mathrm{(i)}$ if 
$k=n/2$, then 
\begin{equation}\label{eq-expnsion-xi-t-k-middle}
\Big|\xi(t)-\sqrt{1-\sqrt[k]{h}}t- \sum_{i=0}^m a_ie^{-\nu_it}\Big|
+\Big|\xi_t(t)-\sqrt{1-\sqrt[k]{h}}+ \sum_{i=1}^m a_i\nu_ie^{-\nu_it}\Big|\le Ce^{-\nu_{m+1}t},
\end{equation}
where $C$ is a positive constant;

$\mathrm{(ii)}$ if
$n/2<k\le n$, then 
\begin{equation}\label{eq-expnsion-xi-t}
\Big|\xi(t)-t- \sum_{i=0}^m a_ie^{-\nu_it}\Big|
+\Big|\xi_t(t)-1+ \sum_{i=1}^m a_i\nu_ie^{-\nu_it}\Big|\le Ce^{-\nu_{m+1}t},
\end{equation}
where $C$ is a positive constant. 
\end{lemma}

\begin{proof} We will discuss (ii) only. Let $\rho_0$ be as given in \eqref{eq-definition-rho}. 
Set, for $\psi=\psi(t)$,  
\begin{equation}\label{eq-definition-L}
L_0\psi=\psi''+\rho_0 \psi'.\end{equation} 
Note that $L_0\psi=0$ has two linearly independent solutions $1$ and $e^{-\rho_0 t}$. 
Hence, Assumption \ref{assum-kernel-linear-equation-ODE} is valid for $L_0$ with 
$\psi^+=e^{-\rho_0 t}$ and $\psi^-=1$, 
and therefore, Lemma \ref{lemma-decay-particular-solution-ODE}
and Corollary \ref{cor-decay--sol-Ljw=f-ODE} are applicable. 
More generally, 
we have, for each $j$,  
\begin{equation}\label{eq-identity}
L_0(e^{-j\rho_0 t})=j(j-1)\rho_0^2e^{-j\rho_0 t}.\end{equation}

By \eqref{eq-k-Yamabe-radial-equiv} and \eqref{eq-first-integral}, we have 
\begin{equation}\label{eq-expression-xi-tt}
\xi''=\frac12(1-\xi'^2)\Big(\rho_0+\frac{n}{k}\cdot \frac{e^{-n\xi}}{e^{-n\xi}+h}\Big).\end{equation}
In the following, all estimates hold for $t\ge 1$. 
With $a_0$ as in Lemma \ref{lemma-xi-t-large-k}, set 
\begin{equation}\label{eq-definition-phi}
\psi=\xi-t-a_0.\end{equation} 
By Lemma \ref{lemma-xi-t-large-k}, we have
\begin{equation}\label{eq-estimate-phi-and-derivative-t}
|\psi(t)|+|\psi'(t)|\leq Ce^{-\rho_0 t}.
\end{equation}
By substituting $\xi=\psi+t+a_0$ in \eqref{eq-expression-xi-tt} and a straightforward calculation, we have 
\begin{equation}\label{eq-expression-phi-tt}L_0\psi=F,\end{equation}
where 
\begin{equation}\label{eq-expression-F} 
F=-\frac12\rho_0\psi'^2-\frac{n}{2k}(2\psi'+\psi'^2)\frac{e^{-n\psi -nt-na_0}}{h+e^{-n\psi-nt-na_0}}.
\end{equation}

By \eqref{eq-estimate-phi-and-derivative-t}, the second term in $F$ has an order $e^{-(n+\rho_0)t}$, and hence
$$|F|\le Ce^{-2\rho_0 t}.$$
Since $\psi$ is a solution of \eqref{eq-expression-phi-tt} with $\psi\to0$ as $t\to\infty$, by 
Corollary \ref{cor-decay--sol-Ljw=f-ODE} with $\gamma=2\rho_0$, there is a constant $a_1$ such that 
$$|\psi(t)-a_1e^{-\rho_0 t}|\le Ce^{-2\rho_0 t}.$$
By \eqref{eq-first-integral}, we obtain
\begin{equation}\label{eq-nonzero-a1} a_1=\frac{k}{2k-n}\frac{\sqrt[k]{h}}{2}e^{-(2k-n)a_0},\end{equation}
and hence $a_1\neq 0$. 
Set 
$$\psi_1=\psi-a_1e^{-\rho_0 t}.$$ Then, 
$$L_0\psi_1=F,$$ 
with 
$$|\psi_1(t)|+|F(t)|\le Ce^{-2\rho_0 t}.$$ 
The interior gradient estimates imply 
$$|\psi_1'(t)|\le Ce^{-2\rho_0 t}.$$
We now expand the first term of $F$. With $\psi=a_1e^{-\rho_0 t}+\psi_1$, we have 
$$\psi'^2=(-a_1\rho_0 e^{-\rho_0 t}+\psi_1')^2=a_1^2\rho_0^2 e^{-2\rho_0 t}-2a_1\rho_0 e^{-\rho_0 t}\psi_1'
+\psi_1'^2.$$
Then, 
$$|\psi'^2-a_1^2\rho_0^2 e^{-2\rho_0 t}|\le Ce^{-3\rho_0 t},$$ 
and hence
$$\Big|F+\frac12a_1^2\rho_0^3e^{-2\rho_0 t}\Big|\le Ce^{-3\rho_0 t}.$$
By \eqref{eq-identity} with $j=2$, we have 
$$L_0(\psi_1-a_2e^{-2\rho_0 t})=F-2a_2\rho_0^2e^{-2\rho_0 t}.$$
We now require $-2a_2\rho_0^2=a_1^2\rho_0^3/2$, and hence $a_2=a_1^2\rho_0/4$. Set 
$$\psi_2=\psi_1-a_2e^{-2\rho_0 t},$$ and 
$$F_2=F+ \frac12a_1^2\rho_0^3e^{-2\rho_0 t}.$$ 
Then, 
$$L_0\psi_2=F_2,$$ 
with 
$$|F_2(t)|\le Ce^{-3\rho_0 t}.$$ 
By Corollary \ref{cor-decay--sol-Ljw=f-ODE} with $\gamma=3\rho_0$, we obtain 
$$|\psi_2(t)|\le Ce^{-3\rho_0 t}.$$ 
The interior gradient estimates imply 
$$|\psi_2'(t)|\le C^{-3\rho_0 t}.$$

Note that the second term in $F$ given by \eqref{eq-expression-F} has an order $e^{-(n+\rho_0)t}$. 
Let $l$ be the largest integer such that $l\rho_0<n+\rho_0$. 
By proceeding inductively, we define, for appropriate constants 
$a_1, \cdots, a_l$, 
$$\psi_l=\psi-\sum_{i=1}^la_ie^{-i\rho_0 t},$$ 
such that 
$$L_0\psi_l=F_l,$$
where $\psi_l$ and $F_l$ satisfy 
$$|\psi_l(t)|+|F_l(t)|\le Ce^{-t\min\{n+\rho_0, (l+1)\rho_0\}}.$$

If $(l+1)\rho_0>n+\rho_0$, then
$$|\psi_l(t)|+|F_l(t)|\le Ce^{-(n+\rho_0)t}.$$
The term with the lowest decay rate $e^{-(n+\rho_0)t}$ in $F_l$ 
comes only from the second term of $F$. By proceeding similarly, 
there exists a constant $a_{l+1}$ such that 
$$|\psi_l(t)-a_{l+1}e^{-(n+\rho_0)t}|\le Ce^{-(l+1)\rho_0 t}.$$

If $(l+1)\rho_0=n+\rho_0$, then
$$|\psi_l(t)|+|F_l(t)|\le Ce^{-(l+1)\rho_0 t}.$$
The term with the lowest decay rate $e^{-(l+1)\rho_0 t}$ in $F_l$ 
comes from both terms of $F$. By proceeding similarly, 
there exists a constant $a_{l+1}$ such that 
$$|\psi_l-a_{l+1}e^{-(l+1)\rho_0 t}|\le Ce^{-(l+2)\rho_0 t}.$$

In both cases, we can continue indefinitely. 
\end{proof}

\section{Linearized Equations}\label{sec-Linearized-Equations}

In the following, we always assume that 
$\xi$ is a  solution of \eqref{eq-k-Yamabe-radial-equiv} on $\mathbb R$ 
in the $\Gamma_k^+$ class, 
satisfying \eqref{eq-first-integral} for some $h>0$. 
Let $\mathcal L$ be the linear operator given by \eqref{eq-linear-operator-form}, i.e., 
\begin{equation*}
\mathcal L\varphi=\varphi_{tt}+a\Delta_\theta\varphi+b\varphi_t
+c\varphi,\end{equation*}
where $a$ and $b$ are as in \eqref{eq-expressions-bc}, and 
$$c=\frac{ne^{-n\xi}}{e^{-n\xi}+h}(1-\xi_t^2).$$


We now project the operator $\mathcal L$ to spherical harmonics. 
Let $\{\lambda_i\}$ be the sequence of eigenvalues of $-\Delta_{\theta}$ on $\mathbb S^{n-1}$, 
arranged in an increasing order
with $\lambda_i\to\infty$ as $i\to\infty$, and let  $\{X_i\}$ be 
a sequence of the corresponding normalized eigenfunctions of 
$-\Delta_\theta$ on $L^2(\mathbb S^{n-1})$; namely, $-\Delta_{\theta}X_i=\lambda_iX_i$ for each $i\ge 0$. 
Note that each $X_i$ is a spherical harmonic of certain degree. In the following, 
we fix such a sequence $\{X_i\}$, which forms an orthonormal basis in $L^2(\mathbb S^{n-1})$.  
Refer to Appendix \ref{sec-Linear-Equations} for details.

For a fixed $i\ge 0$ and any $\psi=\psi(t)\in C^2(\mathbb R_+)$, we write 
$$\mathcal L(\psi X_i)=L_i(\psi)X_i.$$ 
By $-\Delta_\theta X_i=\lambda_iX_i$, we have 
\begin{equation}\label{eq-linearization-j}
L_i(\psi)=\psi_{tt}+p(t)\psi_t+q_i(t)\psi,
\end{equation}
where $p=b$ and $q_i=c-a\lambda_i$, or specifically 
\begin{align}\label{eq-expression-p}
p=\Big[2-\frac{n}{k}-\frac{n(k-1)}{k}\cdot\frac{e^{-n\xi}}{e^{-n\xi}+h}\Big]\xi_t,\end{align}
and 
\begin{align}\label{eq-expression-q}
q_i=-\frac{(n-k)\lambda_i}{k(n-1)}-\frac{n(k-1)\lambda_i}{k(n-1)}
\frac{e^{-n\xi}}{e^{-n\xi}+h}+\frac{ne^{-n\xi}}{e^{-n\xi}+h}(1-\xi_t^2).
\end{align}
In this section, we will characterize $\mathrm{Ker}(L_i)$, for each $i\geq 0$. 
We always denote by $W(\psi_i^+, \psi_i^-)$
the Wronskian of a basis $\psi_i^+, \psi_i^-$ of $\mathrm{Ker}(L_i)$.


\begin{lemma}\label{lemma-kernel-L-j-small-k1} 
Suppose $2\le k<n/2$ and $\xi$ is a constant.  

$\mathrm{(i)}$ For $i=0$, $\mathrm{Ker}(L_0)$ has a basis $\cos(\sqrt{n-2k}t)$ and $\sin(\sqrt{n-2k}t)$.

$\mathrm{(ii)}$  There exists an increasing sequence of positive constants $\{\rho_i\}_{i\ge 1}$, 
divergent to $\infty$, such that  
for any $i\ge 1$, 
$\mathrm{Ker}(L_i)$ has a basis 
$\psi_i^+=e^{-\rho_i t}$ and $\psi_i^-=e^{\rho_i t}$. Moreover, $\rho_1=\cdots=\rho_n=1$. 
\end{lemma} 

\begin{proof} We note that there is only one constant solution of \eqref{eq-k-Yamabe-radial-equiv}. 
Let $\xi$ be such a solution. Then, $p=0$ and 
$$q_i=-\frac{(n-k)\lambda_i}{k(n-1)}-\frac{n(k-1)\lambda_i}{k(n-1)}
(n-2k)+(n-2k).$$
Hence, $q_0=n-2k$, and $q_i<0$ for $i\ge 1$. In particular, $q_1=\cdots=q_n=-1$. 
We have the desired result. 
\end{proof} 


\begin{lemma}\label{lemma-kernel-L-j-small-k2} 
Suppose $2\le k<n/2$ and  $\xi$ is a positive nonconstant periodic solution. 

$\mathrm{(i)}$  For $i=0$, $\mathrm{Ker}(L_0)$ has a basis  
$\psi_0^+=p_0^+$ and $\psi_0^-=atp_0^++p_0^-$, for some smooth periodic functions $p_0^+$ and $p_0^-$ on $\mathbb R$,  
and some constant $a$. 

$\mathrm{(ii)}$ There exists an increasing sequence of positive constants $\{\rho_i\}_{i\ge 1}$, 
divergent to $\infty$, such that  
for any $i\ge 1$, 
$\mathrm{Ker}(L_i)$ has a basis 
$\psi_i^+=e^{-\rho_i t}p_i^+$ and $\psi_i^-=e^{\rho_i t}p_i^-$, for some smooth periodic functions 
$p_i^+$ and $p_i^-$  on $\mathbb R$. Moreover, $\rho_1=\cdots=\rho_n=1$. 

In both cases, the Wronskian $W(\psi_i^+, \psi_i^-)$ is a periodic function with a fixed sign. 
In addition, all periodic functions in $\mathrm{(i)}$ and $\mathrm{(ii)}$ have the same period as $\xi$.
\end{lemma} 

\begin{proof} 
Since $\xi$ is periodic, then $p$ and $q_i$ are also periodic. 
For a proof of the characterization of $\mathrm{Ker}(L_i)$, refer to \cite{HanLi2010}, which is based on the Floquet theory. 
We now provide a simple computation concerning the Wronskian. 

For each $i\ge 0$, the Wronskian $W_i$ of  $\psi_i^+$ and $\psi_i^-$ satisfies 
$W'_i=-p W_i,$
where $p$ is given by \eqref{eq-expression-p}. 
A simple integration yields
\begin{align}\label{eq-integral-p}
e^{-\int p}=(h+e^{-n\xi})^{-\frac{k-1}{k}}e^{-(2-\frac{n}{k})\xi},
\end{align}
and hence 
\begin{align}\label{eq-Wronskian}
W_i=c_i(h+e^{-n\xi})^{-\frac{k-1}{k}}e^{-(2-\frac{n}{k})\xi},
\end{align}
for some nonnegative constant $c_i$. 
\end{proof}

We now consider $n/2< k<n$. 
By \eqref{eq-estimate-xi-t-large-k}, 
$\xi_t\to 1$ as $t\to\infty$. Hence, as $t\to\infty$, 
\begin{align*}
p\to 2-\frac{n}{k},\quad 
q_i\to-\frac{(n-k)\lambda_i}{k(n-1)}.
\end{align*}
Replacing $p$ and $q_i$ in $L_i$ given by \eqref{eq-linearization-j}
by their limits as $t\to\infty$,  we have the following operator with 
constant coefficients: 
\begin{equation}\label{eq-linearization-j-infty}
L^\infty_i(\psi)=\psi_{tt}+\Big(2-\frac{n}{k}\Big)\psi_t-\frac{(n-k)\lambda_i}{k(n-1)}\psi.
\end{equation}
For each $i\ge 0$, solutions of $L_i^\infty(\psi)=0$ are given by linear combinations of 
\begin{equation}\label{eq-solutions-constant-ODE}
e^{-\rho_it}\quad\text{and}\quad 
e^{\tau_it},\end{equation}
where 
\begin{equation}\label{eq-rho-j}
\rho_i=\Big[\Big(\frac{n}{2k}\Big)^2+\frac{n-k}{k(n-1)}(\lambda_i-n+1)\Big]^{\frac12}+\Big(1-\frac{n}{2k}\Big),
\end{equation}
and 
\begin{equation}\label{eq-beta-j}
\tau_i=\Big[\Big(\frac{n}{2k}\Big)^2+\frac{n-k}{k(n-1)}(\lambda_i-n+1)\Big]^{\frac12}-\Big(1-\frac{n}{2k}\Big).
\end{equation}
Both $\{\rho_i\}$ and $\{\tau_i\}$ are nonnegative and increasing sequences and diverge to $\infty$, with 
\begin{equation}\label{eq-rho-0}
\rho_0=2-\frac{n}{k}, \,\, \tau_0=0,
\end{equation} 
and
\begin{equation}\label{eq-rho-1}\rho_i=1, \,\, \tau_i=\frac{n}{k}-1\quad\text{for }i=1, \cdots, n.\end{equation}
We note that $\rho_0$ in \eqref{eq-rho-0} is the same as $\rho_0$ in \eqref{eq-definition-rho}. 
We also have, for $i\ge 0$,  
\begin{equation}\label{eq-relation-rho-tau-j-000}
\rho_i-\tau_i=2-\frac{n}{k}. 
\end{equation}

For $k=n/2$, the above computation still holds, although $\xi_t$ has a different limit as $t\to\infty$. 
In this case, \eqref{eq-linearization-j-infty}, \eqref{eq-rho-j}, and \eqref{eq-beta-j} reduce to 
$$L^\infty_i(\psi)=\psi_{tt}-\frac{\lambda_i}{n-1}\psi,$$
and 
\begin{equation}\label{eq-relation-rho-tau-j}\rho_i=
\tau_i=\Big[\frac{\lambda_i}{n-1}\Big]^{1/2}.
\end{equation}
For $i=0$, solutions of $L_0^\infty(\psi)=0$ are given by linear combinations of 
$1$ and 
$t$. 
For each $i\ge 1$, solutions of $L_i^\infty(\psi)=0$ are given by linear combinations of 
$e^{-\rho_it}$ and 
$e^{\rho_it}$. 

We will prove that solutions of 
$L_i(\psi)=0$ behave similarly as those of $L_i^\infty(\psi)=0$ 
for $n/2\le k<n$.

For simplicity of presentations, we introduce a space of functions. 
Let $\{\nu_i\}_{i=0}^\infty$ be an  increasing sequence of nonnegative constants. 

\begin{definition}\label{def-space-E} 
Let $\psi$ be a function defined on $\mathbb R_+$. We say $\psi\in \mathcal E$ if 
there exists a sequence of contants $\{c_i\}_{\{i\ge 0\}}$ with $c_0\neq 0$ such that
for any integer $l\ge 1$ and any $t\ge 1$,
\begin{equation}\label{eq-expnsion-definition}
\Big|\psi(t)-\sum_{i=0}^lc_ie^{-\nu_it}\Big|\leq Ce^{-\nu_{l+1}t}.
\end{equation}
\end{definition} 

We emphasize that leading terms of functions in $\mathcal E$ are nonzero constants. 
In particular, functions in $\mathcal E$ are bounded on $\mathbb R_+$. 

We first characterize $\mathrm{Ker}(L_0)$. 

\begin{lemma}\label{lemma-kernel-L-0-large-k} 
The space $\mathrm{Ker}(L_0)$ has a basis $\psi_0^+$ and $\psi_0^-$ such that

$\mathrm{(i)}$ for $k=n/2$, $\psi_0^+\in \mathcal E$, 
and $\psi_0^-=at\psi_0^++\eta$ for some nonzero constant $a$ and some function $\eta\in \mathcal E$, 
and $W(\psi_0^+, \psi_0^-)\in \mathcal E$; 

$\mathrm{(ii)}$ for $n/2<k\le n$, $e^{(2-\frac{n}{k}) t}\psi_0^+, \psi_0^-\in \mathcal E$, 
and $e^{(2-\frac{n}{k}) t}W(\psi_0^+, \psi_0^-)\in \mathcal E$.\end{lemma} 

\begin{proof} 
We only prove (ii). 
By a simple differentiation of \eqref{eq-k-Yamabe-radial-equiv} with respect to $t$, we obtain 
$L_0(\xi_t)=0$. With $\{a_i\}$ given in Lemma \ref{lemma-expansion-xi-t} with $a_1\neq 0$ by 
\eqref{eq-nonzero-a1}, we have 
$\xi_t\in \mathcal E$ by \eqref{eq-expnsion-xi-t}. 
We set $\psi_0^-=\xi_t$. 
Recall that $\rho_0$ is given by \eqref{eq-definition-rho}. 

Next, we recall a well-known fact. Since $\xi_t$ is a solution of $L_0(\psi)=0$, then 
a linearly independent solution can be given by 
$$\psi_0^+(t)=\xi_t(t)\int_t^\infty \xi_t^{-2}e^{-\int p}d\tau.$$
By \eqref{eq-integral-p}, we have 
$$\psi_0^+(t)=\xi_t(t)\int_t^\infty \xi_t^{-2}(h+e^{-n\xi})^{-\frac{k-1}{k}}e^{-\rho_0\xi}d\tau.$$
Note that the leading term in $\xi_t$ is $1$. Then, the leading term in $\psi_0^+$ is given by 
$$h^{-\frac{k-1}{k}}\int_t^\infty e^{-\rho_0 \tau}d\tau=\frac{1}{\rho_0}h^{-\frac{k-1}{k}}e^{-\rho_0 t}.$$
The coefficient is not zero. 
By \eqref{eq-expnsion-xi-t}, 
we have $e^{\rho_0 t}\psi_0^+\in \mathcal E$. 

The assertion for the Wronskian follows from \eqref{eq-Wronskian}.
\end{proof}

By Proposition 2 in \cite{HanLi2010}, 
there is a pair of linearly independent elements in $\mathrm{Ker}(L_i)$, $i\geq 1$, 
one of which is bounded and another unbounded on $\mathbb R_+$. 
We now characterize 
these two functions precisely. Recall that $\rho_i$ and $\tau_i$ are defined by 
\eqref{eq-expression-p}
and 
\eqref{eq-expression-q}, respectively.

\begin{lemma}\label{lemma-kernel-L-j-large-k} For $n/2\le k<n$ and any $i\ge 1$, 
the space $\mathrm{Ker}(L_i)$ has a basis $\psi_i^+$ and $\psi_i^-$ such that 
$e^{\rho_i t}\psi_i^+$, $e^{-\tau_i t}\psi_i^-\in \mathcal E$, and 
$e^{\rho_0t}W(\psi_i^+, \psi_i^-)\in \mathcal E$. 
\end{lemma} 

\begin{proof} We consider only the case $n/2< k<n$. 

For $i=1, \cdots, n$, $\rho_i$ and $\tau_i$ are given by 
\eqref{eq-rho-1}. We take $\psi_i^+=(1+\xi_t)e^{-t}$ and $\psi_i^-=(1-\xi_t)e^t$. 
Then, $\psi^+_i$  and $\psi^-_i$ are two linearly independent solutions of $L_i(\psi)=0$. 
It is easy to get 
$e^{\rho_i t}\psi_i^+\in \mathcal E$. 
Next, by  \eqref{eq-first-integral}, we have 
\begin{equation*}
\psi_i^-=(1-\xi_t)e^t
=\frac{\sqrt[k]{h+e^{-n\xi}}}{1+\xi_t}\cdot e^{(-2+\frac{n}{k})(\xi-t)} \cdot e^{(\frac{n}{k}-1)t}.
\end{equation*}
Hence, $e^{-\tau_i t}\psi_i^-\in \mathcal E$. 

Next, we consider $i\ge n+1$. 
By (\ref{eq-linearization-j}) and (\ref{eq-linearization-j-infty}), we rewrite $L_{i}(\psi)=0$ as
\begin{equation*}
L^{\infty}_{i}(\psi)=f, 
\end{equation*}
where 
\begin{equation*}
f=(2-\frac{n}{k})(1-\xi_t)\psi_t
+\left[\frac{n(k-1)}{k}\xi_t\psi_t-\frac{n(k-1)\lambda_i}{k(n-1)}\psi+n(1-\xi_t^2)\psi\right]\frac{e^{-n\xi}}{e^{-n\xi}+h}.
\end{equation*}
In the first term in $f$, the factor $1-\xi_t$ contributes a factor $e^{-\rho_0 t}$, and in the second term, the 
factor $e^{-n \xi}$ contributes a factor $e^{-n t}$. 
Recall that solutions of $L^{\infty}_{i}(\psi)=0$ are given by linear combinations of $e^{-\rho_it}$ and $e^{\tau_it}$. 

Let $\psi^+_i\in \mathrm{Ker}(L_i)$ be a bounded function on $\mathbb R_+$. Arguing as in the proof of 
Lemma \ref{lemma-expansion-xi-t}, there exists a constant $a_0$ such that, for any $t\ge 1$,  
$$|\psi_i^+(t)-a_0e^{-\rho_it}|\le Ce^{-(\rho_i+\rho_0)t}.$$
Similarly, we get expansions of higher orders. Hence, $e^{\rho_i t}\psi_i^+\in \mathcal E$. 

Let $\psi^-_i\in \mathrm{Ker}(L_i)$ be an unbounded function on $\mathbb R_+$. Then, 
there exist constants $b_0$ and $c_0$ such that, for any $t\ge 1$,  
$$|\psi_i^-(t)-b_0e^{\tau_i t}-c_0e^{-\rho_it}|\le Ce^{(\tau_i-\rho_0)t}.$$
By considering $\psi_i^--c_0\psi_i^+$ instead of $\psi_i^-$, we may assume $c_0=0$ and 
$$|\psi_i^-(t)-b_0e^{\tau_i t}|\le Ce^{(\tau_i-\rho_0)t}.$$
Similarly, we get expansions of higher orders. Hence, $e^{-\tau_i t}\psi_i^-\in \mathcal E$. 
\end{proof} 

We now make an important remark concerning the sequence $\{L_i\}_{i\ge 0}$.

\begin{remark}\label{remark-application-linear-L} 
Suppose $2\le k<n$ and let $\{L_i\}_{i\ge 0}$ be given by \eqref{eq-linearization-j}. 
By Lemma \ref{lemma-kernel-L-j-small-k1}, Lemma \ref{lemma-kernel-L-j-small-k2}, Lemma \ref{lemma-kernel-L-0-large-k}, 
and Lemma \ref{lemma-kernel-L-j-large-k}, the sequence $\{L_i\}_{i\ge 0}$
satisfies Assumption \ref{assum-kernel-linear-equation}. As a consequence, 
Lemma \ref{lemma-nonhomogeneous-linearized-eq} is applicable to the operator $\mathcal L$
given by \eqref{eq-linear-operator-form}.
\end{remark} 

We point out that the sequences $\{\rho_i\}$ and $\{\tau_i\}$ 
in Lemma \ref{lemma-kernel-L-j-small-k1}, Lemma \ref{lemma-kernel-L-j-small-k2}, Lemma \ref{lemma-kernel-L-0-large-k}, 
and Lemma \ref{lemma-kernel-L-j-large-k}
are different for different $k$. 
For $2\le k< n/2$, $\rho_i=\tau_i$ for each $i\ge 0$, $\rho_0=0$, and $\{\rho_i\}_{i\ge 1}$ is determined 
by the Floquet theory and hence depends on the radial solution $\xi$ where the linearization was computed. 
For $n/2\le k<n$, $\{\rho_i\}$ and $\{\tau_i\}$ are given by \eqref{eq-rho-j}
and \eqref{eq-beta-j}, respectively, and are determined only by $n$ and $k$. Moreover, 
$\rho_0=0$ for $k=n/2$, and $0<\rho_0<1$ for $n/2<k<n$.  In all cases, we have $\rho_1=\cdots=\rho_n=1$. 

The solutions $\psi_0^+$ and $\psi_0^-$ of $L_0(\psi)=0$ also behave differently for different $k$. 
For $2\le k\le n/2$, $\psi_0^+$ is bounded and $\psi_0^-$ has at most a linear growth as $t\to\infty$. 
For $n/2<k<n$, $\psi_0^+$ decays exponentially as $t\to\infty$ and $\psi_0^-$ is bounded.


\section{Asymptotic Expansions}\label{sec-proof-main}

In this section, we discuss asymptotic expansions of solutions of \eqref{eq-k-Yamabe-w}. 
We first describe our strategy, which is already adopted in the proof of Theorem \ref{thrm-Asymptotic-uk}. 

Let $w=w(t, \theta)$ be a solution of \eqref{eq-k-Yamabe-w}, and $\xi$ be a radial solution. 
By setting $\varphi=w-\xi$, we write \eqref{eq-k-Yamabe-w} as \eqref{eq-nonlinear-linear-form}, i.e., 
$$\mathcal L\varphi=\mathcal R(\varphi),$$
where $\mathcal L$ and $\mathcal R$ are given by 
\eqref{eq-linear-operator-form} and \eqref{eq-nonlinear-operaor-form}, respectively. 
According to Remark \ref{remark-application-linear-L}, we can apply Lemma 
\ref{lemma-nonhomogeneous-linearized-eq} to the linear operator $\mathcal L$. 
Since $\mathcal R(\varphi)$ is nonlinear in $\varphi$, we will apply Lemma 
\ref{lemma-nonhomogeneous-linearized-eq} successively. In each step, we aim to 
get a decay estimate of $\mathcal R(\varphi)$, with a decay rate better than that of $\varphi$. 
Then, we can subtract expressions generated by Lemma 
\ref{lemma-nonhomogeneous-linearized-eq} with the lower decay rates to improve the decay rate of $\varphi$.

We now examine $\mathcal R(\varphi)$, 
which is given by \eqref{eq-nonlinear-operaor-form}, i.e., 
$$\mathcal R(\varphi) 
= \frac{ne^{-n\xi}}{2k(e^{-n\xi}+h)}(1-\xi_t^2)(e^{-2k\varphi}-1+2k\varphi)
+Q_2(\varphi)+\sum_{l=2}^{k}(1-\xi_t^2)^{1-l}P_{l}(\varphi).
$$
We note that negative powers of $1-\xi_t^2$ appear in the third term in the expression of $\mathcal R(\varphi)$. 
For $2\le k\le n/2$, $1-\xi_t^2$ has a positive lower bound, and hence negative powers of $1-\xi_t^2$ do not 
cause any trouble. 
However, for $n/2<k<n$, $1-\xi_t^2$ decays exponentially at the order of $\rho_0$, as $t\to\infty$, 
as hence negative powers of $1-\xi_t^2$ grow exponentially as $t\to\infty$.  Therefore, 
in order to estimate $\mathcal R(\varphi)$, we need to use the decay of  
$P_l(\varphi)$ to counterbalance the growth of $(1-\xi_t^2)^{1-l}$. 

We explain the initial step in slightly more details. 
For $n/2<k<n$, 
according to an estimate proved by Gursky and Viacolvsky \cite{Gursky-Viacolvsky2006}, and Li \cite{Li2006}, 
$\varphi=w-\xi$ decays as $e^{-\rho_0 t}$, as $t\to\infty$. 
For $l=2$ in $\mathcal R(\varphi)$, $(1-\xi_t^2)^{-1}$ grows as $e^{\rho_0 t}$, and the corresponding 
$P_2(\varphi)$ is at least quadratic in $\varphi$ and hence decays as $e^{-2\rho_0 t}$. As a result, the product 
$(1-\xi_t^2)^{-1}P_2(\varphi)$ decays as $e^{-\rho_0 t}$. There is no improvement in the decay rates, comparing 
those of $\mathcal R(\varphi)$ over $\varphi$. As the initial step in the successive improvement of decay rates, 
we prove $\varphi$ decays better than $e^{-\rho_0 t}$.  
This is crucial in proving Theorem \ref{thrm-approximation-order-1}.

Let $w=w(t, \theta)$ be a solution of \eqref{eq-k-Yamabe-w}, and set 
\begin{equation}
    \gamma(t)=\fint_{S^{n-1}} w(t,\theta) d\theta.
\end{equation}
By Theorem E and Proposition 1 in \cite{HanLi2010}, we have, for any $t\ge 0$,  
\begin{equation}\label{eq-decay-average}
    |w(t,\theta)-\gamma(t)|\leq Ce^{-t},
\end{equation}
and, for any $\delta>0$ small, there exists a constant $C=C(\delta)>0$ such that, for $l=1,2$, 
\begin{equation}\label{eq-decay-average-derivative}
    |\nabla^{l}_{t,\theta}(w(t,\theta)-\gamma(t))|\leq Ce^{-(1-\delta)t}. 
\end{equation}

We now derive an identity which plays a key role in our study of asymptotic expansions of $w$, 
for the case $n/2<k<n$. 
Compare \eqref{eq-first int for average} below with (48) in \cite{HanLi2010}. 

\begin{lemma}\label{first integral for average} 
For $n\ge 3$ and ${n}/{2}< k < n$, 
let $w(t,\theta)$ be a smooth solution of \eqref{eq-k-Yamabe-w} on $(0,\infty)\times \mathbb S^{n-1}$ 
in the $\Gamma_k^+$ class, with a nonremovable singularity at infinity. Then, there exists a constant $h>0$ such that
\begin{equation}\label{eq-first int for average}
    e^{(2k-n)\gamma}\Big((1-\gamma_t^2)^k(1+\eta)
    +\sum_{l=1}^{k}(1-\gamma_t^2)^{k-l}\eta_{l}\Big)=h+e^{-n\gamma}(1+\eta),
\end{equation}
where $\eta=O(e^{-2t})$ and
$\eta_l= O(e^{-l(1-\delta)t})$, for $l= 1, \cdots, k$, and 
any $\delta$ small. 
\end{lemma}

Some computations below already appeared in the proof of Lemma \ref{lemma-linearized-operator}. 
\begin{proof} 
Set
\begin{equation}\label{eq-hat-w}
    \hat{w}(t,\theta)= w(t,\theta)-\gamma(t).
\end{equation}
In the following, we always denote 
by $P_l$ a homogeneous polynomial of degree $l$ in terms of $\hat w_{ab}$, $\gamma_t\hat w_{a}$, 
and $\hat w_a\hat w_b$. They may change from line to line. 

Our proof starts with a Pohozaev type identity for the solution $w(t,\theta)$. 
Let $\Lambda=\Lambda(w)$ and $T$ be the Newton tensor associated with $\sigma_k(\Lambda)$ defined by 
\begin{equation}\label{eq-Newton-tensor}
T=T_{k-1}[w]=T_{k-1}(\Lambda)=\sum_{j=0}^{k-1}(-1)^{j}\sigma_{k-1-j}(\Lambda)\Lambda^j.\end{equation}
Write $T=(T_{ab})$. Then, for some $h>0$, 
\begin{equation}\label{eq-Pohozaev type identity}
   \fint_{\mathbb S^{n-1}}\Big[\frac{n}{2kc_k}e^{(2k-n)w}\sum_{a=1}^{n}T_{1a}w_{ta}
    -e^{-nw}\Big]d\theta=\frac{2k-n}{2k}h.
\end{equation}
The identity  \eqref{eq-Pohozaev type identity} was first derived by Viaclovsky in \cite{Viaclovsky2000}
and the present form is (60) in \cite{HanLi2010}.

By Lemma A \cite{Reilly1973},  we have
\begin{equation}\label{eq-identity-Newton tensor}
    \frac{d}{dt}\Big|_{t=0}\sigma_{k}(\Lambda+tB)=\mathrm{tr}(TB), 
\end{equation}
for any diagonalizable matrix $B$. We now express $T_{1a}$ in terms of $\Lambda_{ab}$. 
Let $\bar \Lambda_a$ be the $(n-1)\times (n-1)$ matrix by deleting the first row and the $a$-th column 
from $\Lambda$. The matrix $\bar\Lambda_1$ is symmetric, but not other $\bar \Lambda_i$ 
for $i=2, \cdots, n$. Take $B$ in \eqref{eq-identity-Newton tensor} 
to be the matrix $E^0$, whose $(1,1)$ component is 1 and all other 
components are zero. Then, 
\begin{equation}\label{eq-element-T-11}
    T_{11}=\sigma_{k-1}(\bar{\Lambda}_1). 
\end{equation}
Next, for any $i=2, \cdots, n$, take $B$ in \eqref{eq-identity-Newton tensor} to be the matrix $E^i$ 
whose $(1,i)$ and $(i,1)$ components are 1 and all others are 0.
A straightforward computation yields 
\begin{equation}\label{eq-element-T-1i}
    T_{1i}=(-1)^{1+i}\hat\sigma_{k-1}(\bar{\Lambda}_i),
\end{equation}
where $\hat\sigma_{k-1}(\bar{\Lambda}_i)$ is the sum of all $(k-1)\times(k-1)$ minors of $\bar{\Lambda}_i$. 

As \eqref{eq-expression-w-phi}, we have 
\begin{align}\label{eq-expression-w-hat}
\begin{split}
\Lambda_{11}&=\gamma_{tt}-\frac{1}{2}(1-\gamma_t^2)+\hat{w}_{tt}
+\gamma_t\hat{w}_t-\frac{1}{2}|\nabla_{\theta}\hat{w}|^2+\frac{1}{2}\hat{w}_t^2,\\
\Lambda_{1i}&=\hat{w}_{ti}+(\gamma_t+\hat{w}_t)\hat{w}_i\quad\text{for } 2\le i\le n,\\
{\Lambda}_{ii}&=\frac12(1-\gamma_t^2)+\hat{w}_{ii}-\gamma_t\hat{w}_t+\hat{w}_i^2
-\frac{1}{2}|\nabla_{\theta}\hat{w}|^2-\frac{1}{2}\hat{w}_t^2\quad\text{for } 2\le i\le n,\\
{\Lambda}_{ij}&=\hat{w}_{ij}+\hat{w}_i\hat{w}_j\quad\text{for } 2\le i\neq j\le n.
\end{split}
\end{align}
We now expand $T_{1a}$ according to the powers of $1-\gamma_t^2$. 
For $2\le i\le n$, it is easy to see that 
$\hat\sigma_{k-1}(\bar{\Lambda}_i)$ is a polynomial of $1-\gamma_t^2$ of 
degree $k-2$, and specifically
$$T_{1i}=\sum_{l=1}^{k-1}(1-\gamma_t^2)^{k-1-l}P_{l}.$$ 
By \eqref{eq-element-T-11} and \eqref{eq-element-T-1i}, we obtain
\begin{equation}\label{eq-summation-2-to-n}\sum_{a=1}^{n}T_{1a}w_{ta}=
(\gamma_{tt}+\hat w_{tt})\sigma_{k-1}(\bar \Lambda_1)+\sum_{l=2}^{k}(1-\gamma_t^2)^{k-l}P_{l},
\end{equation}
where $\bar \Lambda_1$ is the $(n-1)\times (n-1)$ matrix by deleting the first row and the first column 
from $\Lambda$. 

We now eliminate $\gamma_{tt}$ with the help of the equation  \eqref{eq-k-Yamabe-w}. 
Similar as \eqref{eq-sigma-k-decomposition}, we have 
$$\sigma_k(\Lambda)=\Lambda_{11}\sigma_{k-1}(\bar\Lambda_1)
+\sigma_{k}(\bar\Lambda_1)+\sum_{l=2}^{k}(1-\gamma_t^2)^{k-l}P_{l}.$$
By replacing $\Lambda_{11}$ by the first formula in \eqref{eq-expression-w-hat} and then substituting in 
\eqref{eq-summation-2-to-n}, we obtain 
\begin{align}\label{eq-summation-2-to-n-another}\begin{split}
\sum_{a=1}^{n}T_{1a}w_{ta}&=
\sigma_k(\Lambda)+\frac12(1-\gamma_t^2)\sigma_{k-1}(\bar \Lambda_1)-\sigma_{k}(\bar\Lambda_1)\\
&\qquad -\Big(\gamma_t\hat{w}_t+\frac{1}{2}\hat{w}_t^2-\frac{1}{2}|\nabla_{\theta}\hat{w}|^2\Big)
\sigma_{k-1}(\bar \Lambda_1)
+\sum_{l=2}^{k}(1-\gamma_t^2)^{k-l}P_{l}.
\end{split}\end{align}
By the definition of $\bar\Lambda_1$, it is clear that 
\begin{equation}\label{eq-sigma k of bar(Lambda)}
\begin{split}
\sigma_{k-1}(\bar\Lambda_1)&=\Big(\begin{matrix}n-1\\k-1\end{matrix}\Big)(\frac{1}{2}(1-\gamma_t^2))^{k-1}
        +\sum_{l=1}^{k-1}(1-\gamma_t^2)^{k-1-l}P_{l}.
    \end{split}
\end{equation}
A similar expression holds for $\sigma_{k}(\bar\Lambda_1)$. 
Substituting \eqref{eq-sigma k of bar(Lambda)} and the corresponding expression for $\sigma_{k}(\bar\Lambda_1)$
in \eqref{eq-summation-2-to-n-another}, we obtain 
\begin{equation*}
\begin{split}
    \sum_{a=1}^{n} T_{1a}w_{ta}=\sigma_{k}(\Lambda)
    +\frac{2k-n}{k}\Big(\begin{matrix}n-1\\k-1\end{matrix}\Big)(\frac{1}{2}(1-\gamma_t^2))^k
        +\sum_{l=1}^{k}(1-\gamma_t^2)^{k-l}P_{l}. 
\end{split}
\end{equation*}
Recall that $\sigma_{k}(\Lambda)=c_ke^{-2kw}=c_ke^{-2k\gamma}e^{-2k\hat{w}}$, by \eqref{eq-k-Yamabe-w}.
Therefore, 
\begin{equation*}
\begin{split}
    \frac{n}{2kc_k}\sum_{a=1}^{n} T_{1a}w_{ta}
    =\frac{n}{2k}e^{-2k\gamma}e^{-2k\hat{w}}+\frac{2k-n}{2k}(1-\gamma_t^2)^k
        +\sum_{l=1}^{k}(1-\gamma_t^2)^{k-l}P_{l}. 
\end{split}
\end{equation*}
Integrating over $\mathbb S^{n-1}$, we have 
\begin{equation*}
\begin{split}
    &\fint_{\mathbb S^{n-1}}\Big[\frac{n}{2kc_k}e^{(2k-n)w}\sum_{a=1}^{n} T_{1a}w_{ta}-e^{-nw}\Big]d\theta\\
    &\quad=-\frac{2k-n}{2k}e^{-n\gamma}\fint_{\mathbb S^{n-1}}e^{-n\hat{w}}d\theta
    +\frac{2k-n}{2k}e^{(2k-n)\gamma}(1-\gamma_t^2)^k\fint_{\mathbb S^{n-1}}e^{(2k-n)\hat{w}}d\theta\\
    &\qquad+\sum_{l=1}^{k}e^{(2k-n)\gamma}(1-\gamma_t^2)^{k-l}
    \fint_{\mathbb S^{n-1}}e^{(2k-n)\hat{w}}P_{l}d\theta.
\end{split}
\end{equation*}
By  \eqref{eq-decay-average}, we have, for any $l$, 
\begin{equation*}
    e^{l\hat{w}}= 1+ l\hat{w}+O(e^{-2t}),
\end{equation*}
and thus, with $\int_{\mathbb S^{n-1}}\hat{w}d\theta=0$, 
\begin{equation*}
\fint_{\mathbb S^{n-1}}e^{(2k-n)\hat{w}}d\theta=1+\eta.
\end{equation*}
Substituting these into \eqref{eq-Pohozaev type identity}, 
we obtain the desired identity with the help of \eqref{eq-decay-average-derivative}. 
\end{proof}

Now, we prove an  approximation by radial solutions, better than the known approximation, for the case $n/2<k<n$. 
Recall that $\rho_0=2-n/k$ as in \eqref{eq-definition-rho}.

\begin{lemma}\label{lemma-Improved-decay-estimate}
For $n\ge 3$ and ${n}/{2}< k < n$, 
let $w(t,\theta)$ be a smooth solution of \eqref{eq-k-Yamabe-w} on $\mathbb R_+\times \mathbb S^{n-1}$ 
in the $\Gamma_k^+$ class, with a nonremovable singularity at infinity. 
Then, there exist a constant $0<\varepsilon<\min\{1-\rho_0, \rho_0\}$ 
and a
radial solution $\xi(t)$ of \eqref{eq-k-Yamabe-w} on $\mathbb R\times \mathbb S^{n-1}$ 
in the $\Gamma_k^+$ class such that, for any $t>1$, 
\begin{equation}\label{eq-approximation-00-improved}
|w(t,\theta)-\xi(t)|\le Ce^{-(\rho_0+\varepsilon)t},
\end{equation} 
where $C$ is a positive constant.
\end{lemma}

\begin{proof}
Gursky and Viacolvsky \cite{Gursky-Viacolvsky2006}, and Li \cite{Li2006} proved that
the corresponding solution $u$ to \eqref{eq-k-Yamabe-u} satisfies 
$u\in C^{\rho_0}(B_1)$. This implies that we can take $\alpha=\rho_0$ in 
Theorem C. Hence, there exists a
radial solution $\xi(t)$ of \eqref{eq-k-Yamabe-w} on $\mathbb R\times \mathbb S^{n-1}$ 
in the $\Gamma_k^+$ class such that, for any $t>1$, 
\begin{equation}\label{eq-approximation-00}
|w(t,\theta)-\xi(t)|\le Ce^{-\rho_0 t},
\end{equation} 
where $C$ is a positive constant. In the following, $t$ is always at least 1.

Set $a_0=(\xi-t)(0)$. 
By \eqref{eq-expnsion-xi-t} and \eqref{eq-approximation-00}, we have
$$|w(t,\theta)-t-a_0|\leq Ce^{-\rho_0 t}.$$
A simple integration over $\mathbb S^{n-1}$ yields
\begin{equation}\label{eq-expansion-gamma-0 order}
    |\gamma(t)-t-a_0|\leq C e^{-\rho_0 t}.
\end{equation}
By combining \eqref{eq-expansion-gamma-0 order} and \eqref{eq-first int for average}, we get 
\begin{equation}\label{eq-first int for average 2}
   e^{(2k-n)t}\Big((1-\gamma_t^2)^k(1+\eta)+\sum_{l=1}^{k}(1-\gamma_t^2)^{k-l}\eta_{l}\Big)
   =he^{-(2k-n)a_0}+O(e^{-\rho_0 t}).
\end{equation}
For each $l=1, \cdots, k-1$, we write, for $\delta'>0$ small, 
$$(1-\gamma_t^2)^{k-l}\eta_{l}=(1-\gamma_t^2)^{k-l}e^{-\frac {k-l}k\delta't}\cdot \eta_{l}e^{\frac {k-l}k\delta't}.$$
By the H\"older inequality, we have 
$$|(1-\gamma_t^2)^{k-l}\eta_{l}|\le (1-\gamma_t^2)^{k}e^{- \delta't}
+C \eta_{l}^{\frac{k}{l}}e^{\frac {k-l}{l}\delta't}.$$
Note that
$$\eta_{l}^{\frac{k}{l}}=O(e^{-k(1-\delta)t}).$$ 
Hence, 
$$|(1-\gamma_t^2)^{k-l}\eta_{l}|\le (1-\gamma_t^2)^{k}e^{- \delta't}
+C e^{-k(1-\delta)t}e^{(k-1)\delta't}.$$
For $l=k$, we have 
$$(1-\gamma_t^2)^{k-l}\eta_{l}=\eta_k=O(e^{-k(1-\delta)t}).$$
By substituting these estimates in \eqref{eq-first int for average 2}, we obtain 
$$e^{(2k-n)t}(1-\gamma_t^2)^k\big(1+O(e^{- \delta't})\big)+O\big(e^{(k-n+k\delta+(k-1)\delta')t}\big)
   =he^{-(2k-n)a_0}+O(e^{-\rho_0 t}).$$
Since $k<n$, we choose $\delta$ and $\delta'$ sufficiently small such that, 
for some $\varepsilon>0$, 
$$e^{(2k-n)t}(1-\gamma_t^2)^k\big(1+O(e^{- \delta't})\big)=he^{-(2k-n)a_0}+O(e^{-\varepsilon t}),$$
and hence, 
\begin{equation}\label{ineq in improved decay estimate}e^{(2k-n)t}(1-\gamma_t^2)^k=he^{-(2k-n)a_0}+O(e^{-\varepsilon t}).
\end{equation}
This implies 
$$e^{\rho_0 t}(1-\gamma_t^2)\to \sqrt[k]{h}e^{-\rho_0 a_0}\quad\text{as }t\to\infty.$$

Set 
\begin{equation}\label{eq-expression-beta}\beta(t)=1-\gamma_t^2-\sqrt[k]{h}e^{-\rho_0(t+a_0)}.\end{equation}
Then, 
\begin{equation}\label{eq-beta-convergence}e^{\rho_0 t}\beta(t)\to 0\quad\text{as }t\to\infty.\end{equation}
By combining \eqref{ineq in improved decay estimate} and \eqref{eq-expression-beta}, we have 
$$\big(e^{\rho_0 t}\beta(t)+\sqrt[k]{h}e^{-\rho_0 a_0}\,\big)^k=he^{-(2k-n)a_0}+O(e^{-\varepsilon t}).$$
We now arrange the left-hand side according to powers of $\beta$. Note that the 
zero-order term is $he^{-(2k-n)a_0}$, which is the same as the first term in the right-hand side, and that 
$e^{\rho_0 t}\beta$ is a common factor of the rest terms in the left-hand side. Hence, 
$$e^{\rho_0 t}\beta(t)\sum_{j=1}^k
\Big(\,\begin{matrix} k\\j\end{matrix}\,\Big)\big(e^{\rho_0 t}\beta(t)\big)^{j-1}\big(\sqrt[k]{h}e^{-\rho_0a_0}\,\big)^{k-j}
=O(e^{-\varepsilon t}).$$
In the summation above, the term corresponding to $j=1$ is a positive constant and all other terms 
converge to 0 as $t\to\infty$ by \eqref{eq-beta-convergence}. 
Therefore, 
$$e^{\rho_0 t}\beta(t)
=O(e^{-\varepsilon t}),$$
or 
$$\beta(t)
=O(e^{-(\rho_0+\varepsilon) t}).$$
By \eqref{eq-expression-beta}, we obtain 
$$\gamma_t^2=1-\sqrt[k]{h}e^{-\rho_0(t+a_0)}+O(e^{-(\rho_0+\varepsilon) t}),$$
and hence 
\begin{equation}\label{eq-approximation-gamma-t-1}
\gamma_t=1-\frac12\sqrt[k]{h}e^{-\rho_0(t+a_0)}+O(e^{-(\rho_0+\varepsilon) t}).\end{equation}
A simple integration yields 
\begin{equation*}
\gamma(t)=t+a_0+\frac{k}{2(2k-n)}\sqrt[k]{h}e^{-\rho_0(t+a_0)}+O(e^{-(\rho_0+\varepsilon)t}),
\end{equation*}
where  $a_0$ is the constant from (\ref{eq-expansion-gamma-0 order}).
We always take $\varepsilon<\min\{1-\rho_0, \rho_0\}$.

By Lemma \ref{lemma-expansion-xi-t} and, in particular, \eqref{eq-nonzero-a1}, we have 
$$|\gamma(t)-\xi(t)|\le Ce^{-(\rho_0+\varepsilon)t}.$$
Hence, with \eqref{eq-decay-average}, 
\begin{equation*}
\begin{split}
    |w(t,\theta)-\xi(t)|\leq |w(t,\theta)-\gamma(t)|+|\gamma(t)-\xi(t)|
    \leq Ce^{-(\rho_0+\varepsilon)t}.
\end{split}
\end{equation*}
This is the desired result. \end{proof}

We are ready to prove Theorem \ref{thrm-approximation-order-1}.

\begin{proof}[Proof of Theorem \ref{thrm-approximation-order-1}]
Let $\varepsilon$ be as in Lemma \ref{lemma-Improved-decay-estimate}. 
By \eqref{eq-expnsion-xi-t} and \eqref{eq-approximation-gamma-t-1}, we have, 
for all $t> 0$ and $l=0,1,2$,
\begin{equation*}
    |\nabla^{l}_{t,\theta}(\gamma(t)-\xi(t))|\leq Ce^{-(\rho_0+\varepsilon)t}. 
\end{equation*}
Combining with Lemma \ref{lemma-Improved-decay-estimate} and \eqref{eq-decay-average-derivative}, 
we obtain, for all $t> 0$ and $l=0,1,2$,
\begin{equation}\label{eq-estimate-step0}
    |\nabla^{l}_{t,\theta}(w(t,\theta)-\xi(t))|\leq Ce^{-(\rho_0+\varepsilon)t}. 
\end{equation}

Set $\varphi=w-\xi$. By \eqref{eq-nonlinear-linear-form}, we can write \eqref{eq-k-Yamabe-w} as 
\begin{equation*}\mathcal L\varphi=\mathcal R(\varphi),\end{equation*}
where $\mathcal L\varphi$ and $\mathcal R(\varphi)$ are given by 
\eqref{eq-linear-operator-form} and \eqref{eq-nonlinear-operaor-form}, respectively. 
In $\mathcal R(\varphi)$, 
each $P_l(\varphi)$ is a homogeneous polynomial of degree $l$ in terms of $\varphi_{ab}$, $\xi_t\varphi_{a}$, 
and $\varphi_a\varphi_b$, for $l=2, \cdots, k$.

By \eqref{eq-estimate-step0}, we have, for any $t>1$ 
and $l=0,1,2$, 
\begin{equation}\label{eq-estimate-step1}
|\nabla^l_{(t,\theta)}\varphi(t,\theta)|\leq C e^{-(\rho_0+\varepsilon)t}.\end{equation}
By the expression of $\mathcal R(\varphi)$ in \eqref{eq-nonlinear-operaor-form}, 
the leading term in $\mathcal R(\varphi)$ is $(1-\xi_t^2)^{-1}P_2(\varphi)$, given by $l=2$. Hence, 
\begin{equation}\label{RHS-linearized-eq-0}
    |\mathcal R(\varphi)|\leq C e^{-(\rho_0+2\varepsilon) t}.
\end{equation}
If $\rho_0+2\varepsilon<1$, by Lemma \ref{lemma-nonhomogeneous-linearized-eq}, we have 
$$ |\mathcal \varphi|\leq C e^{-(\rho_0+2\varepsilon) t}.$$
Then, as in getting \eqref{RHS-linearized-eq-0}, we obtain 
\begin{equation}\label{RHS-linearized-eq-1}
    |\mathcal R(\varphi)|\leq C e^{-(\rho_0+4\varepsilon) t}.
\end{equation}
By adjusting $\varepsilon$, we assume $\rho_0+2^{l-1}\varepsilon<1<\rho_0+2^{l}\varepsilon$, 
for some positive integer $l$. Then, after finitely many steps, we obtain 
\begin{equation}\label{RHS-linearized-eq-l}
    |\mathcal R(\varphi)|\leq C e^{-(\rho_0+2^l\varepsilon) t}.
\end{equation}
For $\rho_0+2^l\varepsilon>1$,  by Lemma \ref{lemma-nonhomogeneous-linearized-eq}, 
there is a spherical harmonic $Y_1$ of degree 1 such that 
\begin{equation*}
    |\varphi(t,\theta)-(1+\xi_t)e^{-t}Y_1|\leq C e^{-(\rho_0+2^l\varepsilon)t}.
\end{equation*}
This is the desired result. 
\end{proof}

We now discuss higher order expansions and make two preparations for the proof of 
Theorem \ref{thrm-approximation-order-arbitrary}. 

In the first preparation, we introduce the {\it index set} $\mathcal I$. 
Let $\xi$ be a solution of \eqref{eq-k-Yamabe-radial-equiv-z} on $\mathbb R$ 
in the $\Gamma_k^+$ class, satisfying \eqref{eq-first-integral-z}  for some $h>0$, 
and let $\{\rho_i\}$ be the sequence of nonnegative  
constants as in Lemma \ref{lemma-kernel-L-j-small-k1}, Lemma \ref{lemma-kernel-L-j-small-k2}, 
Lemma \ref{lemma-kernel-L-0-large-k}, 
and Lemma \ref{lemma-kernel-L-j-large-k}. We note that $\rho_0=0$ for $2\le k\le n/2$ 
and $\rho_0=2-{n}/{k}\in (0,1)$ for $n/2<k<n$, and $\rho_1=1$. 
We also note that $\{\rho_i\}$ is determined by $\xi$ for $2\le k< n/2$, and by only $n$ and $k$ 
for  $n/2\le k<n$. In the following, we denote by $\mathbb Z_+$ the collection of nonnegative integers. 

We first consider $2\le k< n/2$. 
Define the index set $\mathcal I$ by  
\begin{align}\label{eq-def-index-sigma-k-small}
\mathcal I=\Big\{\sum_{i\ge 1} m_i\rho_i;\, m_i\in \mathbb Z_+\text{ with finitely many }m_i>0\Big\}.
\end{align}
This is the same as the index set defined in \eqref{eq-def-index} for $k=1$. 
For $k\ge 2$, the nonlinear term $\mathcal R(\varphi)$ in \eqref{eq-nonlinear-linear-form} 
involves $\xi$ and $\xi_t$. 
For $2\le k<n/2$, $\xi$ is periodic. 
Hence, $\xi$ and $\xi_t$ in $\mathcal R(\varphi)$ do not contribute extra decay orders. 

Next, we modify $\mathcal I$ in \eqref{eq-def-index-sigma-k-small} to construct index sets for 
$k=n/2$ and $n/2<k<n$. 

For $k=n/2$, we define 
\begin{align}\label{eq-def-index-sigma-k-middle}
\mathcal I=\Big\{m_0n\sqrt{1-\sqrt[k]{h}}+\sum_{i\ge 1}m_i\rho_i;\, m_0, m_i\in\mathbb Z_+
\text{ with finitely many }m_i>0\Big\}.
\end{align}
For $k=n/2$, $n\sqrt{1-\sqrt[k]{h}}$ is 
the coefficient of $t$  in the asymptotic expansion of $\xi$ as in \eqref{eq-estimate-xi-t-middle-k}. 
Hence, the factor $e^{-n\xi}$ 
in the first term of $\mathcal R(\varphi)$ contributes an extra decay order 
$n\sqrt{1-\sqrt[k]{h}}$. We note that $e^{-n\xi}$ does not appear by itself and is always coupled with 
nonlinear factors of $\varphi$. Hence, we need to add positive integer multiple of $n\sqrt{1-\sqrt[k]{h}}$
to the collection in \eqref{eq-def-index-sigma-k-small}. This results in the index set in 
\eqref{eq-def-index-sigma-k-middle}. 

Last, we consider $n/2<k<n$. With $\rho_0=2-{n}/{k}$, define   
\begin{align}\label{eq-def-index-sigma-k-large}
\mathcal I=[1,\infty)\cap \Big\{
m_0\rho_0+\sum_{i\ge 1}m_i\rho_i;\, m_0\in \mathbb Z, m_i\in\mathbb Z_+
\text{ with finitely many }m_i>0\Big\}.
\end{align}
For $n/2<k<n$, the factor $1-\xi_t^2$ in the first term of $\mathcal R(\varphi)$ contributes an extra decay order 
$\rho_0=2-{n}/{k}$. Similar as $k=n/2$, this factor does not appear by itself and is always coupled with 
nonlinear factors of $\varphi$. The other factor $e^{-n\xi}$ 
in the first term of $\mathcal R(\varphi)$ contributes a decay order $n$, 
which is simply $n\rho_1$. In the third term of $\mathcal R(\varphi)$, we have 
functions of the form $e^{(l-1)\rho_0t}P_l(\varphi)$, for $l=2, \cdots, k$. The positive exponent 
$(l-1)\rho_0$ means we need to subtract $(l-1)\rho_0$ from the decay order of $P_l(\varphi)$. 
This allows the coefficient of $\rho_0$ to be negative. However, the entire 
linear combination has to be at least 1. 
This results in the index set in 
\eqref{eq-def-index-sigma-k-large}. 

We note that $\rho_1=\cdots=\rho_n=1$ in all cases.

\smallskip 

In the second preparation, we examine the nonlinear term $\mathcal R(\varphi)$. 
For the Yamabe equation \eqref{eq-U}, the nonlinear term 
$\mathcal R(\varphi)$ is given by  \eqref{eq-Asymptotic-Uk-3} and involves $\varphi$ only. 
Hence, in the proof of Theorem \ref{thrm-Asymptotic-uk}, 
we need to employ Lemma \ref{lemma-Asymptotics-U2a} to write the product 
of spherical harmonics in terms of a linear combination of finitely many spherical harmonics. 
For the $\sigma_k$-Yamabe equation, the nonlinear term 
$\mathcal R(\varphi)$ is given by  \eqref{eq-nonlinear-operaor-form} and involves $\varphi$ and its 
derivatives up to order 2. In order to follow the proof of Theorem \ref{thrm-Asymptotic-uk}, 
we need to demonstrate that certain derivatives of spherical harmonics are 
linear combinations of finitely many spherical harmonics.

To address this issue, we introduce the following space. 

\begin{definition}\label{def-expansion-spherical} 
For each integer $d\ge 1$, denote by $\s_d$  the collection of finite sums  
$$\sum{c_i}(t)Y_i,$$
where $Y_i$ is a spherical harmonic on $\mathbb S^{n-1}$ of degree not exceeding $d$, and 
$c_i$ is a smooth function on an interval $I\subset \mathbb R$ with bounded derivatives, for each $i$. 
\end{definition}

In other words, functions in $\s_d$ can be expressed as a finite linear combination of spherical harmonics of degree 
up to $d$. 
For our purpose, we take the interval $I=[1,\infty)$.

We need the following crucial lemma. 

\begin{lemma}\label{lemma-Pk-spherical} 
For each $l=2, \cdots, k$, 
let $P_l(\varphi)$ be the homogeneous polynomial of degree $l$ 
in terms of $\varphi_{ab}$, $\xi_t\varphi_{a}$, 
and $\varphi_a\varphi_b$, as in the expression of $\mathcal R(\varphi)$ in 
\eqref{eq-nonlinear-operaor-form}. 
If $\varphi\in \s_d(t)$, then $P_l(\varphi)\in \s_{2ld}(t)$. 
\end{lemma}

We postpone the proof of Lemma \ref{lemma-Pk-spherical} to Appendix \ref{sec-appendix1}, 
which involves some lengthy computations.  

We can prove Theorem \ref{thrm-approximation-order-arbitrary} by a similar method as 
in the proof of Theorem \ref{thrm-Asymptotic-uk}, with Lemma \ref{lemma-Asymptotics-U2a} 
replaced by Lemma \ref{lemma-Pk-spherical}. We omit details.

\appendix 

\section{Asymptotic Behaviors of Solutions of Linear Equations}\label{sec-Linear-Equations}

In this section, we consider  linear equations on 
$\mathbb R_+\times \mathbb S^{n-1}$ and study asymptotic behaviors of their solutions as $t\to\infty$. 
Linear equations in this section are modeled after the linearization of the Yamabe equation 
and the $\sigma_k$-Yamabe equation \eqref{eq-k-Yamabe-w}.  
Results in this section are well-known. For completeness, we include proofs, which are adapted from 
\cite{HanLi2010},  \cite{KorevaarMPS1999}, 
\cite{MazzeoP1999},
and \cite{MazzeoDU1996}. 

We first consider equations on $\mathbb R_+$. 
Let $p,q$ be smooth functions on $\mathbb R_+$, with bounded derivatives of arbitrary orders. 
Consider the linear operator 
$L$  given by 
\begin{equation}\label{eq-linear-operaor-form-ODE}
L\psi=\psi_{tt}+p(t)\psi_t
+q(t)\psi.\end{equation}
We will discuss the linear equation 
\begin{equation}\label{eq-linearization-ODE}
L\psi=f\quad\text{on }[t_0, \infty),
\end{equation}
for some $t_0>0$. 

We now examine a simple case that $p=0$ and $q=-\lambda$ for some constant $\lambda\ge 0$. 
If $\lambda>0$,
there is a pair of linearly independent elements in $\mathrm{Ker}(L)$
given by $e^{-\sqrt{\lambda}t}$ and $e^{\sqrt{\lambda}t}$,
one of which decays exponentially on $\mathbb R_+$ and another grows exponentially  on $\mathbb R_+$. 
If $\lambda=0$,
there is a pair of linearly independent elements in $\mathrm{Ker}(L)$
given by $1$ and $t$,
one of which is bounded on $\mathbb R_+$ and another unbounded on $\mathbb R_+$. 

We introduce an assumption for the general $L$ with a similar property. 

\begin{assumption}\label{assum-kernel-linear-equation-ODE} Let $L$ be the operator given by 
\eqref{eq-linear-operaor-form-ODE}. 
The space $\mathrm{Ker}(L)$ has a basis $\psi^+$ and $\psi^-$ such that, for some nonnegative constants $\rho\ge \tau$, 

(i)  if $\rho>0$, then $e^{\rho t}\psi^+$ and $e^{-\tau t}\psi^-$ 
are bounded on $\mathbb R_+$, 
and $\psi^-$ does not converge to 0 as $t\to\infty$; 

(ii)  if $\rho=\tau=0$, then $\psi^+$ is bounded on $\mathbb R_+$, $\psi^-=\eta+at\psi^+$ 
for some constant $a$ and bounded function $\eta$ on  $\mathbb R_+$, 
and $\psi^+, \psi^-$ do not converge to 0 as $t\to\infty$.

In both cases,  the Wronskian $W$ of $\{\psi^+,\psi^-\}$ is assumed to satisfy, for any $t\in \mathbb R_+$ 
\begin{equation}\label{eq-Wonskian-ODE}c_1\le |W(t)|e^{(\rho-\tau) t}\le c_2,\end{equation}
for some positive constants $c_1$ and $c_2$. 
\end{assumption}

We point out that the constant $a$ in Assumption \eqref{assum-kernel-linear-equation-ODE}(ii) may be zero, 
in which case $\psi^+$ and $\psi^-$ are bounded but do not converge to 0 as $t\to\infty$.

We first construct a special solution of \eqref{eq-linearization-ODE}. 

\begin{lemma}\label{lemma-decay-particular-solution-ODE}
Suppose Assumption \ref{assum-kernel-linear-equation-ODE} holds. 
Let $\gamma>0$ be a constant,  $m\geq 0$ be an integer, 
and $f$ be a smooth function on $[t_0,\infty)$ for some $t_0>0$, satisfying, 
for all $t\in[t_0,\infty)$, 
$$|f(t)|\leq C t^me^{-\gamma t}.$$
Then, there is a solution $\psi_*$ of $L\psi=f$ on $(t_0,\infty)$ such that, for all $t\in(t_0,\infty)$, 
\begin{equation*}
|\psi_*(t)|\leq\left\{
\begin{aligned}
&Ct^me^{-\gamma t}\hspace{1.4cm} \gamma\neq \rho,\\
&Ct^{m+1}e^{-\gamma t}\hspace{1cm} \gamma=\rho.
\end{aligned}
\right.
\end{equation*}
\end{lemma}

\begin{proof}
Let $\psi^+$ and $\psi^-$ be the two solutions of $L\psi=0$ 
as in Assumption \ref{assum-kernel-linear-equation-ODE}, 
and $W$ be their Wronskian. 

We first consider the case $\rho>0$; namely, $e^{\rho t}\psi^+$ and $e^{-\tau t}\psi^-$ 
are bounded on $(t_0,\infty)$. 
Then, by \eqref{eq-Wonskian-ODE}, 
$$\left|\frac{\psi^+ f}{W}\right|\le Ct^me^{(-\tau-\gamma)t},\quad 
\left|\frac{\psi^- f}{W}\right|\le 
Ct^m e^{(\rho-\gamma)t}.$$
Thus, a particular solution $\psi_*$ can be given by the following expressions: 
for $\gamma\leq \rho$,
$$
\psi_*(t)=-\psi^+(t)\int_{t_0}^t\frac{\psi^-(s)f(s)}{W(s)}ds-\psi^-(t)\int_t^{\infty}\frac{\psi^+(s)f(s)}{W(s)}ds,$$ 
and, for $\gamma> \rho$,
\begin{equation}\label{eq-expression-gamma-large}
\psi_*(t)=\psi^+(t)\int_{t}^{\infty}\frac{\psi^-(s)f(s)}{W(s)}ds-\psi^-(t)\int_t^{\infty}\frac{\psi^+(s)f(s)}{W(s)}ds.
\end{equation}
Hence, for $\gamma\neq \rho$,
\begin{equation}\label{eq-decay-particular}
|\psi_*(t)|\leq C t^me^{-\gamma t},
\end{equation}
and, for $\gamma =\rho$,
\begin{equation*}
|\psi_*(t)|\leq C t^{m+1}e^{-\gamma t}. 
\end{equation*}
This is the desired result. 

Next, we consider the case $\rho=\tau=0$; namely, 
$\psi^+$ is bounded on $(t_0,\infty)$, and $\psi^-=\eta+at\psi^+$ 
for some bounded function $\eta$ on  $(t_0,\infty)$.  
A particular solution $\psi_*$ can be given by \eqref{eq-expression-gamma-large}. 
Note that $|W|$ has positive upper and lower bounds by \eqref{eq-Wonskian-ODE}. 
If both $\psi^+$ and $\psi^-$ are bounded,  then \eqref{eq-decay-particular}
follows easily. Let $\psi^-$ be as assumed. Then, 
\begin{align*}\psi_*(t)&=\psi^+(t)\int_{t}^\infty\frac{\eta(s)}{W(s)}f(s)ds-
\eta(t)\int_{t}^\infty\frac{\psi^+(s)}{W(s)}f(s)ds\\
&\qquad+a\psi^+(t)\int_{t}^\infty\frac{s\psi^+(s)}{W(s)}f(s)ds-
at\psi^+(t)\int_{t}^\infty\frac{\psi^+(s)}{W(s)}f(s)ds\\
&=\psi^+(t)\int_{t}^\infty\frac{\eta(s)}{W(s)}f(s)ds-
\eta(t)\int_{t}^\infty\frac{\psi^+(s)}{W(s)}f(s)ds\\
&\qquad+a\psi^+(t)\int_{t}^\infty\int_s^\infty\frac{\psi^+(\tau)}{W(\tau)}f(\tau)d\tau ds.\end{align*}
Hence, \eqref{eq-decay-particular} follows. 
\end{proof}

\begin{corollary}\label{cor-decay--sol-Ljw=f-ODE}
Suppose Assumption \ref{assum-kernel-linear-equation-ODE} holds. 
Let $\gamma>0$ be a constant, $m\geq 0$ be an integer, and $f$ be a smooth function on $[t_0,\infty)$ for some $t_0>0$, 
satisfying, for all $t\in[t_0,\infty)$,  
$$|f(t)|\leq C t^me^{-\gamma t}.$$
Let $\psi$ be a solution of $L\psi=f$ on $(t_0,\infty)$ such that 
$\psi(t)\rightarrow 0$ as $t\rightarrow \infty$.
Then, for all $t\in[t_0,\infty)$,  
\begin{equation*}
|\psi(t)|\leq\left\{
\begin{aligned}
&Ct^me^{-\gamma t}\hspace{1cm}\text{if } \gamma< \rho,\\
&Ct^{m+1}e^{-\gamma t}\hspace{0.6cm}\text{if } \gamma=\rho,
\end{aligned}
\right.
\end{equation*}
and there exists a constant $c$ such that
\begin{equation*}
|\psi(t)-c\psi^+|\leq
Ct^me^{-\gamma t}\hspace{0.6cm}\text{if } \gamma> \rho.
\end{equation*}
Moreover, $c=0$ if  $\rho=0$. 
\end{corollary}

\begin{proof}
Any solution $\psi$ can be written as
$$\psi=c^+\psi^++c^-\psi^-+\psi_*,$$
for some constants $c^+, c^-$ 
and the function $\psi_*$ constructed in Lemma \ref{lemma-decay-particular-solution-ODE}.
Since $\psi\rightarrow 0$ as $t\rightarrow \infty$, we must have $c^-=0$, 
and $c^+=0$ if $\rho=0$. Thus, the desired conclusion 
follows from Lemma \ref{lemma-decay-particular-solution-ODE}.
\end{proof}


We are interested in two cases, $e^{\rho t}\psi^+$ and $e^{-\tau t}\psi^-$ are either periodic 
or in the space $\mathcal E$  introduced in Definition \ref{def-space-E}. In the latter case, 
we can integrate term by term. For the former case, we need the following simple calculus lemma. 
All periodic functions involved here have the same period. 

\begin{lemma}\label{lemma-calculus-integral} Let $p(t)$ be a smooth periodic function on $\mathbb R$. 
Then, for any nonnegative integer $m$ and any positive constant $\alpha$, 
\begin{align}\label{eq-identity-monomial}
\int_0^ts^mp(s)ds=at^{m+1}+\sum_{i=0}^mt^ip_i(t),\end{align}
and 
\begin{align}
\label{eq-identity-positive-exp}
\int_0^ts^me^{\alpha s}p(s)ds&=\sum_{i=0}^mt^ie^{\alpha t} p_i(t)+a,\\
\label{eq-identity-negative-exp}
\int_{t}^\infty s^me^{-\alpha s}p(s)ds&=\sum_{i=0}^mt^ie^{-\alpha t} p_i(t),
\end{align}
where $a$ is a constant, and $p_0, p_1, \cdots, p_m$ (different from line to line)
are smooth periodic functions on $\mathbb R$.
\end{lemma}

The proof is elementary and hence omitted.

\begin{remark}\label{remark-ODE-period}  
We now strengthen Assumption \ref{assum-kernel-linear-equation-ODE}. 
Assume that $e^{-\rho t}\psi^+(t), e^{\tau t}\psi^-(t)$ in (i) are periodic, 
or $\psi^+, \eta$ in (ii) are periodic, 
and 
$$f(t)=t^me^{-\gamma t}q(t),$$
for some periodic function $q$, 
some positive constant $\gamma$, and some nonnegative  integer $m$. Then, the particular solution 
$\psi_*$ in Lemma \ref{lemma-decay-particular-solution-ODE} has the form, for $\gamma\neq \rho$,  
$$\psi_*(t)=\sum_{i=0}^mt^ie^{-\gamma t} r_i(t),$$
and for $\gamma= \rho$,  
$$\psi_*(t)=\sum_{i=0}^{m+1}t^ie^{-\gamma t} r_i(t),$$
where $r_0, \cdots, r_{m+1}$ are periodic functions. 
\end{remark}

We now turn to linear elliptic equations on 
$\mathbb R_+\times \mathbb S^{n-1}$. 
Denote by $\theta$ the local coordinates on $\mathbb S^{n-1}$. 
Let $a, b, c$ be smooth functions on $\mathbb R$, with bounded derivatives of arbitrary orders. 
Consider the linear operator 
$\mathcal L$  given by 
\begin{equation}\label{eq-linear-operator-form-PDE}
\mathcal L\varphi=\varphi_{tt}+a(t)\Delta_\theta\varphi+b(t)\varphi_t
+c(t)\varphi.\end{equation}
The operator $\mathcal L$ reduces to the standard Laplacian on $\mathbb R\times \mathbb S^{n-1}$ if 
$a=1$ and $b=c=0$. 
We will discuss the linear equation 
\begin{equation}\label{eq-linearization-PDE}
\mathcal L\varphi=f\quad\text{on }[t_0, \infty)\times \mathbb S^{n-1},
\end{equation}
for some $t_0>0$. In the following, we assume 
\begin{equation}\label{eq-lower-bound-a}a\ge a_0,\end{equation}
for some positive constant $a_0$. 

We now project the operator $\mathcal L$ to spherical harmonics. 

Let $\{\lambda_i\}$ be the sequence of eigenvalues of $-\Delta_{\theta}$ on $\mathbb S^{n-1}$, 
arranged in an increasing order
with $\lambda_i\to\infty$ as $i\to\infty$, and let  $\{X_i\}$ be 
a sequence of the corresponding normalized eigenfunctions of 
$-\Delta_\theta$ on $L^2(\mathbb S^{n-1})$; namely, for each $i\ge 0$, 
\begin{equation}\label{eq-Laplacian-eigenpair}
-\Delta_{\theta}X_i=\lambda_iX_i.\end{equation}
Here, the multiplicity is considered. Hence, 
$$\lambda_0=0, \,\, \lambda_1=\cdots=\lambda_n=n-1, \,\,\lambda_{n+1}=2n, \,\,\cdots.$$
Note that each $X_i$ is a spherical harmonic of certain degree. In the following, 
we fix such a sequence $\{X_i\}$. 

For a fixed $i\ge 0$ and any $\psi\in C^2(\mathbb R)$, 
we can write 
\begin{equation*}
\mathcal L(\psi X_i)=(L_i\psi) X_i,
\end{equation*}
where $L_i$ is given by  
\begin{equation}\label{eq-linearization-j-ODE}
L_i\psi=\psi_{tt}+b(t)\psi_t+(c(t)-a(t)\lambda_i)\psi.
\end{equation}
We point out that $b$ is a fixed function, independent of $i$. 

We now examine $\mathrm{Ker}(L_i)$ for a special case $a=1$ and $b=c=0$, where $L_i$ is given by 
$$L_i\psi=\psi_{tt}-\lambda_i\psi.$$
For $i=0$, there is a pair of linearly independent elements in $\mathrm{Ker}(L_i)$
given by $1$ and $t$,
one of which is bounded on $\mathbb R_+$ and another is unbounded  on $\mathbb R_+$. 
For $i\geq 1$, there is a pair of linearly independent elements in $\mathrm{Ker}(L_i)$
given by $e^{-\sqrt{\lambda_i}t}$ and $e^{\sqrt{\lambda_i}t}$,
one of which decays exponentially on $\mathbb R_+$ and another grows exponentially  on $\mathbb R_+$.

We introduce an assumption for the general $\mathcal L$ with a similar property. 

\begin{assumption}\label{assum-kernel-linear-equation} Let $L_i$ be the operator given by 
\eqref{eq-linearization-j-ODE}. 
For each $i\ge 0$, the space $\mathrm{Ker}(L_i)$ has a basis $\psi_i^+$ and $\psi_i^-$ with the following properties: 

(i)   there exist two sequences of scalars $\{\rho_i\}_{i=0}^\infty$ and $\{\tau_i\}_{i= 0}^\infty$,  
nonnegative, increasing, and divergent to $\infty$, with $\tau_0=0$ and, for each $i\ge 1$,
\begin{equation}\label{eq-relation-rho-tau-j-ODE}
\rho_i-\tau_i=\rho_0;
\end{equation}

(ii) if $i\ge 1$ or $i=0$ with $\rho_0>0$, then $e^{\rho_i t}\psi_i^+$ and $e^{-\tau_i t}\psi_i^-$ 
are bounded on $(0,\infty)$, 
and $\psi_i^-$ does not converge to 0 as $t\to\infty$; 
if  $\rho_0=0$, then $\psi_0^+$ is bounded on $(0,\infty)$, $\psi_0^-=\eta+at\psi_0^+$ 
for some constant $a$ and bounded function $\eta$ on  $(0,\infty)$, 
and $\psi_0^+$ does not converge to 0 as $t\to\infty$;

(iii)  the Wronskian $W_0$ of $\{\psi^+_0,\psi^-_0\}$ satisfies 
\begin{equation}\label{eq-Wonskian-ODE-i}c_1\le |W_0(t)|e^{\rho_0t}\le c_2\quad\text{for any }t\in \mathbb R_+,\end{equation}
for some positive constants $c_1$ and $c_2$. 
\end{assumption}

Roughly speaking, for each $i\ge 1$, $\psi_i^+$ decays exponentially at the order $\rho_i$, and $\psi_i^-$ grows 
exponentially at the order $\tau_i$, as $t\to\infty$. For $i=0$, if $\rho_0>0$, then $\psi_0^+$ decays exponentially 
at the order $\rho_0$, as $t\to\infty$, and $\psi_0^-$ is bounded;
if $\rho_0=0$, then $\psi_0^+$ is bounded but not convergent to 0 as $t\to\infty$, and $\psi_0^-$ is unbounded. 

We now make a remark concerning the Wronskian for $i\ge 1$. 

\begin{remark}\label{remark-Wronskian}
Let $\psi_i^+$ and $\psi_i^-$ be the two solutions of $L_i\psi=0$ 
as in Assumption \ref{assum-kernel-linear-equation}, 
and $W_i$ be their Wronskian. Then,
\begin{equation*}
W_i'=-b W_i,
\end{equation*}
where $b$ is the coefficient of $t$-derivative term in \eqref{eq-linearization-j-ODE}, 
which is independent of $i$. Hence, $W_i$ differs from $W_0$ by a nonzero constant factor
for any $i\ge 1$. Thus, $W_i$ satisfies \eqref{eq-Wonskian-ODE-i}, or 
by \eqref{eq-relation-rho-tau-j-ODE}, 
\begin{equation*}c_1\le |W_i(t)|e^{(\rho_i-\tau_i)t}\le c_2\quad\text{for any }t\in \mathbb R_+.\end{equation*}
As a consequence, Lemma \ref{lemma-decay-particular-solution-ODE} and 
Corollary \ref{cor-decay--sol-Ljw=f-ODE} are applicable to $L_i$, for each $i\ge 0$. 
\end{remark}


Next we analyze solutions of \eqref{eq-linearization-PDE} 
on $(t_0,\infty)\times \mathbb S^{n-1}$. 
In the following,  $\{X_i\}$ is a fixed orthonormal basis of $L^2(\mathbb S^{n-1})$, 
formed by eigenfunctions of $-\Delta_\theta$ as in 
\eqref{eq-Laplacian-eigenpair}, and $\{\psi^+\}$ is as in Assumption \ref{assum-kernel-linear-equation}. 

\begin{lemma}\label{lemma-nonhomogeneous-linearized-eq}
Suppose Assumption \ref{assum-kernel-linear-equation} holds. 
Let $\gamma>0$ be a constant, $m\ge 0$ be an integer, and $f$ be a smooth function in 
$[t_0,\infty)\times \mathbb S^{n-1}$ for some $t_0>0$, 
satisfying, for all $(t,\theta)\in [t_0,\infty)\times \mathbb S^{n-1}$, 
$$|f(t, \theta)|\leq C t^me^{-\gamma t}.$$
Let $\varphi$ be a solution of \eqref{eq-linearization-PDE}  in $(t_0,\infty)\times\mathbb S^{n-1}$ 
such that $\varphi(t,\theta)\rightarrow 0$ as $t\rightarrow \infty$ uniformly for $\theta\in \mathbb S^{n-1}$.

$\mathrm{(i)}$ If $\gamma\le \rho_0$, then, 
for any $(t,\theta)\in (t_0, \infty)\times\mathbb S^{n-1}$, 
$$|\varphi(t,\theta)|\le \begin{cases} Ct^me^{-\gamma t} &\text{if }\gamma<\rho_0,\\ 
Ct^{m+1}e^{-\gamma t}&\text{if }\gamma=\rho_0.\end{cases}$$

$\mathrm{(ii)}$ If $\rho_l<\gamma\le \rho_{l+1}$ for some nonnegative integer $l$, 
then, 
for any $(t,\theta)\in (t_0, \infty)\times\mathbb S^{n-1}$, 
$$\Big|\varphi(t,\theta)-\sum_{i=0}^lc_i\psi^+_{i}(t)X_i(\theta)\Big|\le 
\begin{cases} Ct^me^{-\gamma t} &\text{if }\rho_l<\gamma<\rho_{l+1},\\ 
Ct^{m+1}e^{-\gamma t}&\text{if }\gamma=\rho_{l+1},\end{cases}$$
for some constant $c_i$, for 
$i=0,1, \cdots, l$. 

Moreover, if $\rho_0=0$, then $\mathrm{(i)}$ does not appear 
and the summation in $\mathrm{(ii)}$ starts from $i=1$. 
\end{lemma}

\begin{proof}
For each $i\ge 0$, 
set 
$$\varphi_{i}(t)=\int_{\mathbb S^{n-1}}\varphi(t, \theta)X_{i}(\theta)\dif\theta,\quad 
f_{i}(t)=\int_{\mathbb S^{n-1}}f(t, \theta)X_{i}(\theta)\dif\theta.$$
Then, for any 
$t\in (t_0, \infty)$,
$$|f_{i}(t)|\le C_0t^me^{-\gamma t}.$$
By multiplying \eqref{eq-linearization-PDE} by $X_{i}$ and integrating over $\mathbb S^{n-1}$, we obtain 
$$L_i\varphi_{i}= f_{i}\quad\text{on }(t_0,\infty).$$
Since $\varphi(t,\theta)\to 0$ as $t\to\infty$ uniformly for $\theta\in\mathbb S^{n-1}$, 
then $\varphi_{i}(t)\to 0$ as $t\to\infty$, for each $i$. 
We now apply Corollary \ref{cor-decay--sol-Ljw=f-ODE} to $L_i$. 
For $i\ge 0$, we have, 
for any $t\in(t_0,\infty)$,
\begin{equation}\label{eq-U2-03-0c}|\varphi_{i}(t)|\le  \begin{cases} Ct^me^{-\gamma t} &\text{if }\gamma<\rho_i,\\ 
Ct^{m+1}e^{-\gamma t}&\text{if }\gamma=\rho_i,\end{cases}\end{equation} 
and there exists a constant $c_{i}$ such that
\begin{equation}\label{eq-U2-03-0b}
|\varphi_{i}(t)-c_{i}\psi_i^+(t)|\leq Ct^me^{-\gamma t}\quad\text{if } \gamma> \rho_i.
\end{equation}

For some positive integer $l_*$ to be determined, set 
\begin{equation}\label{eq-def-hat-w}\widehat \varphi(t,\theta)=\varphi(t,\theta)
-\sum_{i=0}^{l_*}\varphi_{i}(t)X_{i}(\theta),\end{equation}
and 
$$\widehat f(t,\theta)=f(t,\theta)
-\sum_{i=0}^{l_*}f_{i}(t)X_{i}(\theta).$$
A simple calculation yields 
\begin{equation}\label{eq-U2-03-1}\mathcal L\widehat \varphi=\widehat f.\end{equation}
For simplicity, we write $\widehat \varphi(t)=\widehat \varphi(t,\cdot)$ and $\widehat f(t)=\widehat f(t,\cdot)$.
By multiplying \eqref{eq-U2-03-1} by $\widehat \varphi(t)$ and integrating over $\mathbb S^{n-1}$, we obtain 
\begin{align*}&\int_{\mathbb S^{n-1}}\widehat \varphi_{tt}(t)\widehat \varphi(t)\dif\theta
+b(t)\int_{\mathbb S^{n-1}}\widehat \varphi_{t}(t)\widehat \varphi(t)\dif\theta
+c(t)\int_{\mathbb S^{n-1}}\widehat{\varphi}^2(t)\dif\theta\\
&\qquad-a(t) \int_{\mathbb S^{n-1}}|\nabla_\theta\widehat \varphi(t)|^2\dif\theta
=\int_{\mathbb S^{n-1}}\widehat f(t)\widehat \varphi(t)\dif\theta.
\end{align*}
Set
$$y(t)=\bigg[\int_{\mathbb S^{n-1}}\widehat \varphi^2(t)\dif\theta\bigg]^{1/2}.$$
Then, 
$$y(t)y'(t)=\int_{\mathbb S^{n-1}}\widehat \varphi(t)\widehat \varphi_t(t)\dif\theta,$$
and 
$$y(t)y''(t)+[y'(t)]^2=\int_{\mathbb S^{n-1}}
\big[\widehat \varphi(t)\widehat \varphi_{tt}(t)+\widehat \varphi_t^2(t)\big]\dif\theta.$$
The Cauchy inequality implies that, if  $y(t)>0$, then
$$[y'(t)]^2\le \int_{\mathbb S^{n-1}}\widehat \varphi_t^2(t)\dif\theta,$$
and hence
$$y(t)y''(t)\ge \int_{\mathbb S^{n-1}}\widehat \varphi(t)\widehat \varphi_{tt}(t)\dif\theta.$$
Note that 
$$\int_{\mathbb S^{n-1}}|\nabla_\theta\widehat \varphi(t)|^2\dif\theta\ge \lambda_{l_*+1}
\int_{\mathbb S^{n-1}}\widehat \varphi^2(t)\dif\theta.$$
Therefore, with $a\ge a_0>0$ by \eqref{eq-lower-bound-a}, 
$$y(t)y''(t)+b(t)y(t)t'(t)+(c(t)-a(t)\lambda_{l_*+1})y^2(t)\ge -y(t)\|\widehat f(t)\|_{L^2(\mathbb S^{n-1})}.$$
If $y(t)>0$, we obtain 
$$L_{l_*+1}y\ge -C_0t^me^{-\gamma t}.$$
For $i\ge 0$,  by \eqref{eq-lower-bound-a}, we have
\begin{align*}
c-a\lambda_i \le c-a_0\lambda_i.
\end{align*}
Hence, $c-a\lambda_i$ is negative on $(t_0, \infty)$ for large $i$, since $\lambda_i\to \infty$. Next, 
\begin{align*}L_i(t^me^{-\gamma t})=\left(c-a\lambda_i+\gamma^2-p\gamma
+mpt^{-1}-2m\gamma t^{-1}+m(m-1)t^{-2}\right)t^me^{-\gamma t}.\end{align*}
We take $l_*$ large such that $c-a\lambda_{l_*+1}<0$ and 
\begin{align}\label{eq-super-solution}L_{l_*+1}(t^me^{-\gamma t})
\le -t^me^{-\gamma t}\quad\text{on }(t_0, \infty),\end{align}
and also require $l_*\ge l$ if $\gamma\in (\rho_l, \rho_{l+1}]$, for some nonnegative integer $l$. 
Set 
$$z(t)=Ct^me^{-\gamma t}.$$
Then, for some constant $C$ sufficiently large, we have $z(t_0)\ge y(t_0)$ and, if $y(t)>0$,  
\begin{equation}\label{eq-U2-03-2}L_{l_*+1}(z-y)(t)\le 0.\end{equation}
Note that $y(t)\to 0$ as $t\to\infty$. If $z-y$ is negative somewhere in $(t_0,\infty)$, then the minimum of $z-y$ 
is negative and is attained at some $t_*\in (t_0,\infty)$. Hence, $y(t_*)>z(t_*)>0$ and \eqref{eq-U2-03-2}
holds at $t_*$. However, $(z-y)(t_*)<0$, $(z-y)'(t_*)=0$, 
and $(z-y)''(t_*)\ge0$, contradicting \eqref{eq-U2-03-2}. Therefore, for any $t>t_0$,
$y(t)\le z(t)$, and hence
$$\|\widehat \varphi(t,\cdot)\|_{L^2(\mathbb S^{n-1})}\le Ct^me^{-\gamma t}.$$
For any fixed $(t,\theta)\in (t_0+1, \infty)\times\mathbb S^{n-1}$, by applying the interior $L^\infty$-estimate 
to the equation  \eqref{eq-U2-03-1} in $(t-1,t+1)\times\mathbb S^{n-1}$, we obtain 
\begin{equation}\label{eq-U2-03-0d}|\widehat \varphi(t,\theta)|\le 
C\big\{\|\widehat \varphi\|_{L^2((t-1,t+1)\times\mathbb S^{n-1})}
+\|\widehat f\|_{L^2((t-1,t+1)\times\mathbb S^{n-1})}\big\}\le Ct^me^{-\gamma t}.\end{equation}

First, we consider the case $0<\gamma\le \rho_0$. By the definition of $\widehat \varphi$ in \eqref{eq-def-hat-w},
we write 
$$\varphi(t,\theta)=\widehat \varphi(t,\theta)+\sum_{i=0}^{l_*}\varphi_{i}(t)X_{i}(\theta),$$
By combining \eqref{eq-U2-03-0c} and \eqref{eq-U2-03-0d}, 
we obtain the desired result. 

Next, we consider the case $\rho_l<\gamma\le \rho_{l+1}$ for some nonnegative integer $l$. 
We write
\begin{align}\label{eq-expansion-order-l}\begin{split}
\varphi(t,\theta)&=\sum_{i=0}^lc_{i}\psi_i^+(t)X_{i}(\theta)
+\sum_{i=0}^l\big(\varphi_{i}(t)-c_{i}\psi_i^+(t)\big)X_{i}(\theta)  \\
&\qquad
+\sum_{i=l+1}^{l_*}\varphi_{i}(t)X_{i}(\theta)+\widehat \varphi(t,\theta).\end{split}\end{align}
By combining \eqref{eq-U2-03-0c}, \eqref{eq-U2-03-0b}, and \eqref{eq-U2-03-0d}, 
we obtain the desired result. 

Last, we consider the case $\rho_0=0$.  First,  (i) does not appear since $\gamma$ is assumed 
to be positive. Next, in case (ii) the summation starts from $i=1$ since $\psi_0^+$ does not converge to 
0 as $t\to\infty$. 
\end{proof}

We now make a remark concerning the proof. In order to expand $\varphi$ up to order $l$, 
we may need to expand up to a much higher order  $l_*$ to ensure that $t^me^{-\gamma t}$ is a supersolution as in 
\eqref{eq-super-solution}. In the expansion of $\varphi$ up to order $l_*$, the part up to order $l$ has a nice form 
as shown in \eqref{eq-expansion-order-l}.

\section{Spherical Harmonics}\label{sec-appendix1} 

In this section, we discuss expansions of functions on $I\times S^{m}$ in terms of spherical harmonics on $S^{m}$, 
where $I$ is an interval in $\mathbb R$, with $m=n-1$.  
Recall the space $\s_d$ introduced in Definition \ref{def-expansion-spherical}. 
Our primary goal is to prove that
for some smooth function $\varphi\in \s_d$ 
certain combinations of derivatives of $\varphi$ are in $\s_{\tilde d}$ for some $\tilde d$. 

We first recall the following result. 
If $\varphi\in \s_p$ and $\psi\in\s_q$ for some $p$ and $q$, then $\varphi\psi\in \s_{p+q}$. 
This follows from Lemma \ref{lemma-SphericalHarmonics}. 

In the following, we denote by $\nabla_{\theta}$ the induced connection on  $S^{m}$.

\begin{lemma}\label{lemma-gradient-spherical} 
$\mathrm{(i)}$ If $\varphi\in \s_p$ for some $p$, then $\Delta_\theta \varphi\in \s_{p}$. 

$\mathrm{(ii)}$ If $\varphi\in \s_p$ and $\psi\in\s_q$ for some $p$ and $q$, then 
$\langle \nabla_\theta\varphi, \nabla_\theta\psi\rangle \in \s_{p+q}$. 
\end{lemma}

\begin{proof} (i) If $X$ is a spherical harmonic of degree $i$, then $\Delta_\theta X=-\lambda_iX$, 
where $\lambda_i$ is an eigenvalue of $-\Delta_\theta$ on $S^{m}$.

(ii) We note that 
$$\Delta_\theta(\varphi\psi)=\varphi\Delta_\theta \psi+\psi\Delta_\theta \varphi+2\langle \nabla_\theta \varphi, \nabla_\theta \psi\rangle.$$
This implies the desired result with the help of (i) and Lemma \ref{lemma-SphericalHarmonics}. 
\end{proof} 

Now we start to study more complicated expressions of derivatives. 
Let $\varphi^1,\cdots, \varphi^k$ be $k$ smooth functions on $I\times S^{m}$. Define $\nabla_{\theta}^2\varphi^1\circ\dots\circ\nabla_{\theta}^2\varphi^k\in\Gamma(T^2(I\times S^{m}))$ by contracting from the 
second to the $(2k-1)^{\rm{th}}$ indices. 
That is, if we we choose normal coordinates $\{t,\theta^2,\cdots,\theta^n\}$ at $x\in I\times S^{m}$, then
\begin{equation*}
    \nabla_\theta^2\varphi^1\circ\dots\circ\nabla_\theta^2\varphi^k(p)
    =\varphi^1_{i_1i_2}\varphi^2_{i_2i_3}\cdots\varphi^k_{i_{k-1}i_k}d\theta^{i_1}\otimes d\theta^{i_k},
\end{equation*}
where the lower indices of $\varphi^l_{ij}$ are covariant derivatives and we take summation whenever the indices are repeated. 
If $A^p=(g^{ij}\varphi^p_{jl})$ is the symmetric matrix associated with $\nabla_\theta^2\varphi^p$, 
then the matrix associated with $\nabla_\theta^2\varphi^1\circ\dots\circ\nabla_\theta^2\varphi^k$ is 
the product $A^1A^2\cdots A^k$. 

In the following, we always choose normal coordinates $\{t,\theta^2,\cdots,\theta^n\}$ 
at arbitrarily given point $x\in I\times S^{m}$ for computation. 
Denote by $\mathfrak{S}_k$ the collection of all permutations of $\{1, \cdots, k\}$. 

Set 
\begin{equation}\label{eq-Hk-no-permuation}
H_k(\varphi^1, \cdots, \varphi^k)=\mathrm{tr}(\nabla_{\theta}^2\varphi^1 \circ \dots \circ \nabla_{\theta}^2\varphi^k).\end{equation} 
The function $H_k$ in \eqref{eq-Hk-no-permuation} is symmetric for $k=2, 3$. 
However, the symmetry breaks down for $k\ge 4$. 
For example, there may be 3 different values of $H_4(\varphi^{\tau(1)}, \cdots, \varphi^{\tau(4)})$ 
if $\tau$ runs over the set of 24 elements of $\mathfrak{S}_4$. 

One primary objective in this section is to prove 
$H_k(\varphi, \cdots, \varphi)=\mathrm{tr}((\nabla_{\theta}^2\varphi)^k)\in \s_{kd}$ if $\varphi\in \s_d$. 
We intend to employ the induction on $k$ to prove this. The proof of Lemma \ref{lemma-formula-trace} 
below shows that even if we start with identical functions $\varphi^1=\cdots=\varphi^k=\varphi$, 
a different function $|\nabla_\theta \varphi|^2$ will emerge in the step from $H_k$ to $H_{k-1}$ 
in the term $H_{k-1}(|\nabla_\theta\varphi|^2, \varphi, \cdots, \varphi)$. In the process in reducing $k$, 
more different functions will appear. When the symmetry breaks down, we cannot continue. 
For a remedy,  we need to consider all possible combinations of
$H_k(\varphi^{\tau(1)}, \cdots, \varphi^{\tau(k)})$ and let $\tau$ run over all 
elements of $\mathfrak{S}_k$.


Set, for $k\ge 1$, 
\begin{align}\label{eq-definitionH-k}
\mathcal H_k(\varphi^1,\cdots,\varphi^k)&=\sum_{\tau\in\mathfrak{S}_k}\mathrm{tr}(\nabla_{\theta}^2\varphi^{\tau(1)} \circ \cdots \circ \nabla_{\theta}^2\varphi^{\tau(k)}),
\end{align}
and, for $k\ge 2$, 
\begin{align}
\label{eq-Q-k}
\mathcal Q_k(\varphi^1, \cdots, \varphi^{k})&=  \sum_{\tau\in\mathfrak{S}_{k}}\nabla^2_{\theta}\varphi^{\tau(1)}\circ\cdots\circ\nabla^{2}_{\theta}\varphi^{\tau(k-2)}
   \big(\nabla_{\theta}\varphi^{\tau(k-1)},\nabla_{\theta}\varphi^{\tau(k)}\big).
   \end{align}
Note $\mathcal{H}_1(\varphi^1)=\Delta_\theta \varphi^1$ and 
$\mathcal Q_2(\varphi^1, \varphi^2)=2\langle\nabla_\theta\varphi^1, \nabla_\theta\varphi^2\rangle$. 
For convenience, we set $\mathcal{H}_0=\mathrm{tr}(I_{m})=m$. We emphasize that no $t$-derivatives are involved.

\smallskip

We first study the function $\mathcal H_k$. 

\begin{lemma}\label{lemma-formula-trace}
For $k\geq 2$, let $\varphi^1, \cdots, \varphi^k$ be smooth functions on $I\times S^{m}$. Then, 
\begin{equation*}
\begin{split}
\mathcal H_k(\varphi^1,\cdots,\varphi^k)
    &=\frac{1}{2}\sum_{p\neq q}\mathcal H_{k-1}(\langle\nabla_\theta\varphi^p, \nabla_\theta\varphi^q\rangle, \varphi^1, \cdots,
    \widehat{\varphi^p}, \cdots, \widehat{\varphi^q}, \cdots, \varphi^k)\\
    &\qquad-\frac{1}{k-1}\sum_{p=1}^{k}\langle\nabla_\theta \varphi^{p},\nabla_\theta \mathcal H_{k-1}(\varphi^1, \cdots
    \widehat{\varphi^p}, \cdots, \varphi^k)\rangle\\
    &\qquad-\sum_{p\neq q}\langle \nabla_{\theta} \varphi^{p},\nabla_{\theta} \varphi^{q}\rangle \mathcal{H}_{k-2}(\varphi^1, \cdots,
    \widehat{\varphi^p}, \cdots, \widehat{\varphi^q}, \cdots, \varphi^k)\\
    &\qquad+\mathcal{Q}_{k}(\varphi^1, \cdots, \varphi^k).
\end{split}
\end{equation*}
\end{lemma}

Here and hereafter, $\widehat{\varphi^p}$ means $\varphi^p$ is deleted. 

\begin{proof} Let $H_k$ be defined by \eqref{eq-Hk-no-permuation}. 
We will prove 
\begin{equation*}\label{eq-trace-power of hessians}
\begin{split}
    \sum_{\tau\in\mathfrak{S}_k}H_k(\varphi^{\tau(1)},\cdots,\varphi^{\tau(k)})
    &=\sum_{\tau\in\mathfrak{S}_k}\Big\{\frac{1}{2}H_{k-1}
    (\langle\nabla_\theta \varphi^{\tau(1)},\nabla_\theta \varphi^{\tau(2)}\rangle,\varphi^{\tau(3)},\cdots,\varphi^{\tau(k)})\\
    &\qquad-\frac{1}{k-1}\langle\nabla_\theta \varphi^{\tau(1)},\nabla_\theta H_{k-1}(\varphi^{\tau(2)},\cdots,\varphi^{\tau(k)})\rangle\\
    &\qquad-\langle \nabla_{\theta} \varphi^{\tau(1)},\nabla_{\theta} \varphi^{\tau(2)}\rangle H_{k-2}(\varphi^{\tau(3)},\cdots,\varphi^{\tau(k)})\\
    &\qquad+\nabla^2_{\theta}\varphi^{\tau(3)}\circ\cdots\circ\nabla^2_{\theta}\varphi^{\tau(k)}(\nabla_{\theta}  \varphi^{\tau(1)},\nabla_{\theta} \varphi^{\tau(2)})\Big\}.
\end{split}
\end{equation*}
Note
$$H_k(\varphi^1,\cdots,\varphi^k)=\varphi^1_{i_1i_2}\cdots \varphi^k_{i_ki_1}.$$
By 
$$\Big(\varphi^{\tau(1)}_j\varphi^{\tau(2)}_j\Big)_{i_1i_2}=\varphi^{\tau(1)}_{ji_1i_2}\varphi^{\tau(2)}_j+\varphi^{\tau(1)}_{ji_1}\varphi^{\tau(2)}_{ji_2}+\varphi^{\tau(1)}_{j}\varphi^{\tau(2)}_{ji_1i_2}+\varphi^{\tau(1)}_{ji_2}\varphi^{\tau(2)}_{ji_1},$$
we get
\begin{equation*}
\begin{split}
    \sum_{\tau\in\mathfrak{S}_k}&H_{k-1}(\langle\nabla_\theta \varphi^{\tau(1)},\nabla_\theta \varphi^{\tau(2)}\rangle,\varphi^{\tau(3)},\cdots,\varphi^{\tau(k)})\\
    &=\sum_{\tau\in\mathfrak{S}_k}\Big(\varphi^{\tau(1)}_j\varphi^{\tau(2)}_j\Big)_{i_1i_2}
    \varphi^{\tau(3)}_{i_2i_3}\cdots \varphi^{\tau(k)}_{i_{k-1}i_1}\\
    &=2\sum_{\tau\in\mathfrak{S}_k}\varphi^{\tau(1)}_{ji_1}\varphi^{\tau(2)}_{ji_2}\varphi^{\tau(3)}_{i_2i_3}\cdots \varphi^{\tau(k)}_{i_{k-1}i_1}
    +2\sum_{\tau\in\mathfrak{S}_k}\varphi^{\tau(1)}_j\varphi^{\tau(2)}_{ji_1i_2}\varphi^{\tau(3)}_{i_2i_3}\cdots \varphi^{\tau(k)}_{i_{k-1}i_1}\\
     &=2\sum_{\tau\in\mathfrak{S}_k}H_k(\varphi^{\tau(1)},\cdots,\varphi^{\tau(k)})
     +2\sum_{\tau\in\mathfrak{S}_k}\varphi^{\tau(1)}_j\varphi^{\tau(2)}_{ji_1i_2}\varphi^{\tau(3)}_{i_2i_3}\cdots \varphi^{\tau(k)}_{i_{k-1}i_1}.
\end{split}
\end{equation*}
For the second term, we first apply the Ricci identity. Since $S^{m}$ has constant curvature $1$, we have
\begin{align*}
    \varphi^{\tau(2)}_{ji_1i_2}&=\varphi^{\tau(2)}_{i_1i_2j}+R_{i_2j l i_1}\varphi^{\tau(2)}_{l}
    =\varphi^{\tau(2)}_{i_1i_2j}+(\delta_{i_2i_1}\delta_{jl}-\delta_{i_2l}\delta_{ji_1})\varphi^{\tau(2)}_l\\
    &=\varphi^{\tau(2)}_{i_1i_2j}+\delta_{i_1i_2}\varphi^{\tau(2)}_{j}-\delta_{ji_1}\varphi^{\tau(2)}_{i_2}.
\end{align*}
Corresponding to the three terms above for $\varphi^{\tau(2)}_{ji_1i_2}$, we write 
$$\sum_{\tau\in\mathfrak{S}_k}\varphi^{\tau(1)}_j\varphi^{\tau(2)}_{ji_1i_2}\varphi^{\tau(3)}_{i_2i_3}\cdots \varphi^{\tau(k)}_{i_{k-1}i_1}=
I_1+I_2+I_3,$$ 
where
\begin{equation*}
\begin{split}
I_1&=\sum_{\tau\in\mathfrak{S}_k}\varphi^{\tau(1)}_j \varphi^{\tau(2)}_{i_1i_2j}\varphi^{\tau(3)}_{i_2i_3}\cdots \varphi^{\tau(k)}_{i_{k-1}i_1},\\
I_2&=\sum_{\tau\in\mathfrak{S}_k}\varphi^{\tau(1)}_{j}\varphi^{\tau(2)}_{j}\varphi^{\tau(3)}_{i_2i_3}\cdots \varphi^{\tau(k)}_{i_{k-1}i_2},\\
I_3&=-\sum_{\tau\in\mathfrak{S}_k}\varphi^{\tau(1)}_{i_1}\varphi^{\tau(2)}_{i_2}\varphi^{\tau(3)}_{i_2i_3}\cdots \varphi^{\tau(k)}_{i_{k-1}i_1}.
\end{split}
\end{equation*}
For $I_1$, we have 
\begin{equation*}
\begin{split}
    I_1
    &=\sum_{l=1}^{k}\sum_{\substack{\tau\in\mathfrak{S}_{k}\\ \tau(1)=l}}\varphi^{l}_j\varphi^{\tau(2)}_{i_1i_2j}\varphi^{\tau(3)}_{i_2i_3}\cdots \varphi^{\tau(k)}_{i_{k-1}i_1}\\
    &=\frac{1}{k-1}\sum_{l=1}^{k}\sum_{\substack{\tau\in\mathfrak{S}_{k}\\ \tau(1)=l}}\varphi^{l}_j(\varphi^{\tau(2)}_{i_1i_2}\varphi^{\tau(3)}_{i_2i_3}\cdots \varphi^{\tau(k)}_{i_{k-1}i_1})_j\\
    &=\frac{1}{k-1}\sum_{\tau\in\mathfrak{S}_{k}}\Big\langle\nabla_\theta \varphi^{\tau(1)},\nabla_\theta H_{k-1}(\varphi^{\tau(2)},\cdots, \varphi^{\tau(k)})\Big\rangle.
\end{split}
\end{equation*}
For $I_2$ and $I_3$, we simply note
\begin{equation*}
I_2=
    \sum_{\tau\in\mathfrak{S}_k}\langle \nabla_{\theta} \varphi^{\tau(1)},\nabla_{\theta} \varphi^{\tau(2)}\rangle H_{k-2}(\varphi^{\tau(3)},\cdots,\varphi^{\tau(k)}),
\end{equation*}
and 
\begin{equation*}
I_3
    =-\sum_{\tau\in\mathfrak{S}_k}\nabla^2_{\theta}\varphi^{\tau(3)}\circ\cdots\circ\nabla^2_{\theta}\varphi^{\tau(k)}(\nabla_{\theta}  \varphi^{\tau(1)},\nabla_{\theta} \varphi^{\tau(2)}).
\end{equation*}
We get the desired result by a simple substitution. 
\end{proof}

Lemma \ref{lemma-formula-trace} demonstrates that $\mathcal H_k$ can be expressed in terms of 
$\mathcal H_{k-1}$, $\mathcal H_{k-2}$, and $\mathcal Q_k$. We next study $\mathcal Q_k$. 

\begin{lemma}\label{lemma-Q-k}
For $k\geq 2$, assume $\varphi^i\in\s_{d_i}$ for each $i=1, \cdots, k$. Then, 
$\mathcal Q_k(\varphi^1, \cdots, \varphi^{k})\in \s_{d_1+\cdots+d_k}$. 
\end{lemma}

\begin{proof}
We prove by induction on $k$. We consider $k=2,3$ separately. 
Note
$$\mathcal{Q}_2(\varphi^1, \varphi^2)=2\langle\nabla_{\theta}\varphi^1,\nabla_{\theta}\varphi^2\rangle,$$
and
\begin{equation*}
\begin{split}
    \mathcal{Q}_3(\varphi^1,\varphi^2,\varphi^3)&=\sum_{\tau\in\mathfrak{S}_3}\varphi^{\tau(1)}_{i_1i_2}\varphi^{\tau(2)}_{i_2}\varphi^{\tau(3)}_{i_1}=\frac{1}{2}\sum_{\tau\in\mathfrak{S}_3}(\varphi^{\tau(1)}_{i_2}\varphi^{\tau(2)}_{i_2})_{i_1}\varphi^{\tau(3)}_{i_1}\\
    &=\frac{1}{2}\sum_{\tau\in\mathfrak{S}_3}\langle\nabla_{\theta}\langle\nabla_{\theta}\varphi^{\tau(1)},\nabla_{\theta}\varphi^{\tau(2)}\rangle,\nabla_{\theta}\varphi^{\tau(3)}\rangle.
\end{split}
\end{equation*}
Hence, $\mathcal{Q}_2(\varphi^1, \varphi^2)\in \s_{d_1+d_2}$ and $\mathcal{Q}_3(\varphi^1, \varphi^2, \varphi^3)\in \s_{d_1+d_2+d_3}$ by Lemma \ref{lemma-gradient-spherical}(ii).

We next consider $k\ge 4$. Assume the desired result holds for any collection of $l$ such functions, for any $l\leq k-1$. 
Let $\varphi^1, \cdots, \varphi^{k}$ be functions as given and denote by $\mathcal{Q}_k$ the quantity in \eqref{eq-Q-k} for simplicity, i.e., $\mathcal{Q}_k=\mathcal{Q}_k(\varphi^1,\cdots,\varphi^k)$.
Consider the following quantity:
\begin{align*}\mathcal R&=\sum_{q=1}^{k}\langle \nabla_\theta\mathcal{Q}_{k-1}(\varphi^1, \cdots, \widehat {\varphi^q}, \cdots, \varphi^{k}),\nabla_\theta\varphi^q\rangle\\
&=\sum_{\tau\in\mathfrak{S}_{k}} \Big\langle\nabla_\theta\big(\nabla_\theta^2\varphi^{\tau(1)}
\circ\cdots\circ\nabla_\theta^{2}\varphi^{\tau(k-3)}(\nabla_\theta\varphi^{\tau(k-2)},\nabla_\theta\varphi^{\tau(k-1)})\big),\nabla_\theta\varphi^{\tau(k)}\Big\rangle.\end{align*}
Then, $\mathcal R\in \s_{d_1+\cdots+d_k}$ by the induction hypothesis. By 
\begin{equation*}\begin{split}
\Big(\prod_{p=1}^{k-3}\varphi^{\tau(p)}_{i_pi_{p+1}}\varphi^{\tau(k-2)}_{i_{k-2}}\varphi^{\tau(k-1)}_{i_1}\Big)_j
&=\sum_{l=1}^{k-3}\varphi^{\tau(l)}_{i_li_{l+1}j}\prod_{\substack{p=1\\p\neq l}}^{k-3}\varphi^{\tau(p)}_{i_pi_{p+1}}\varphi^{\tau(k-2)}_{i_{k-2}}\varphi^{\tau(k-1)}_{i_1}\\
&\quad+\prod_{p=1}^{k-3}\varphi^{\tau(p)}_{i_pi_{p+1}}\varphi^{\tau(k-2)}_{i_{k-2}j}\varphi^{\tau(k-1)}_{i_1}
+\prod_{p=1}^{k-3}\varphi^{\tau(p)}_{i_pi_{p+1}}\varphi^{\tau(k-2)}_{i_{k-2}}\varphi^{\tau(k-1)}_{i_1j},
\end{split}
\end{equation*}
we have
\begin{equation*}
\begin{split}
\mathcal R &=\sum_{\tau\in\mathfrak{S}_{k}}\Big(\prod_{p=1}^{k-3}\varphi^{\tau(p)}_{i_pi_{p+1}}\varphi^{\tau(k-2)}_{i_{k-2}}\varphi^{\tau(k-1)}_{i_1}\Big)_j\varphi^{\tau(k)}_j\\
    &=\sum_{\tau\in\mathfrak{S}_{k}}\sum_{l=1}^{k-3}\varphi^{\tau(l)}_{i_li_{l+1}j}\prod_{\substack{p=1\\p\neq l}}^{k-3}\varphi^{\tau(p)}_{i_pi_{p+1}}\varphi^{\tau(k-2)}_{i_{k-2}}\varphi^{\tau(k-1)}_{i_1}\varphi^{\tau(k)}_j
    +2\mathcal{Q}_k.
\end{split}
\end{equation*}
As in the proof of Lemma \ref{lemma-formula-trace}, we get 
$$\varphi^{\tau(l)}_{i_li_{l+1}j}=\varphi^{\tau(l)}_{ji_li_{l+1}}+\delta_{ji_l}\varphi^{\tau(l)}_{i_{l+1}}-\delta_{i_li_{l+1}}\varphi^{\tau(l)}_{j},$$ 
and write accordingly 
\begin{equation*}
\mathcal R =\mathcal{I}_1+\mathcal{I}_2+\mathcal{I}_3
    +2\mathcal{Q}_k,
\end{equation*}
where 
\begin{equation*}
\begin{split}
    \mathcal{I}_1&=\sum_{\tau\in\mathfrak{S}_{k}}\sum_{l=1}^{k-3}\varphi^{\tau(l)}_{ji_li_{l+1}}\prod_{\substack{p=1\\p\neq l}}^{k-3}\varphi^{\tau(p)}_{i_pi_{p+1}}\varphi^{\tau(k-2)}_{i_{k-2}}\varphi^{\tau(k-1)}_{i_1}\varphi^{\tau(k)}_j,\\
    \mathcal{I}_2&=\sum_{l=1}^{k-3}\sum_{\tau\in\mathfrak{S}_{k}}\delta_{ji_l}\varphi^{\tau(l)}_{i_{l+1}}\prod_{\substack{p=1\\p\neq l}}^{k-3}\varphi^{\tau(p)}_{i_pi_{p+1}}\varphi^{\tau(k-2)}_{i_{k-2}}\varphi^{\tau(k-1)}_{i_1}\varphi^{\tau(k)}_j,\\
    \mathcal{I}_3&=-\sum_{l=1}^{k-3}\sum_{\tau\in\mathfrak{S}_{k}}\delta_{i_li_{l+1}}\varphi^{\tau(l)}_{j}\prod_{\substack{p=1\\p\neq l}}^{k-3}\varphi^{\tau(p)}_{i_pi_{p+1}}\varphi^{\tau(k-2)}_{i_{k-2}}\varphi^{\tau(k-1)}_{i_1}\varphi^{\tau(k)}_j.
\end{split}
\end{equation*}
For $\mathcal{I}_1$, we first have
\begin{equation*}
\begin{split}
    \mathcal{I}_1
    &=\frac{1}{2}\sum_{\tau\in\mathfrak{S}_{k}}\sum_{l=1}^{k-3}(\varphi^{\tau(l)}_{j}\varphi^{\tau(k)}_j)_{i_li_{l+1}}\prod_{\substack{p=1\\p\neq l}}^{k-3}\varphi^{\tau(p)}_{i_pi_{p+1}}\varphi^{\tau(k-2)}_{i_{k-2}}\varphi^{\tau(k-1)}_{i_1}\\
    &\qquad-\sum_{\tau\in\mathfrak{S}_{k}}\sum_{l=1}^{k-3}\varphi^{\tau(l)}_{ji_l}\varphi^{\tau(k)}_{ji_{l+1}}\prod_{\substack{p=1\\p\neq l}}^{k-3}\varphi^{\tau(p)}_{i_pi_{p+1}}\varphi^{\tau(k-2)}_{i_{k-2}}\varphi^{\tau(k-1)}_{i_1}.
\end{split}
\end{equation*}    
Then, 
\begin{equation*}
\begin{split}
    \mathcal{I}_1    &=\frac{1}{2}\Big(\sum_{\tau\in\mathfrak{S}_{k}}\sum_{l=1}^{k-3}(\varphi^{\tau(l)}_{j}\varphi^{\tau(k)}_j)_{i_li_{l+1}}\prod_{\substack{p=1\\p\neq l}}^{k-3}\varphi^{\tau(p)}_{i_pi_{p+1}}\varphi^{\tau(k-2)}_{i_{k-2}}\varphi^{\tau(k-1)}_{i_1}\\
    &\quad+\frac{2}{k-3}\sum_{\tau\in\mathfrak{S}_{k}}\sum_{l=1}^{k-3}\varphi^{\tau(k-2)}_{i_li_{l+1}}\prod_{\substack{p=1\\p\neq l}}^{k-3}\varphi^{\tau(p)}_{i_pi_{p+1}}(\varphi^{\tau(l)}_{j}\varphi^{\tau(k)}_j)_{i_{k-2}}\varphi^{\tau(k-1)}_{i_1}\Big)-2\mathcal{Q}_{k}-(k-3)\mathcal{Q}_k\\
    &=\sum_{\substack{p,q=1\\p\neq q}}^{k}\mathcal{Q}_{k-1}(\langle\nabla_{\theta}\varphi^p,\nabla_{\theta}\varphi^q\rangle, \varphi^1, \cdots, \widehat{\varphi^p}, \cdots, \widehat{\varphi^q}, \cdots, \varphi^k)-(k-1)\mathcal{Q}_k.
\end{split}
\end{equation*}
For $\mathcal{I}_2$ and $\mathcal{I}_3$, we simply have 
\begin{equation*}
\begin{split}
\mathcal{I}_2 
&=\sum_{l=1}^{k-3}\sum_{\tau\in\mathfrak{S}_{k}}(\prod_{p=1}^{l-1}\varphi^{\tau(p)}_{i_pi_{p+1}}\varphi^{\tau(k-1)}_{i_1}\varphi^{\tau(k)}_{i_l})(\prod_{p=l+1}^{k-3}\varphi^{\tau(p)}_{i_pi_{p+1}}\varphi^{\tau(l)}_{i_{l+1}}\varphi^{\tau(k-1)}_{i_{k-1}})\\
    &=\sum_{l=1}^{k-3}c_{k,l}\sum_{\tau\in\mathfrak{S}_{k}}\mathcal{Q}_{l+1}(\varphi^{\tau(1)}, \cdots, \varphi^{\tau(l+1)})\mathcal{Q}_{k-l-1}(\varphi^{\tau(l+2)}, \cdots, \varphi^{\tau(k)}),
\end{split}
\end{equation*}
for some constant $c_{k,l}$, and 
\begin{equation*}
\begin{split}
\mathcal{I}_3
    &=-\sum_{l=1}^{k-3}\sum_{\tau\in\mathfrak{S}_{k}}\varphi^{\tau(1)}_{i_1i_{2}}\cdots\varphi^{\tau(l-1)}_{i_{l-1}i_{l+1}}\varphi^{\tau(l+1)}_{i_{l+1}i_{l+2}}\cdots\varphi^{\tau(k-1)}_{i_{k-1}i_{k}}\varphi^{\tau(k-1)}_{i_{k-1}}\varphi^{\tau(k-1)}_{i_1}\varphi^{\tau(k)}_j\varphi^{\tau(l)}_j\\
    &=-(k-3)\sum_{p\neq q}\langle\nabla_\theta\varphi^{p},\nabla_\theta\varphi^{q}\rangle\mathcal{Q}_{k-2}(\varphi^1, \cdots, \widehat{\varphi^p},\cdots, \widehat{\varphi^q}, \cdots, \varphi^k).
\end{split}
\end{equation*}
By a simple substitution, we have, for $k\ge 4$,  
\begin{equation}\label{eq-expression-Qk}
\begin{split}
(k-3)\mathcal Q_k&=\sum_{\substack{p,q=1\\p\neq q}}^{k}\mathcal{Q}_{k-1}(\langle\nabla_{\theta}\varphi^p,\nabla_{\theta}\varphi^q\rangle, \varphi^1, \cdots, \widehat{\varphi^p}, \cdots, \widehat{\varphi^q}, \cdots, \varphi^k)\\
    &\quad-\sum_{q=1}^{k}\langle \nabla_\theta\mathcal{Q}_{k-1}(\varphi^1, \cdots, \widehat {\varphi^q}, \cdots, \varphi^{k}),\nabla_\theta\varphi^q\rangle\\
    &\quad-(k-3)\sum_{p\neq q}\langle\nabla_\theta\varphi^{p},\nabla_\theta\varphi^{q}\rangle\mathcal{Q}_{k-2}(\varphi^1, \cdots, \widehat{\varphi^p},\cdots, \widehat{\varphi^q}, \cdots, \varphi^k)\\
    &\quad+\sum_{l=1}^{k-3}c_{k,l}\sum_{\tau\in\mathfrak{S}_{k}}\mathcal{Q}_{l+1}(\varphi^{\tau(1)}, \cdots, \varphi^{\tau(l+1)})\mathcal{Q}_{k-l-1}(\varphi^{\tau(l+2)}, \cdots, \varphi^{\tau(k)}).
\end{split}
\end{equation}
In other words, $\mathcal{Q}_{k}(\varphi^1, \cdots, \varphi^{k})$ can be expressed as a combination of $\mathcal{Q}_l$ for $l\leq k-1$. 
Now the desired result follows by an induction.
\end{proof}

For $k=3$, \eqref{eq-expression-Qk} reduces to an identity. In fact, the first two summations in the right-hand side 
are the same, the last two summations in the right-hand side are absent, and the expression in the left-hand side is zero.

\begin{lemma}\label{lemma-express-H-k}
For $k\geq 1$, assume $\varphi^i\in\s_{d_i}$ for each $i=1, \cdots, k$. Then, 
$\mathcal H_k(\varphi^1, \cdots, \varphi^{k})\in \s_{d_1+\cdots+d_k}$. 
\end{lemma}

\begin{proof}
We prove by induction on $k$. 
By $\mathcal{H}_1(\varphi^1)=\Delta_\theta\varphi^1$, the desired result follows from Lemma \ref{lemma-gradient-spherical}(i). 
Assume the desired result holds for any collection of $l$ functions, for any $l\leq k-1$. Then by Lemma \ref{lemma-formula-trace}, Lemma \ref{lemma-Q-k} and the induction hypothesis, we have the desired result for $k$. 
\end{proof}

Similar to \eqref{eq-Q-k}, we define, for $k\ge 2$, 
\begin{align} 
\label{eq-P-1-k}
     \mathcal{P}^1_{k}(\varphi^1, \cdots, \varphi^{k})&=\sum_{\tau\in\mathfrak{S}_{k}}\nabla^2_{\theta}\varphi^{\tau(1)}\circ\cdots\circ\nabla^{2}_{\theta}\varphi^{\tau(k-2)}(\nabla_{\theta}\varphi^{\tau(k-1)},\nabla_{\theta}\varphi_t^{\tau(k)}),\\
\label{eq-P-2-k}
     \mathcal{P}^2_{k}(\varphi^1, \cdots, \varphi^{k})&=\sum_{\tau\in\mathfrak{S}_{k}}\nabla^2_{\theta}\varphi^{\tau(1)}\circ\cdots\circ\nabla^{2}_{\theta}\varphi^{\tau(k-2)}(\nabla_{\theta}\varphi_t^{\tau(k-1)},\nabla_{\theta}\varphi_t^{\tau(k)}),
\end{align}
where the sub-index $t$ in \eqref{eq-P-1-k} and \eqref{eq-P-2-k} denotes the $t$-derivative.
There is one factor of $t$-derivative in \eqref{eq-P-1-k} and two factors of $t$-derivatives in \eqref{eq-P-2-k}.
We point out that the second derivatives in $t$ are not present.

\begin{lemma}\label{lemma-P-1,2-k}
For $k\geq 2$, assume $\varphi^i\in\s_{d_i}$ for each $i=1, \cdots, k$. Then, 
$\mathcal{P}^1_k(\varphi^1, \cdots, \varphi^{k})$, $\mathcal{P}^2_k(\varphi^1, \cdots, \varphi^{k})\in \s_{d_1+\cdots+d_k}$. 
\end{lemma}

\begin{proof} The proof is based on induction and 
similar as that of Lemma \ref{lemma-Q-k}. We omit details and only record some crucial identities. 
It is clear that $\varphi^i_t\in\s_{d_i}$, for $i=1, \cdots, k$. 
 
We first consider $\mathcal{P}^1_k=\mathcal{P}^1_k(\varphi^1, \cdots, \varphi^{k})$. For $k=2$, we have 
$$\mathcal{P}^1_2(\varphi^1, \varphi^{2})=\langle\nabla_{\theta}\varphi^1,\nabla_{\theta}\varphi^2_t\rangle+\langle\nabla_{\theta}\varphi^2,\nabla_{\theta}\varphi^1_t\rangle.$$ 
Hence, $\mathcal{P}^1_2(\varphi^1, \varphi^2)\in \s_{d_1+d_2}$ by Lemma \ref{lemma-gradient-spherical}(ii). 
For $k\ge3$, we have 
\begin{align*}
(k-2)\mathcal{P}^1_k    
&=\sum_{p\neq q}\mathcal{P}^{1}_{k-1}(\langle\nabla_{\theta}\varphi^p,\nabla_{\theta}\varphi^q\rangle, \varphi^1, \cdots, \widehat{\varphi^p}, \cdots, \widehat{\varphi^q}, \cdots, \varphi^k)\\
&\qquad -\sum_{q=1}^{k}\langle \nabla_{\theta}\mathcal{P}^{1}_{k-1}(\varphi^1, \cdots, \widehat{\varphi^q}, \cdots, \varphi^k),\nabla_{\theta}\varphi^q\rangle\\
    &\qquad-(k-3)\sum_{p\neq q}\langle\nabla_{\theta}\varphi^{p},\nabla_{\theta}\varphi^{q}\rangle\mathcal{P}^1_{k-2}(\varphi^1, \cdots, \widehat{\varphi^p}, \cdots, \widehat{\varphi^q}, \cdots, \varphi^k)\\
&\qquad 
+\sum_{l=1}^{k-3}c_{k,l}\sum_{\tau\in\mathfrak{S}_{k}}\mathcal{P}_{l+1}^{1}(\varphi^{\tau(1)}, \cdots, \varphi^{\tau(l+1)})\mathcal{Q}_{k-l-1}(\varphi^{\tau(l+2)}, \cdots, \varphi^{\tau(k)}).
\end{align*}
The proof is similar as that of \eqref{eq-expression-Qk}. 
Thus by Lemma \ref{lemma-Q-k} and the induction hypothesis, we have $\mathcal{P}^1_k\in \s_{d_1+\cdots+d_k}$.

We next consider $\mathcal{P}_k^2=\mathcal{P}^2_k(\varphi^1, \cdots, \varphi^{k})$. For $k=2$, we have 
$$\mathcal{P}_2^2(\varphi^1,\varphi^2)=2\langle\nabla_{\theta}\varphi^1_t,\nabla_{\theta}\varphi^2_t\rangle.$$ 
Hence, $\mathcal{P}^1_2(\varphi^1, \varphi^2)\in \s_{d_1+d_2}$  by Lemma \ref{lemma-gradient-spherical}(ii). 
For $k\ge 3$, we have 
\begin{align*}
(k-1)\mathcal{P}^2_k 
    &=\sum_{p\neq q}\mathcal{P}^{2}_{k-1}(\langle\nabla_{\theta}\varphi^p,\nabla_{\theta}\varphi^q\rangle, \varphi^1, \cdots, \widehat{\varphi^p}, \cdots, \widehat{\varphi^q}, \cdots, \varphi^k)\\
    &\qquad -\sum_{q=1}^{k}\langle \nabla_{\theta}\mathcal{P}^{1}_{k-1}(\varphi^1, \cdots, \widehat{\varphi^q}, \cdots, \varphi^k),\nabla_{\theta}\varphi^q\rangle\\
    &\qquad-(k-3)\sum_{p\neq q}\langle\nabla_{\theta}\varphi^{p},\nabla_{\theta}\varphi^{q}\rangle\mathcal{P}^{2}_{k-2}(\varphi^1, \cdots, \widehat{\varphi^p}, \cdots, \widehat{\varphi^q}, \cdots, \varphi^k)\\
&\qquad +\sum_{l=1}^{k-3}c_{k,l}\sum_{\tau\in\mathfrak{S}_{k}}\mathcal{P}_{l+1}^{1}(\varphi^{\tau(1)}, \cdots, \varphi^{\tau(l+1)})\mathcal{P}^{1}_{k-l-1}(\varphi^{\tau(l+2)}, \cdots, \varphi^{\tau(k)}).
\end{align*}
The proof is also similar as that of \eqref{eq-expression-Qk}. 
Thus by Lemma \ref{lemma-Q-k} and the first part of this lemma, we conclude $\mathcal{P}^2_k\in \s_{d_1+\cdots+d_k}$.
\end{proof}

For a function $\varphi$ on $I\times S^{n-1}$, we define $\Lambda(\varphi)\in\Omega^1(T^*(I\times S^{n-1}))$ by 
\begin{equation*}
    \Lambda(\varphi)=g_0^{-1}\big\{A_{g_0}+\nabla^2\varphi+d\varphi\otimes d\varphi-\frac{1}{2}|\nabla\varphi|^2g_0\big\},
\end{equation*}
where $g_0=dt^2+d\theta^2$ is the standard cylinder metric. Using local normal coordinates $\{t,\theta^2,\cdots,\theta^n\}$ at $x\in I\times S^{n-1}$, we can write $\Lambda(\varphi)$ as a matrix:
\begin{align*}
\Lambda_{11}&=\xi_{tt}-\frac{1}{2}(1-\xi_t^2)+\varphi_{tt}
+\xi_t\varphi_t-\frac{1}{2}|\nabla_{\theta}\varphi|^2+\frac{1}{2}\varphi_t^2,\\
\Lambda_{1i}&=\varphi_{ti}+(\xi_t+\varphi_t)\varphi_i\quad\text{for } 2\le i\le n,\\
\Lambda_{ii}&=\frac12(1-\xi_t^2)+\varphi_{ii}-\xi_t\varphi_t+\varphi_i^2
-\frac{1}{2}|\nabla_{\theta}\varphi|^2-\frac{1}{2}\varphi_t^2\quad\text{for } 2\le i\le n,\\
\Lambda_{ij}&=\varphi_{ij}+\varphi_i\varphi_j\quad\text{for } 2\le i\neq j\le n.
\end{align*}

The main result is the following proposition.

\begin{prop}\label{prop-appendix-main}
Assume $\varphi\in \mathcal S_d$. Then for any $l\geq 0$, $\sigma_l(\Lambda(\varphi))\in \s_{2ld}$.
\end{prop}

\begin{proof}
Let $\bar\Lambda$ be the $(n-1)\times(n-1)$ matrix obtained by deleting the first row and the 
first column from the matrix $\Lambda$, and $\widetilde\Lambda$ be the $n\times n$ matrix obtained by replacing $\Lambda_{11}$ by $0$ in the matrix $\Lambda$. Then,
\begin{equation*}
    \sigma_l(\Lambda)=\Lambda_{11}\sigma_{l-1}(\bar{\Lambda})+\sigma_{l}(\widetilde{\Lambda}).
\end{equation*}
Note $\Lambda_{11}\in \s_{2d}$. 

First, we prove $\mathrm{tr}(\bar{\Lambda}^l)\in \s_{2ld}$ for any $l\geq 0$ by induction on $l$. For $l=1$, we have 
$$\mathrm{tr}(\bar{\Lambda})=\frac{n-1}2(1-\xi_t^2)+\Delta_\theta\varphi-(n-1)\xi_t\varphi_t
-\frac{n-3}{2}|\nabla_{\theta}\varphi|^2-\frac{n-1}{2}\varphi_t^2.$$
Then, $\mathrm{tr}(\bar{\Lambda})\in \s_{2d}$. 
For some $l\ge 2$, assume the assertion holds for all $j\leq l-1$. 
Set
\begin{equation*}
    \eta=\frac{1}{2}(1-\xi_t^2)-\xi_t\varphi_t-\frac{1}{2}|\nabla_{\theta}\varphi|^2-\frac{1}{2}\varphi_t^2.
\end{equation*}
Then, $\eta\in \s_{2d}$. Note 
$$\bar{\Lambda}_{ij}=\varphi_{ij}+\varphi_{i}\varphi_{j}+\eta\delta_{ij}.$$
Hence, 
\begin{equation*}
\begin{split}
    \mathrm{tr}(\bar{\Lambda}^{l})
    =\sum_{j=0}^{l}\eta^j\mathrm{tr}((\nabla^2 \varphi)^{l-j})
    +\sum_{p+j_1+\cdots+j_p=l}c_{p,j_1,\cdots,j_p}\eta^p\mathcal{Q}_{j_1}\cdots\mathcal{Q}_{j_p},
\end{split}
\end{equation*}
where each $c_{p,j_1,\cdots,j_p}$ is a nonnegative integer and 
each $\mathcal{Q}_j$ is defined by (\ref{eq-Q-k}) for $\varphi^1=\cdots=\varphi^j=\varphi$. Thus, $\mathrm{tr}(\bar\Lambda^l)\in \s_{2ld}$ by Lemma \ref{lemma-Q-k} and Lemma \ref{lemma-express-H-k}.  

Now observe that 
\begin{equation*}
    l\sigma_{l}(\bar{\Lambda})=\sum_{j=0}^{l-1}(-1)^{j}\sigma_{l-1-j}(\bar{\Lambda})\mathrm{tr}(\bar{\Lambda}^{j+1}).
\end{equation*}
This follows from (\ref{eq-identity-Newton tensor}) by taking $\Lambda=B=\bar{\Lambda}$. 
Therefore, $\sigma_{l}(\bar{\Lambda})\in \s_{2ld}$ for any $l\geq 0$, by another induction on $l$.

Next, we consider $\widetilde\Lambda$. Note, for $l\ge 0$,  
\begin{equation*}
\begin{split}
    \mathrm{tr}(\widetilde{\Lambda}^l)=\mathrm{tr}(\bar\Lambda^l)+\sum_{p+q+j_1+\cdots+j_p=l}c_{p,q,j_1,\cdots,j_p}\eta^p|\nabla_{\theta}\varphi_t+(\xi_t+\varphi_t)\nabla_{\theta}\varphi|^{2q}\mathcal{T}_{j_1}\cdots\mathcal{T}_{j_p},
\end{split}
\end{equation*}
where each $c_{p,q,j_1,\cdots,j_p}$ is a nonnegative integer and 
each $\mathcal{T}_j$ is one of  $\{\mathcal{Q}_j\}_{j=2}^{l-2}$, $\{\mathcal{P}^1_j\}_{j=2}^{l}$, and $\{\mathcal{P}^2_j\}_{j=2}^{l}$, defined by (\ref{eq-Q-k}), (\ref{eq-P-1-k}), and (\ref{eq-P-2-k}), respectively, for $\varphi^1=\cdots=\varphi^j=\varphi$.
Thus by Lemma \ref{lemma-Q-k} and Lemma \ref{lemma-P-1,2-k}, we conclude $\mathrm{tr}(\widetilde\Lambda^l)\in \s_{2ld}$ for any $l\geq 0$. As a consequence, we have $\sigma_l(\widetilde{\Lambda})\in \s_{2ld}$ for $l\ge 0$. 
\end{proof}

We note that Lemma \ref{lemma-Pk-spherical} follows from Proposition \ref{prop-appendix-main}.

To end this section, we make a final remark concerning the proof of 
Lemmas \ref{lemma-formula-trace}
-\ref{lemma-P-1,2-k}. As mentioned earlier, one objective is to prove 
$H_k(\varphi, \cdots, \varphi)\in \s_{kd}$ if $\varphi\in \s_d$, 
for the $H_k$ defined in \eqref{eq-Hk-no-permuation}. 
Due to the lack of the symmetry of $H_k(\varphi^1, \cdots, \varphi^{k})$ 
for $k\ge 4$, we first prove a more general result that
$\mathcal H_k(\varphi^1, \cdots, \varphi^{k})\in \s_{d_1+\cdots+d_k}$
if $\varphi^i\in\s_{d_i}$ for $i=1, \cdots, k$, 
and then take $\varphi^1=\cdots=\varphi^{k}=\varphi$. For $k=2, 3$, we can prove directly that 
$H_k(\varphi^1, \cdots, \varphi^{k})\in \s_{d_1+\cdots+d_k}$. 
In fact, other methods are available to prove such a result for small $k$. 
For example, for each function $\varphi_i\in\s_{d_i}$, we can 
consider the corresponding homogeneous polynomial $u_i$ by restoring an appropriate power 
of $|x|$ for each spherical harmonic in $\varphi_i$,
and then attempt to prove that $H_k(\varphi^1, \cdots, \varphi^{k})$ is the restriction of 
some homogeneous polynomial of degree $d_1+\cdots+d_k$ to the unit ball. 
Then, we can use the decomposition of homogeneous polynomials 
as in the proof of Lemma \ref{lemma-SphericalHarmonics} to get the desired result. 
This process can be carried out for small $k$ such as $k=2,3$. 
However, we encounter similar difficulties for $k\ge 4$ in the induction.


\begin{thebibliography}{DG}


\bibitem{CaffarelliGS1989} L. Caffarelli, B. Gidas, J. Spruck, 
\emph{Asymptotic symmetry and local behavior of semilinear elliptic
equations with critical Sobolev growth}, Comm. Pure Appl. Math., 42(1989), 271-297.

\bibitem{CaffarelliJSX2014} L. Caffarelli, T. Jin, Y. Sire, J. Xiong,
\emph{Local analysis of solutions of fractional semi-linear elliptic
equations with isolated singularities}, Arch. Ration. Mech. Anal., 213(2014), 245-268.

\bibitem{ChangHanYang2005} S.-Y. A. Chang, Z.-C. Han, P. Yang, 
\emph{Classification of singular radial solutions to the 
$\sigma_k$ Yamabe equation on annular domains}, J. Diff. Eq., 216(2005), 482-501. 

\bibitem{Chang-Hang-Yang2004} S.-Y. A. Chang, F. Hang, P. Yang, 
{\it On a class of locally conformally flat manifolds}, Int.
Math. Res. Not., 2004(2004), 185-209.


\bibitem{CoddingtonL1955} 
E. A. Coddington, N. Levinson, \emph{Theory of Ordinary Differential Equations}, McGraw-Hill, 
New York-Toronto-London, 1955. 

\bibitem{Gonzalez2005} M. Gonz\'{a}lez, {\it Singular sets of a class of locally conformally flat manifolds}, Duke Math.
J., 129(2005), 551-572. 

\bibitem{Guan-Lin-Wang2005} P. Guan, C.-S. Lin, G. Wang, 
{\it Schouten tensor and some topological properties},  Comm.
Anal. Geom., 13(2005), 887-902. 

\bibitem{Gursky-Viacolvsky2006} 
M. Gursky, J. Viacolvsky,  \emph{Convexity and singularities of curvature equations in conformal
geometry}, Int. Math. Res. Not. (2006), Art. ID 96890, 43 pp. 

\bibitem{Han2006} Z.-C. Han,
\emph{A Kazdan-Warner type identity for the $\sigma_k$ curvature}, C. R. Math., 347(2006), 475-478. 

\bibitem{HanLi2010} Z.-C. Han, Y.-Y. Li, E. V. Teixeira, 
\emph{Asymptotic behavior of solutions to the $k$-Yamabe equation
near isolated singularities}, Invent. Math., 182(2010),  635-684.

\bibitem{KorevaarMPS1999} N. Korevaar, R. Mazzeo, F. Pacard and R. Schoen, 
\emph{Refined asymptotics for constant scalar curvature
metrics with isolated singularities},  Invent. Math., 135(1999),  233-272.

\bibitem{Li2006} Y.-Y. Li, 
\emph{Conformally invariant fully nonlinear elliptic equations and isolated singularities}, J. Funct.
Anal., 233(2006), 380-425.




\bibitem{Marques2008} F. C. Marques, \emph{Isolated singularities of solutions to the Yamabe equation}, 
Calc. Var. \& P. D. E., 32(2008), 349-371.

\bibitem{MazzeoP1999} R. Mazzeo, F. Pacard, 
\emph{Constant scalar curvature metrics with isolated singularities}, Duke
Math. J., 99(1999), 353-418.

\bibitem{MazzeoDU1996} R. Mazzeo, D. Pollack, K. Uhlenbeck, 
\emph{Moduli spaces of singular Yamabe
metrics}, J. Amer. Math. Soc., 9(1996), 303-344.


\bibitem{Mazzieri-Ndiaye} L. Mazzieri, C. B. Ndiaye, \emph{Existence of solutions for the singular $\sigma_k$-Yamabe problem}, 
preprint.

\bibitem{Mazzieri-Segatti2012} L. Mazzieri, A. Segatti, \emph{Constant $\sigma_k$-curvature metrics with Delaunay type ends}, 
Adv. Math., 229(2012),  3147-3191. 

\bibitem{Reilly1973} R. C. Reilly, 
\emph{Variational properties of functions of the mean curvatures for hypersurfaces in space forms}, 
J. Diff. Geom., 8(1973), 465-477.

\bibitem{Schoen1988}
R. Schoen,
\emph{The existence of weak solutions with prescribed singular behavior for a conformally
invariant scalar equation}, Comm. Pure Appl. Math., 41(1988), 317-392.

\bibitem{SchoenYau1988} R. Schoen, S.-T. Yau, 
\emph{Conformally flat manifolds, Kleinian groups and scalar curvature}, Invent.
Math., 92(1988), 47-71.

\bibitem{Stein1971} E. Stein, G. Weiss, \emph{Fourier Analysis on Euclidean Spaces}, 
Princeton University Press, 1971. 

\bibitem{Viaclovsky2000} J. Viaclovsky,
\emph{Some fully nonlinear equations in conformal geometry}, In: Differential Equations and Mathematical Physics, Birmingham, AL, 1999, AMS/IP Stud. Adv. Math., vol. 16, 425-433, Amer. Math. Soc., Providence, 2000.

\bibitem{Wang2013} Y. Wang, \emph{Asymptotic behaviors of solutions to the conformal quotient equation}, 
Ph. D. Thesis, Rutgers University, 2013. 

\end{thebibliography}
\end{document}